\documentclass[11pt,reqno,a4paper]{amsart}
\usepackage[british]{babel}

\usepackage{comment}
\usepackage{multicol}
\usepackage{graphicx}
\usepackage{amscd}
\usepackage{amsmath}
\usepackage{amsfonts}
\usepackage{amssymb}
\usepackage{comment}
\usepackage{color}
\usepackage{pdfsync}
\usepackage{bbm}
\usepackage{hyperref}
\usepackage{enumerate}
\usepackage{mathabx}

\definecolor{trp}{rgb}{1,1,1}

\definecolor{red}{rgb}{1,0,.2}

\newtheorem{theorem}{Theorem}[section]
\theoremstyle{plain}

\newtheorem{claim}{Claim}

\newtheorem{lemma}[theorem]{Lemma}

\newtheorem{prop}[theorem]{Proposition}

\numberwithin{equation}{section}

\newcommand{\R}{\mathbb{R}}
\newcommand{\N}{\mathbb{N}}

\newcommand{\xv}{\underline{x}}
\newcommand{\pv}{\underline{p}}
\newcommand{\pvv}{\mathbf{p}}
\newcommand{\yv}{\underline{y}}

\newcommand{\ii}{\mathbf{i}}

\newcommand{\jj}{\mathbf{j}}

\newcommand{\supp}{\mathrm{supp}}

\newcommand{\iv}{\overline{\imath}}
\newcommand{\jv}{\overline{\jmath}}

\newcommand{\ALPHA}{\alpha}
\begin{document}
\title[Shrinking targets on Bedford-McMullen carpets]{Shrinking targets on Bedford-McMullen carpets}

\author{Bal\'azs B\'ar\'any}
\address[Bal\'azs B\'ar\'any]{Budapest University of Technology and Economics, MTA-BME Stochastics Research Group, P.O.Box 91, 1521 Budapest, Hungary}
\address[Bal\'azs B\'ar\'any]{Einstein Institute of Mathematics, Hebrew Univeristy of Jerusalem, Edmund J. Safra Campus, Givat Ram, Jerusalem 9190401, Izrael}
\email{balubsheep@gmail.com}

\author{Micha\l\ Rams}
\address[Micha\l\ Rams]{Institute of Mathematics, Polish Academy of Sciences, ul. \'Sniadeckich 8, 00-656 Warszawa, Poland}
\email{rams@impan.pl}

\thanks{Bal\'azs B\'ar\'any acknowledges support from grant OTKA K104745, ERC grant 306494 and the J\'anos Bolyai Research Scholarship of the Hungarian Academy of Sciences. Micha\l\ Rams was supported by National Science Centre grant
2014/13/B/ST1/01033 (Poland).}

\begin{abstract}
We describe the shrinking target set for the Bedford-McMullen carpets, with targets being either cylinders or geometric balls.
\end{abstract}
\date{\today}

\maketitle

\thispagestyle{empty}

\section{Introduction and Statements}

\subsection{Introduction} The shrinking target problem is a general name for a class of problems, first investigated by Hill and Velani in \cite{HillVelanifirst}. This class of problems is related to other distribution type questions (mass escape, return time distribution), it also appears in the number theory (Diophantine approximations, as an example the classical Jarnik-Besicovitch theorem can be interpreted as a special shrinking target problem for irrational rotations).

The setting is as follows. Given a dynamical system $(X,T)$ and a sequence of sets $B_i\subset X$, we define
\[
\Gamma = \{x\in X; T^ix\in B_i\ {\rm i.o.}\}
\]
and ask how large is the set $\Gamma$.
The size of the shrinking target set we will describe by calculating its Hausdorff dimension.
For any Borel set $A$, let
$$
\mathcal{H}_{\delta}^s(A)=\inf\{\sum_i|U_i|^s:A\subseteq\bigcup_iU_i\ \&\ |U_i|<\delta\}.
$$
Then the $s$-dimensional Hausdorff measure of $A$ is $\mathcal{H}^s(A)=\lim_{\delta\to0+}\mathcal{H}_{\delta}^s(A)$. We denote the Hausdorff dimension of $A$ by
$$
\dim_HA=\inf\{s>0:\mathcal{H}^s(A)=0\}.
$$
For further details, see Falconer \cite{falconerfractalgeom}.

The answer to the shrinking target problem will, naturally, depend on the sets $B_i$, one usually chooses some especially interesting (in a given setting) class of those sets. After the original paper of Hill and Velani \cite{HillVelanifirst}, this question was asked in many different contexts, let us just mention expanding maps of the interval considered by Fan, Schmeling and Troubetzkoy~\cite{FanSchmellingTroubetzkoy}, Li, Wang, Wu and Xu~\cite{LiWangWuXu}, \cite{Reeve}, Liao and Seuret~\cite{LiaoSeuret}, Persson and Rams~\cite{PerssonRams} and irrational rotations studied by Schmeling and Troubetzkoy~\cite{SchmellingTroubetzkoy}, Bougeaud~\cite{Bugeaud}, Fan and Wu~\cite{FanWu}, Xu~\cite{Xu}, Liao and Rams~\cite{LiaoRams} and Kim, Rams and Wang~\cite{KimRamsWang}.

All the examples above are in dimension 1. In higher dimensions there appears a significant technical problem: the maps are not necessarily conformal. The dimension theory for nonconformal dynamical systems is lately very rapidly developing, but we will be mostly interested in the subclass: the dimension theory of affine iterated function systems. In this class there are several examples of systems for which we can exactly calculate the Hausdorff dimension of the attractor, the simplest of them (and the one we will investigate in this paper) are the Bedford-McMullen carpets, \cite{bedford1984crinkly}, \cite{McMullen}. We should also mention examples considered by Lalley and Gatzouras \cite{GatzourasLalley}, Kenyon and Peres \cite{kenyon1996measures}, Bara\'nski \cite{baranski2007hausdorff}, Hueter and Lalley \cite{hueter1995falconer}, B\'ar\'any \cite{barany2014ledrappier}, as well as the generic results of Falconer \cite{falconer1988hausdorff}, Solomyak \cite{solomyak1998measure}, B\'ar\'any, K\"aenm\"aki, Koivusalo \cite{barany2016dimension} and B\'ar\'any, Rams and Simon~\cite{baranyrams2016dimension,barany2017dimension}.

The Bedford-McMullen carpets are defined as follows.
Let $M>N\geq2$ integer numbers and let $\tau=\frac{\log M}{\log N}$. Moreover, let $S$ be a non-empty subset of $\{1,\dots,N\}$ and for every $a\in S$ let $P_a$ be a non-empty subset of $\{1,\dots,M\}$. Denote
$$
Q=\{(a,b):a\in S\ b\in P_a\}.
$$
Let us denote the number of the elements in the sets $S, P_a, Q$ by $R, T_a, D$ respectively. For every $(a,b)\in Q$ let
$$
F_{(a,b)}(x,y)=\left(\frac{x+a-1}{N},\frac{y+b-1}{M}\right).
$$
By Hutchinson's Theorem \cite{Hutchinson}, there exists a unique non-empty compact set $\Lambda$ such that
\begin{equation}\label{eq:defcarpet}
\Lambda=\bigcup_{(a,b)\in Q}F_{(a,b)}(\Lambda).
\end{equation}
This construction gives us an iterated function system, to obtain a dynamical system (repeller of which will be $\Lambda$) we need to take the inverse maps $F_{(a,b)}^{-1} : F_{(a,b)}([0,1]^2) \to [0,1]^2$.

Let us now go back to the shrinking target problem. There are two important results for the shrinking target problem in a higherdimensional nonconformal settting. Koivusalo and Ramirez \cite{koivusalo2014recurrence} investigated a general-type self-affine iterated function systems (Falconer and Solomyak's setting) and obtained an almost-sure type result using the Falconer's singular value pressure function, the shrinking targets in this paper are cylinder sets. Hill and Velani \cite{HillVelanitori} investigated a special type of Bedford-McMullen carpet (with $Q=\{0,\ldots, N-1\} \times \{0,\ldots, M-1\}$), their method might be also applicable to a more general class of product-like Bedford-McMullen carpets (where for each choice of $a$ the number of possible choices of $b, (a,b)\in D$ is the same).

In this paper we will present an answer to the shrinking target problem valid for all Bedford-McMullen carpets, with the shrinking target chosen as either cylinders or geometric balls.

\subsection{Main theorem} Now, we state the main theorem of this paper.  Let
\begin{equation}\label{eq:transform}
T(x,y)=(Nx\mod1,My\mod1)
\end{equation}
be a uniformly expanding map on the unit square. Then it is easy to see that $\Lambda$ (defined in \eqref{eq:defcarpet} ) is $T$ invariant. Let $\mathcal{P}$ be natural partition, i.e. $\mathcal{P}(x,y)=\left[\frac{\lfloor Nx\rfloor}{N},\frac{\lfloor Nx\rfloor+1}{N}\right]\times\left[\frac{\lfloor My\rfloor}{M},\frac{\lfloor My\rfloor+1}{M}\right]$. Moreover, let \begin{equation}\label{eq:cyl1}
\mathcal{P}_n(x,y)=\bigcap_{k=0}^{n-1} T^{-k}\mathcal{P}(T^k(x,y))
\end{equation}

Denote the simplex of probability vectors $\pv=(p_{a,b})_{(a,b)\in Q}$ by $\Upsilon$, i.e.
$$
\Upsilon=\left\{\pv\in\R^Q:\ p_{a,b}\geq\ \&\ \sum_{(a,b)\in Q}p_{a,b}=1\right\}.
$$
Define for a probability vector $\pv=(p_{a,b})_{(a,b\in Q)}$ the Bernoulli measure $\nu_{\pv}=\{\pv\}^{\N}$, define also its $$h(\nu_{\pv})=h(\pv)=-\sum_{(a,b)\in Q}p_{a,b}\log p_{a,b}$$ entropy and $$h_r(\nu_{\pv})=h_r(\pv)=-\sum_{a\in S}(\sum_{b\in P_a}p_{a,b})\log(\sum_{b\in P_a}p_{a,b})$$ row-entropy. Let us define the dimension of $\pv$ as follows,
$$
\dim(\pv)=\frac{h(\pv)+(\tau-1)h_r(\pv)}{\log M}.
$$
Moreover, let
\begin{eqnarray*}
m(\alpha,\tau)&=&\min\{1, (1+\alpha)/\tau\}\\
M(\alpha,\tau)&=&\max\{0, 1- (1+\alpha)/\tau\}=1-m(\alpha,\tau).
\end{eqnarray*}

Now, we introduce the 6 different quantities, depending on probability vectors. We call these functions \textit{dimension functions}.
\[
d_1(\pv_-) = \dim \pv_-,
\]
\[
d_2^{\alpha}(\pv_-, \pv_1) = \frac {m(\alpha,\tau) h(\pv_-) + M(\alpha,\tau) h_r(\pv_1) } {(1+\alpha) \log N}  ,
\]
\[
d_3^{\alpha}(\pv_-, \pv_1, \pv_2) = \frac {m(\alpha,\tau) h(\pv_-) + M(\alpha,\tau) h(\pv_1) + (\tau-1) h_r(\pv_2)} {(\tau + \alpha) \log N},
\]
\[
d_4^{\alpha}(\pv_-, \pv_1, \pv_2, \pv_+, H) = \frac {m(\alpha,\tau) h(\pv_-) + M(\alpha,\tau) h(\pv_1) + (\tau-1) h_r(\pv_2) + \alpha(\tau -1) h_r(\pv_+) + \alpha (1-1/\tau) H} {\tau (1+\alpha) \log N},
\]
\[
d_5^{\alpha}(\pv_-, \pv_1, \pv_2, \pv_+, H) = \frac {m(\alpha,\tau) h(\pv_-) + M(\alpha,\tau) h(\pv_1) + (\tau-1) h(\pv_2) + (\tau +\alpha)(\tau -1) h_r(\pv_+) + \alpha (1-1/\tau) H} {\tau (\tau+\alpha) \log N},
\]
\[
d_6(\pv_+) = \dim \pv_+,
\]
and let
$$
\DJ_{\ALPHA}(\pv_-,\pv_1,\pv_2,\pv_+,h)=\min\left\{d_1(\pv_-),d_2^{\alpha}(\pv_-,\pv_1),d_3^{\alpha}(\pv_-,\pv_1,\pv_2),d_4^{\alpha}(\pv_-,\pv_1,\pv_2,\pv_+,h),d_5^{\alpha}(\pv_-,\pv_1,\pv_2,\pv_+,h),d_6(\pv_+)\right\}.
$$

\begin{theorem}\label{thm:maincyl1}
Let $\xv_n$ be an arbitrary sequence of points on $\Lambda$ and let $f:\N\mapsto\N$ be an arbitrary function such that $\lim_{n\to\infty}\frac{f(n)}{n}=\alpha>0$. Then
$$
\dim_H\{\yv\in\Lambda:T^n\yv\in\mathcal{P}_{f(n)}(x_n)\text{ infinitely often}\}=\max_{\pv_-,\pv_1,\pv_2,\pv_+\in\Upsilon^4}\DJ_{\ALPHA}(\pv_-,\pv_1,\pv_2,\pv_+,0).
$$
\end{theorem}

Let us denote the geometric balls on $\R^2$ centered at $\xv$ and with radius $r$ by $B(\xv,r)$.

\begin{theorem}\label{thm:mainball1}
Let $\mu$ be an ergodic, $T$-invariant measure such that $\supp\mu=\Lambda$. Let $\xv_n$ be a sequence of identically distributed random variables (not necessarily independent) with distribution $\mu$ and let $r:\N\mapsto\R^+$ be an arbitrary function such that $\lim_{n\to\infty}\frac{-\log r(n)}{n\log N}=\alpha>0$. Then
$$
\dim_H\{\yv\in\Lambda:T^n\yv\in B(\xv_n,r(n))\text{ infinitely often}\}=\max_{\pv_-,\pv_1,\pv_2,\pv_+\in\Upsilon^4}\DJ_{\ALPHA}(\pv_-,\pv_1,\pv_2,\pv_+,H),
$$
where $H=\int\log T_{\lfloor Nx\rfloor+1}d\mu(x,y)$.
\end{theorem}

For the heuristic explanation of the result we refer the reader to Section 3.

%%%%%%%%%%%%%%%%%%%%%%%%%%%%%%%%%%%%%%

\section{Symbolic dynamics}

Let us denote by $\Sigma$ the space of all infinite length words formed of symbols in $Q$, i.e. $\Sigma=Q^{\N}$. Let $\Sigma^*$ denote the set of finite length words, i.e $\Sigma^*=\bigcup_{n=0}^{\infty}Q^n$. Usually, we denote the elements of $\Sigma$ by $\ii,\jj$ and we denote the elements of $\Sigma^*$ by $\iv,\jv,\hbar$. For an $\ii=((a_1,b_1),(a_2,b_2),\dots)\in\Sigma$ and $n\geq m\geq1$ integer, let $\ii|_m^n=((a_m,b_m),\dots,(a_n,b_n))$. For $\iv\in\Sigma^*$ and $\ii\in\Sigma$ (or $\jv\in\Sigma^*$), let $\iv\ii$ (or $\iv\jv$) be the concatenation of the words. We use the convention throughout the paper that for a nonnegative number $p\notin\N$, $i_p:=i_{\lfloor p\rfloor}$, where $\lfloor.\rfloor$ denotes the lower integer part.

For $\iv=((a_1,b_1),\dots,(a_n,b_n))\in\Sigma^*$, let
$$
C(\iv)=\left\{\jj=((a_1',b_1'),(a_2',b_2'),\dots):a_k'=a_k\text{ and }b_k'=b_k\text{ for }k=1,\dots,n\right\}.
$$
Moreover, denote $B(\iv)$ the approximate square that contains $\iv=((a_1,b_1),\dots,(a_n,b_n))$ and has size $N^{-|\iv|}$. That is,
$$
B(\iv)=\bigcup_{b_{\lfloor\frac{n}{\tau}\rfloor+1}'\in P_{a_{\lfloor\frac{n}{\tau}\rfloor+1}}}\cdots\bigcup_{b_{n}'\in P_{a_{n}}}C(\iv|_1^{n/\tau}((a_{\lfloor\frac{n}{\tau}\rfloor+1},b_{\lfloor\frac{n}{\tau}\rfloor+1}'),\dots,(a_{n},b_{n}'))).
$$

Denote the cylinder of length $n$ containing $\ii=((a_1,b_1),(a_2,b_2),\dots)\in\Sigma$ by $C_n(\ii)$, i.e.
$$
C_n(\ii)=C(\ii|_1^n)\text{ and }B_n(\ii)=B(\ii|_1^n).
$$

For any $\iv=((a_1,b_1),\dots,(a_n,b_n))\in\Sigma^*$, let
$$
F_{\iv}=F_{(a_1,b_1)}\circ\cdots\circ F_{(a_n,b_n)}.
$$
Let us define the natural projection from $\Sigma$ to $\Lambda$ by $\pi$. That is,
$$
\pi(\ii)=\lim_{n\to\infty}F_{\ii|_1^n}(0,0).
$$

Let us denote the left-shift operator on $\Sigma$ by $\sigma$. It is easy to see that $\sigma$ and $T$ are conjugated on $\Lambda$, i.e.
\begin{equation}\label{eq:conj}
\pi\circ\sigma=T\circ\pi.
\end{equation}

Let $f:\N\mapsto\N$ be a function such that
\begin{equation}\label{eq:lingrov}
  \lim_{n\to\infty}\frac{f(n)}{n}=\alpha>0.
\end{equation}
Let $\{\jj_n\}$ be a sequence in $\Sigma$ and let $\Gamma_C(f,\{\jj_n\}),\Gamma_B(f,\{\jj_n\})$ be the set of points which hit the shrinking targets $\{C_{f(n)}(\jj_n)\}$ and $\{B_{f(n)}(\jj_n)\}$ infinitely often. That is,
$$
\Gamma_C(f,\{\jj_n\})=\{\ii\in\Sigma:\sigma^n\ii\in C_{f(n)}(\jj_n)\text{ i.o.}\}\text{ and }
$$
$$
\Gamma_B(f,\{\jj_n\})=\{\ii\in\Sigma:\sigma^n\ii\in B_{f(n)}(\jj_n)\text{ i.o.}\}.
$$
In other words,
\begin{equation}\label{eq:shrinksymb}
\Gamma_C(f,\{\jj_n\})=\bigcap_{K=1}^{\infty}\bigcup_{k=K}^{\infty}\sigma^{-k}(C_{f(k)}(\jj_k))\text{ and }\Gamma_B(f,\{\jj_n\})=\bigcap_{K=1}^{\infty}\bigcup_{k=K}^{\infty}\sigma^{-k}(B_{f(k)}(\jj_k)).
\end{equation}
For a visualisation of $\sigma^{-n}C_{f(n)}(\jj_n)$, $\sigma^{-n}B_{f(n)}(\jj_n)$, see Figure~\ref{fig:holes}.
\begin{figure}
  \centering
  \includegraphics[width=170mm]{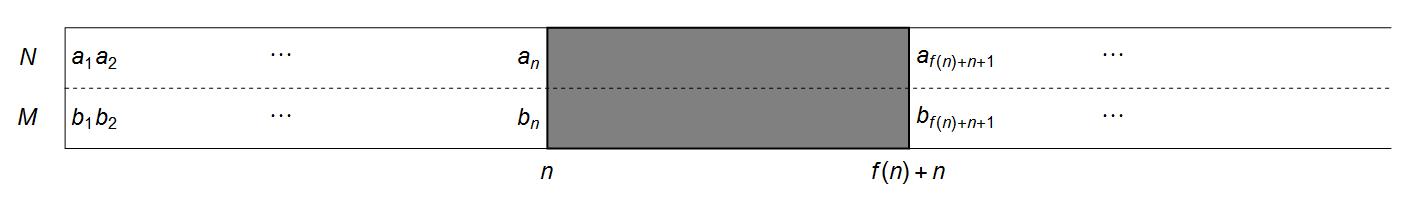}
  \includegraphics[width=170mm]{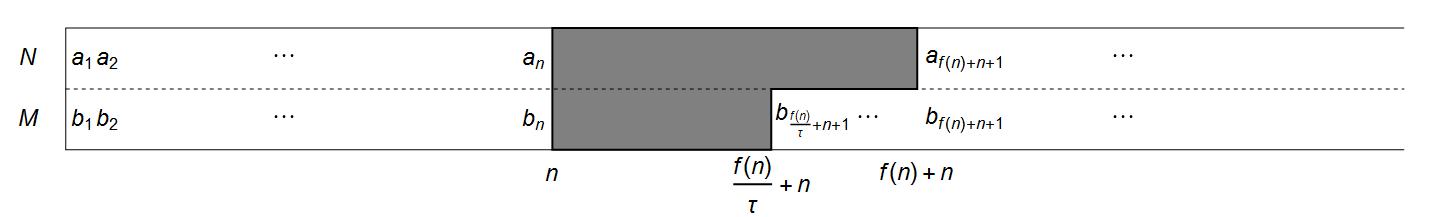}
  \caption{Symbolic representation of holes at $n$th iterations defined by cylinders and balls. That is, the sets $\sigma^{-n}C_{f(n)}(\jj_n)$ and $\sigma^{-n}B_{f(n)}(\jj_n)$.}\label{fig:holes}
\end{figure}

\begin{theorem}\label{thm:maincyl2}
Let $\{\jj_n\}$ be an arbitrary sequence in $\Sigma$, and let $f:\N\mapsto\N$ be a function such that $\lim_{n\to\infty}\frac{f(n)}{n}=\alpha>0$. Then
\begin{equation}\label{eq:thmcyl}
\dim_H\pi\Gamma_C(f,\{\jj_n\})=\max_{\pv_-,\pv_1,\pv_2,\pv_+\in\Upsilon^4}\DJ_{\ALPHA}(\pv_-,\pv_1,\pv_2,\pv_+,0).
\end{equation}
\end{theorem}

\begin{theorem}\label{thm:mainball2}
Let $f:\N\mapsto\N$ be a function such that $\lim_{n\to\infty}\frac{f(n)}{n}=\alpha>0$ and let $\{\jj_n=((a_1^{(n)},b_1^{(n)}),(a_2^{(n)},b_2^{(n)}),\dots)\}$ be a sequence in $\Sigma$ such that $\lim_{n\to\infty}\frac{1}{f(n)(1-1/\tau)}\sum_{k=f(n)/\tau}^{f(n)}\log T_{a_k^{(n)}}=H$. Then
\begin{equation}\label{eq:thmball}
\dim_H\pi\Gamma_B(f,\jj_k)=\max_{\pv_-,\pv_1,\pv_2,\pv_+\in\Upsilon^4}\DJ_{\ALPHA}(\pv_-,\pv_1,\pv_2,\pv_+,H).
\end{equation}
\end{theorem}

\section{Heuristics}

The statements of our results might on the first glance look a bit strange and complicated, but they have a simple geometric meaning.

Fix some $n>0$ and consider the set $\Gamma_n$ of points that hit the target at time $n$. In symbolic description, $\Gamma_n$ consists of points that have prescribed symbols on positions $a_i, i=n+1,\ldots, n+f(n)$ and $b_j, j=n+1,\ldots, n+\tau^{-1}f(n)$ (in the ball case) or on positions $a_i, i=n+1,\ldots, n+f(n)$ and $b_j, j=n+1,\ldots, n+f(n)$ (in the cylinder case). Let $p_-, p_1, p_2, p_+$ be some fixed probabilistic vectors on $D$.

We will define a probabilistic measure $\mu_n=\mu_n(p_-, p_1, p_2, p_+)$ the following way. We will demand that $(a_i, b_i)$ is independent from $(a_j, b_j)$ for all $i\neq j$, and that the distribution of $(a_i, b_i)$ is given by $p_-$ for $i\leq \min(n, \tau^{-1}(n+f(n)))$, by $p_1$ for $\tau^{-1}(n+f(n)) < i \leq n$, by $p_2$ for $n+f(n) < i \leq \tau n+f(n)$, and by $p_+$ for $i>\tau n+f(n)$. If we are in the ball case then we still have to describe $b_i$ for $n+\tau^{-1}f(n) < i \leq n+f(n)$, at those position we have already prescribed the value of $a_i$ and we distribute $b_i$ choosing each of available values of $b_i\in P_{a_i}$ with the uniform probability $1/T_{a_i}$.

Given $m$, we define the local dimension of $\mu_n$ at a point $\jj$ at a scale $N^{-m}$ by

\[
d_m(\mu_n, \jj) = \frac {\log \mu_n(B_m(\jj))} {-m \log N}.
\]

Then, everywhere except at a $\mu_n$-small set of points, we will have that the local dimensions of $\mu_n$ at scales $N^{-m_i}$ corresponds to the dimensions values $d_i$ as follows:
$$
\begin{array}{cccccc}
  m_1\ll n & m_2=n+f(n) & m_3= \tau n + f(n) & m_4=\tau (n+f(n)) & m_5=\tau(\tau n + f(n)) & m_6\gg n+f(n) \\
  d_1(p_-) & d_2^{\alpha}(p_-, p_1) & d_3^{\alpha}(p_-, p_1, p_2) & d_4^{\alpha}(p_-, p_1, p_2, p_+) & d_5^{\alpha}(p_-, p_1, p_2, p_+) & d_6(p_+)
\end{array}
$$
\begin{figure}
  \centering
  \includegraphics[width=170mm]{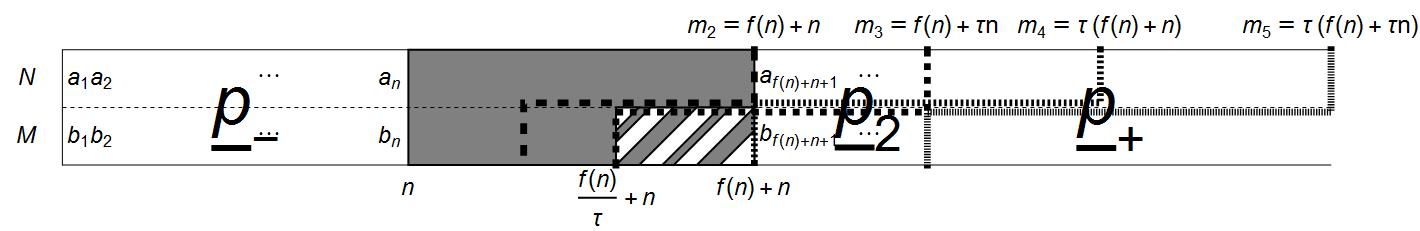}
  \caption{The representations of balls at scale $m_i$ and the probability measure $\mu_n$ in the case $\alpha\geq\tau-1$.}\label{fig:balllarge}\vspace{0.2cm}  \includegraphics[width=170mm]{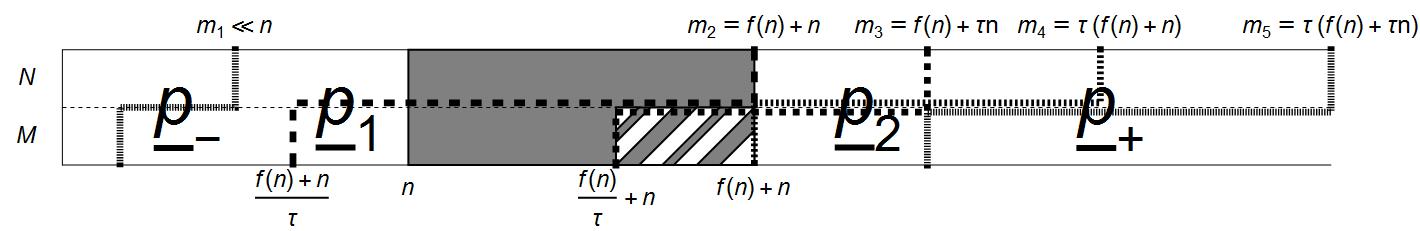}
  \caption{The representations of balls at scale $m_i$ and the probability measure $\mu_n$ in the case $\alpha<\tau-1$.}\label{fig:ballsmall}
\end{figure}
For a visual representation of the approximate squares at level $m_i$, see Figure~\ref{fig:balllarge} and Figure~\ref{fig:ballsmall}. Moreover, one can check that, again everywhere except at a $\mu_n$-small set of points, the local minima of $d_m(\mu_n, \jj)$ happen at (some of) the scales $m_1,\ldots, m_6$. That is, for $m_i<m<m_{i+1}$ we have

\[
d_m(\mu_n, \jj) > \min(d_{m_i}(\mu_n, \jj), d_{m_{i+1}}(\mu_n, \jj)) - \varepsilon(n),
\]

where $\varepsilon(n)$ is small for $n$ large.

The rest of the paper is divided as follows. In the following section we will define a class of measures ({\it piecewise Bernoulli measures}) which are a generalization of the measure $\mu_n$ defined above, and we will check some basic properties of such measures. In particular, all the statements presented in this section without proofs will follow from results of Section 4. We will also present the basic properties of entropy and row entropy, which will let us significantly simplify the statements of our results. In Section~\ref{sec:lb}, we will prove the lower bound estimation for the dimension of the shrinking target set. Idea of proof: for a proper choice of $\{n_i\}$, we will construct a piecewise Bernoulli measure supported on the set of points that hit the targets at times $\{n_i\}$ in such a way that around each scale $N^{-n_i}$ this measure will be similar to $\mu_{n_i}$. In Section~\ref{sec:ub}, we will prove the upper bound estimation. Idea of proof: for any $n$ we will construct a cover for all the points hitting the target at time $n$. Again, the construction of this cover will be closely related to the measure $\mu_n$, though this relation might be difficult to explain right now. We finish the paper with the examples section.

\section{Entropy and piecewise Bernoulli measures}

\subsection{Piecewise Bernoulli measures} We define \textit{the piecewise Bernoulli measures} on $\Sigma$ as follows. Let $\pv^{(k)}$ be an arbitrary sequence in $\Upsilon$ and let $m_k$ be a sequence of positive integers. Let us denote $\sum_{q=1}^km_q$ by $S_k$ with the concept $S_0=0$. Then we call $\mu$ piecewise Bernoulli if
$$
\mu=\prod_{\ell}^{\infty}\eta_{\ell},\text{ where }\eta_{\ell}=\pv^{(k)}\text{ if }\ell\in\left(S_{k-1},S_k\right].
$$

\begin{lemma}\label{cor:weneed}
Let $\pv^{(k)}$ be a sequence of probability vectors in $\Upsilon$ and let $m_k$ be a sequence of integers such that
\begin{equation}\label{eq:piecewisecond}
\sum_{k=1}^{\infty}m_k^{-1}<\infty.
\end{equation}
Let $\mu$ be the piecewise Bernoulli measure corresponding to the sequences $\pv^{(k)}$ and $m_k$. Then there exists a set $\Omega$ with $\mu(\Omega)=1$ such that for every sufficiently small $\varepsilon>0$ and for every $\ii=(i_1,i_2,\dots)\in\Omega$ there exist $K=K(\varepsilon,\ii)$ such that for every $k\geq K$ and every $\varepsilon m_{k-1}<m\leq m_{k}$
$$
\left|\frac{\sharp\{\ell=1+\sum_{q=1}^{k-1}m_q,\dots,m+\sum_{q=1}^{k-1}m_q:i_{\ell}=j\}}{m}-p^{(k)}_j\right|<\varepsilon\text{ and }
$$
$$
\left|\frac{\sharp\{\ell=\sum_{q=1}^{k}m_q-m+1,\dots,\sum_{q=1}^{k}m_q:i_{\ell}=j\}}{m}-p^{(k)}_j\right|<\varepsilon.
$$

\end{lemma}

\begin{proof}
We prove only the first inequality, the proof of the second one is similar.

Let us recall here Chebyshev's inequality. Let $X_1,X_2,\dots$ be independent, uniformly bounded random variables. Then for every $\varepsilon>0$
\begin{equation}\label{eq:Chebisev}
\mathbb{P}\left(\left|\sum_{i=1}^nX_i-\sum_{i=1}^n\mathbb{E}(X_i)\right|>n\varepsilon\right)\leq\frac{C}{n^2}\text{  for all $n\geq1$},
\end{equation}
where $C$ is some constant depending on the uniform bound and $\varepsilon>0$.

Let us fix $\varepsilon>0$. Let $\Delta_{m,k}$ be the set of $\ii\in\Sigma$ such that
$$
\left|\frac{\sharp\{\ell=1+\sum_{q=1}^{k-1}m_q,\dots,m+\sum_{q=1}^{k-1}m_q:i_{\ell}=j\}}{m}-p^{(k)}_j\right|>\varepsilon.
$$
Then by \eqref{eq:Chebisev},
$$
\mu(\Delta_{m,k})<\frac{C}{m^2}.
$$
Hence,
$$\mu\left(\bigcup_{m=\varepsilon m_{k-1}}^{m_k}\Delta_{m,k}\right)\leq\sum_{m=\varepsilon m_{k-1}}^{m_k}\frac{C}{m^2}\leq\frac{C(1-\varepsilon)}{\varepsilon m_{k-1}}.$$
Since the series $\sum_{k=1}^{\infty}m_k^{-1}$ is summable, by Borel-Cantelli Lemma the assertion follows.
\end{proof}

Let $\ii\in\Sigma$ and $q\geq1$. Let $k,\ell$ be integers such that $S_{k-1}\leq q< S_k$ and $\tau S_{\ell-1}\leq q<\tau S_{\ell}$, then
\begin{equation}\label{eq:meas}
\mu(B_q(\ii))= \prod_{n=1}^{\ell-1}\prod_{r=S_{n-1}}^{S_n}p_{a_r,b_r}^{(n)}\cdot\prod_{r=S_{\ell-1}}^{q/\tau}p_{a_r,b_r}^{(\ell)}\cdot\prod_{r=q/\tau}^{S_{\ell}}p_{a_r}^{(\ell)}\cdot\prod_{n=\ell+1}^{k-1}\prod_{r=S_{n-1}}^{S_n}p_{a_r}^{(n)}\cdot\prod_{r=S_{k-1}}^{q}p_{a_r}^{(k)}.
\end{equation}

\begin{lemma}\label{lem:subseq}
Let $C\subset\Upsilon^o$ be compact set and let $\{\pv^{(k)}\}_{k=1}^{\infty}$ be a sequence of prob. vectors such that $\pv^{(k)}\in C$ for all $k\geq1$. Moreover, $\{m_k\}_{k=1}^{\infty}$ be a sequence of integers such that \eqref{eq:piecewisecond} hold. If $\mu$ is the piecewise Bernoulli measure corresponding to $m_k$ and $\pv^{(k)}$ then for $\mu$-a.e. $\ii$
	\begin{multline}\label{eq:justthis}
	\liminf_{k\to\infty}\frac{\log\mu(B_{S_k}(\ii))}{-S_k\log N}=\\
\liminf_{k\to\infty}\frac{\sum_{n=1}^{\ell(k)-1}m_nh(\pv^{(n)})+(S_{k}/\tau-S_{\ell(k)-1})h(\pv^{(\ell)})+(S_{\ell(k)}-S_{k}/\tau)h_r(\pv^{(\ell)})+\sum_{n=\ell+1}^{k}m_nh_r(\pv^{(n)})}{S_{k}\log N},
	\end{multline}
	where $\ell(k)$ is the unique integer such that $\tau S_{\ell(k)-1}\leq S_k<\tau S_{\ell(k)}$, and
	\begin{equation}\label{eq:other}
	\liminf_{\ell\to\infty}\frac{\log\mu(B_{\tau S_{\ell}}(\ii))}{-\tau S_{\ell}\log N}=\liminf_{\ell\to\infty}\frac{\sum_{n=K}^{\ell}m_nh(\pv^{(n)})+\sum_{n=\ell+1}^{k(\ell)-1}m_nh_r(\pv^{(n)})+(\tau S_{\ell}-S_{k(\ell)-1})h_r(\pv^{(k)})}{\tau S_{\ell}\log N},
	\end{equation}
	where $k(\ell)$ is the unique integer such that $S_{k(\ell)-1}\leq \tau S_{\ell}<S_{k(\ell)}$.
\end{lemma}

\begin{proof}
We prove only equation \eqref{eq:justthis}, the proof of \eqref{eq:other} is similar and even simpler.

By \eqref{eq:meas},
$$
\mu(B_{S_k}(\ii))= \prod_{n=1}^{\ell(k)-1}\prod_{r=S_{n-1}}^{S_n}p_{a_r,b_r}^{(n)}\cdot\prod_{r=S_{\ell(k)-1}}^{S_k/\tau}p_{a_r,b_r}^{(\ell)}\cdot\prod_{r=S_k/\tau}^{S_{\ell(k)}}p_{a_r}^{(\ell)}\cdot\prod_{n=\ell(k)+1}^{k}\prod_{r=S_{n-1}}^{S_n}p_{a_r}^{(n)}.
$$
Let $\varepsilon>0$ be arbitrary, but fixed. Let $K=K(\varepsilon,\ii)>0$ be the constant defined in Lemma~\ref{cor:weneed} and let us assume that $\ell(k)>K+1$. Thus,
\begin{multline}
\frac{\log\mu(B_{S_k}(\ii))}{-S_k\log N}\geq\\ \frac{\sum_{n=K}^{\ell(k)-1}m_nh(\pv^{(n)})+\sum_{n=\ell(k)+1}^{k}m_nh_r(\pv^{(n)})-\sum_{r=S_{\ell(k)-1}}^{S_k/\tau}\log p_{a_r,b_r}^{(\ell)}-\sum_{r=S_k/\tau}^{S_{\ell(k)}}\log p_{a_r}^{(\ell)}}{S_k\log N}-C'\varepsilon
\end{multline}
and
\begin{multline}
\frac{\log\mu(B_{S_k}(\ii))}{-S_k\log N}\leq\\ \frac{\sum_{n=K}^{\ell(k)-1}m_nh(\pv^{(n)})+\sum_{n=\ell(k)+1}^{k}m_nh_r(\pv^{(n)})-\sum_{r=S_{\ell(k)-1}}^{S_k/\tau}\log p_{a_r,b_r}^{(\ell)}-\sum_{r=S_k/\tau}^{S_{\ell(k)}}\log p_{a_r}^{(\ell)}}{S_k\log N}+\\+\frac{S_K\hat{P}}{S_k\log N}+C'\varepsilon,
\end{multline}
with some constant $C'>0$ and $\hat{P}=-\min_{\pv\in C}\min_{a,b}\log p_{a,b}$. There are three possible cases,
\begin{itemize}
  \item if $\tau S_{\ell(k)-1}\leq S_k\leq\tau S_{\ell(k)-1}+\varepsilon\tau m_{\ell(k)-1}$ then by Lemma~\ref{cor:weneed}
\begin{multline*}
-\sum_{r=S_{\ell(k)-1}}^{S_k/\tau}\log p_{a_r,b_r}^{(\ell)}-\sum_{r=S_k/\tau}^{S_{\ell(k)}}\log p_{a_r}^{(\ell)}\leq \varepsilon m_{\ell(k)-1}\hat{P}+(S_{\ell(k)}-S_k/\tau)h_r(\pv^{(\ell)})+m_{\ell}\varepsilon\leq\\ (S_k/\tau-S_{\ell(k)-1})h(\pv^{(\ell)})+(S_{\ell(k)}-S_k/\tau)h_r(\pv^{(\ell)})+(1+\hat{P})S_k\varepsilon
\end{multline*}
and
\begin{multline*}
-\sum_{r=S_{\ell(k)-1}}^{S_k/\tau}\log p_{a_r,b_r}^{(\ell)}-\sum_{r=S_k/\tau}^{S_{\ell(k)}}\log p_{a_r}^{(\ell)}\geq (S_{\ell(k)}-S_k/\tau)h_r(\pv^{(\ell)})-m_{\ell}\varepsilon\geq\\ (S_k/\tau-S_{\ell(k)-1})h(\pv^{(\ell)})+(S_{\ell(k)}-S_k/\tau)h_r(\pv^{(\ell)})-\varepsilon S_k\max_{\pv\in C}h(\pv).
\end{multline*}
  \item Similarly, if $\tau S_{\ell(k)}-\varepsilon\tau m_{\ell(k)-1}\leq S_k\leq\tau S_{\ell(k)}$
\begin{multline}%\label{}
  (S_k/\tau-S_{\ell(k)-1})h(\pv^{(\ell)})+(S_{\ell(k)}-S_k/\tau)h_r(\pv^{(\ell)})-\varepsilon S_k\max_{\pv\in C}h_r(\pv)\leq \\
  -\sum_{r=S_{\ell(k)-1}}^{S_k/\tau}\log p_{a_r,b_r}^{(\ell)}-\sum_{r=S_k/\tau}^{S_{\ell(k)}}\log p_{a_r}^{(\ell)}\leq \\
  (S_k/\tau-S_{\ell(k)-1})h(\pv^{(\ell)})+(S_{\ell(k)}-S_k/\tau)h_r(\pv^{(\ell)})+(1+\hat{P})S_k\varepsilon,
\end{multline}
  \item and if $\tau S_{\ell(k)-1}+\varepsilon\tau m_{\ell(k)-1}<S_k<\tau S_{\ell(k)}-\varepsilon\tau m_{\ell(k)-1}$ then
$$
\left|(S_k/\tau-S_{\ell(k)-1})h(\pv^{(\ell)})+(S_{\ell(k)}-S_k/\tau)h_r(\pv^{(\ell)})-\left(-\sum_{r=S_{\ell(k)-1}}^{S_k/\tau}\log p_{a_r,b_r}^{(\ell)}-\sum_{r=S_k/\tau}^{S_{\ell(k)}}\log p_{a_r}^{(\ell)}\right)\right|\leq2\varepsilon S_k.
$$
\end{itemize}
Equation \eqref{eq:justthis} follows from the fact that the choice of $\varepsilon$ was arbitrary.
\end{proof}

\begin{lemma}\label{lem:goodBer}
Let $C\subset\Upsilon^o$ be compact set and let $\{\pv^{(k)}\}_{k=1}^{\infty}$ be a sequence of prob. vectors such that $\pv^{(k)}\in C$ for all $k\geq1$. Moreover, $\{m_k\}_{k=1}^{\infty}$ be a sequence of integers such that \eqref{eq:piecewisecond} hold. If $\mu$ is the piecewise Bernoulli measure corresponding to $m_k$ and $\pv^{(k)}$ then for $\mu$-a.e. $\ii$
$$
\liminf_{q\to\infty}\frac{\log\mu(B_q(\ii))}{-q\log N}=\min\left\{\liminf_{k\to\infty}\frac{\log\mu(B_{S_k}(\ii))}{-S_k\log N},\liminf_{\ell\to\infty}\frac{\log\mu(B_{\tau S_{\ell}}(\ii))}{-\tau S_{\ell}\log N}\right\}.
$$
\end{lemma}

\begin{proof}
Let $\varepsilon>0$ be arbitrary small but fixed. Let $\ii\in\Omega$ and let $K=K(\ii,\varepsilon)$, where the set $\Omega$ and the constant $K=K(\varepsilon,\ii)$ defined in Lemma~\ref{cor:weneed}. Let $k,\ell$ be integers such that $S_{k-1}\leq q< S_k$ and $\tau S_{\ell-1}\leq q<\tau S_{\ell}$. We may assume that $\ell>K+1$.

If $q\in(S_{k-1},S_{k-1}+\varepsilon m_{k-1})$ then
\begin{equation}\label{eq:genlem1}
\frac{\log\mu(B_q(\ii))}{-q\log N}\geq\frac{\log\mu(B_{S_{k-1}}(\ii))}{-(S_{k-1}+\varepsilon m_{k-1})\log N}\geq\frac{1}{1+\varepsilon}\frac{\log\mu(B_{S_{k-1}}(\ii))}{-S_{k-1}\log N}.
\end{equation}
Similarly, if $q\in(\tau S_{\ell-1},\tau S_{\ell-1}+\varepsilon\tau m_{\ell-1})$ then
\begin{equation}\label{eq:genlem2}
\frac{\log\mu(B_q(\ii))}{-q\log N}\geq\frac{1}{1+\varepsilon}\frac{\log\mu(B_{\tau S_{\ell-1}}(\ii))}{-\tau S_{\ell-1}\log N}.
\end{equation}
So without loss of generality, we may assume that $$S_{k-1}+\varepsilon m_{k-1}\leq q\leq S_k\text{ and }\tau S_{\ell-1}+\varepsilon\tau m_{\ell-1}\leq q\leq\tau S_{\ell}.$$
By Lemma~\ref{cor:weneed} and \eqref{eq:meas},
\begin{multline*}
\frac{\log\mu(B_q(\ii))}{-q\log N}\geq\\ \frac{\sum_{n=K}^{\ell-1}m_nh(\pv^{(n)})+(q/\tau-S_{\ell-1})h(\pv^{(\ell)})-\sum_{r=q/\tau}^{S_{\ell}}\log p_{a_r}^{(\ell)}+\sum_{n=\ell+1}^{k-1}m_nh_r(\pv^{(n)})+(q-S_{k-1})h_r(\pv^{(k)})}{q\log N}-C'\varepsilon,
\end{multline*}
where $C'$ depends only on $\varepsilon>0$ and the compact set $C$ but independent of $q$. If $q\in(\tau S_{\ell}-\varepsilon\tau m_{\ell-1},\tau S_{\ell})$ then the right hand side is greater than or equal to
\begin{equation}\label{eq:genlem3}
\frac{\log\mu(B_q(\ii))}{-q\log N}\geq\frac{\sum_{n=K}^{\ell}m_nh(\pv^{(n)})+\sum_{n=\ell+1}^{k-1}m_nh_r(\pv^{(n)})+(\tau S_{\ell}-S_{k-1})h_r(\pv^{(k)})}{\tau S_{\ell}\log N}-C''\varepsilon,
\end{equation}
where $C''\geq C'$ but depend only on $\varepsilon>0$ and the compact set $C$.

On the other hand, if $q\in(\tau S_{\ell-1}+\varepsilon\tau m_{\ell-1},\tau S_{\ell}-\varepsilon\tau m_{\ell-1})$ then by Lemma~\ref{cor:weneed}
\begin{multline*}
\frac{\log\mu(B_q(\ii))}{-q\log N}\geq\\ \frac{\sum_{n=K}^{\ell-1}m_nh(\pv^{(n)})+(q/\tau-S_{\ell-1})h(\pv^{(\ell)})+(S_{\ell}-q/\tau)h_r(\pv^{(\ell)})+\sum_{n=\ell+1}^{k-1}m_nh_r(\pv^{(n)})+(q-S_{k-1})h_r(\pv^{(k)})}{q\log N}-C\varepsilon.
\end{multline*}
Clearly, the right hand side of the inequality is either strictly increasing or strictly decreasing in $q$, thus
\begin{multline}\label{eq:genlem4}
\frac{\log\mu(B_q(\ii))}{-q\log N}\geq\min_{q=S_{k-1},\tau S_{\ell-1},S_k,\tau S_{\ell}}\left\{\frac{\sum_{n=K}^{\ell-1}m_nh(\pv^{(n)})+(q/\tau-S_{\ell-1})h(\pv^{(\ell)})}{q\log N}\right.\\\left.\frac{+(S_{\ell}-q/\tau)h_r(\pv^{(\ell)})+\sum_{n=\ell+1}^{k-1}m_nh_r(\pv^{(n)})+(q-S_{k-1})h_r(\pv^{(k)})}{q\log N}\right\}-C\varepsilon.
\end{multline}
Hence, the statement of the lemma follows by Lemma~\ref{lem:subseq}, \eqref{eq:genlem1},\eqref{eq:genlem2},\eqref{eq:genlem3}, \eqref{eq:genlem4} and the fact that the choice of $\varepsilon>$ was arbitrary.
\end{proof}

\subsection{Notes on entropy and coverings}

Let $\pv_D$ be the unique measure with maximal entropy, let $\pv_R$ be the measure with maximal entropy among the measures with maximal row-entropy, and let $\pv_d$ be the measure with maximal dimension. That is,
\begin{equation}\label{eq:defmeas}
\pv_D=\left(\frac{1}{D}\right)_{(a,b)\in Q}, \pv_R=\left(\frac{1}{T_aR}\right)_{(a,b)\in Q}, \pv_d=\left(\frac{T_a^{\frac{1}{\tau}-1}}{\sum_{a'\in S}T_{a'}^{\frac{1}{\tau}}}\right)_{(a,b)\in Q}.
\end{equation}
The first two formulas are obvious. The proof for the third one can be found in \cite[Theorem~4.5]{bedford1984crinkly} or alternatively \cite[Theorem on p. 1]{McMullen}.

\begin{lemma}
The functions $\pv\mapsto h(\pv)$, $\pv\mapsto h_r(\pv)$, $\pv\mapsto \dim(\pv)$ are continuous and concave on $\Upsilon$.
\end{lemma}

\begin{proof}Proof is straightforward.\end{proof}

Given $0\leq z\leq\log R$ let
$$
\psi(z)=\max\{h(\pv):h_r(\pv)=z\},
$$
and for $0\leq z\leq\log D$
$$
\varphi(z)=\max\{h_r(\pv):h(\pv)=z\}.
$$

\begin{lemma}\label{lem:entropyrelations}
The functions $a\mapsto\psi(a)$, $a\mapsto\varphi(a)$, $a\mapsto\frac{\psi(a)+(\tau-1)a}{\log M}$ and $a\mapsto\frac{a+(\tau-1)\varphi(a)}{\log M}$ are concave on their domains. Moreover, $h_r(\pv_D)\leq h_r(\pv_d)\leq h_r(\pv_R)=\log R$ and $h(\pv_R)\leq h(\pv_d)\leq h(\pv_D)=\log D$ and equalities hold if and only if $T_a$ takes only one value for all $a\in S$.
\end{lemma}

\begin{proof}
First, we prove the concavity of $\psi$. It is easy to see that for any $\pv\in\Upsilon$, $h(\pv')\geq h(\pv)$ and $h_r(\pv')=h_r(\pv)$, where
$$
p_{a,b}'=\frac{\sum_{b}p_{a,b}}{T_a}=\frac{p_a}{T_a}.
$$
Hence,
$$
\psi(z)=z+\max\{\sum_{a\in S}p_a\log T_a:-\sum_{a\in S}p_a\log p_a=z\}.
$$
However,
$$
\max\{\sum_{a\in S}p_a\log T_a:-\sum_{a\in S}p_a\log p_a=z\}=\max\{\sum_{a\in S}p_a\log T_a:-\sum_{a\in S}p_a\log p_a\geq z\}.
$$
Indeed, the $\leq$ direction is trivial. Let $\sharp_1=\sharp\{a\in S: T_a=\min_{a'\in S}T_{a'}\}$ and $\sharp_2=\sharp\{a\in S: T_a=\max_{a'\in S}T_{a'}\}$. If $\{p_a\}_{a\in S}$ is a maximizing vector such that $-\sum_{a\in S}p_a\log p_a>z$ then by choosing
$$
p_a'=\begin{cases}
p_a-\delta/\sharp_1, & \hbox{if }a\text{ s.t. }T_a=\min_{a'\in S}T_{a'}; \\
p_a+\delta/\sharp_2, & \hbox{if }a\text{ s.t. }T_a=\max_{a'\in S}T_{a'}; \\
p_a, & \hbox{otherwise}
\end{cases}
$$
we have $\sum_{a\in S}p_a\log T_a<\sum_{a\in S}p_a'\log T_a$ and $-\sum_{a\in S}p_a'\log p_a'> z$ for sufficiently small choice of $\delta$, which contradicts to the assumption that $\{p_a\}_{a\in S}$ is a maximizing vector.

Therefore,
\begin{multline*}
q\psi(z_1)+(1-q)\psi(z_2)=\\ qz_1+(1-q)z_2+\max\{\sum_{a\in S}q p_a+(1-q)p_a'\log T_a:-\sum_{a\in S}p_a\log p_a\geq z_1,\ -\sum_{a\in S}p_a'\log p_a'\geq z_2\}\leq\\
qz_1+(1-q)z_2+\max\{\sum_{a\in S}q p_a+(1-q)p_a'\log T_a:-\sum_{a\in S}(qp_a+(1-q)p_a')\log (qp_a+(1-q)p_a')\geq qz_1+(1-q)z_2\}=\\
\psi(qz_1+(1-q)z_2).
\end{multline*}

The concavity of $\varphi$ follows by the fact that $\varphi(\psi(z))=z$ for $z\in[h_r(\pv_D),\log R]$.

The proof of the second statement of the lemma follows from simple algebraic manipulations.
\end{proof}
For the graph of the function $\psi$ on $[h_r(\pv_D),\log R]$, see Figure~\ref{fig:psi}.
\begin{figure}
  \centering
  \includegraphics[width=60mm]{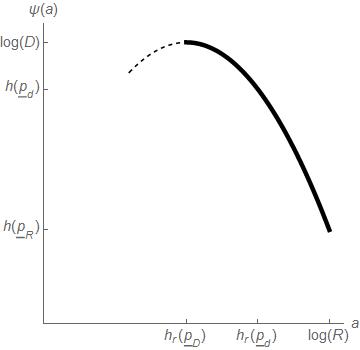}
  \caption{The graph of the function $\psi$ on $[h_r(\pv_D),\log R]$.}\label{fig:psi}
\end{figure}

%Let
%$$
%\Upsilon_{\psi}:=\left\{\pv\in\Upsilon:\psi(h_r(\pv))=h(\pv)\ \&\ h_r(\pv)\geq h_r(\pv_D)\right\}.
%$$

%In particular, if $\pv_1,\pv_2\in\Upsilon_{\psi}$ are such that $h_r(\pv_d)\leq h_r(\pv_1)\leq h_r(\pv_2)$ the

%\textbf{(Some similarity in case $\alpha<\tau-1$ and for $\pv_1,\pv_3$.)}

Let us denote the symbolic space formed by only the row symbols by $\Xi=S^{\N}$ and the set of finite length words by $\Xi^*=\bigcup_{n=0}^{\infty}S^n$. Let $\Pi$ be the natural correspondence function between $\Sigma$ and $\Xi$ (and between $\Sigma^*$ and $\Xi^*$ respectively). That is, for $\ii=((a_1,b_1),(a_2,b_2),\dots)\in\Sigma$ let $\Pi(\ii)=(a_1,a_2,\dots)$ (and for $\iv=((a_1,b_1),\dots,(a_n,b_n))\in\Sigma^*$ let $\Pi(\iv)=(a_1,\dots,a_n)$).

For a finite word $\iv\in\Sigma^*$ and for $\jv\in\Xi^*$ let us define the entropy of $\iv$ and row-entropy of $\jv$ as follows. Let us denote the frequency of a symbol $(a,b)$ and a row-symbol $a$ in $\iv$ and $\jv$ by
$$
\upsilon_{a,b}(\iv)=\frac{\sharp\{k=1,\dots,|\iv|:(a_k,b_k)=(a,b)\}}{|\iv|}\text{ and }\upsilon_{a}(\jv)=\frac{\sharp\{k=1,\dots,|\jv|:a_k=a\}}{|\jv|}.
$$
We can also define the frequency of rows for finite words $\iv\in\Sigma^*$ in the natural way,
$$
\upsilon_a(\iv):=\upsilon_a(\Pi(\iv))=\sum_{b\in P_a}\upsilon_{a,b}(\iv).
$$
And let for $\iv\in\Sigma^*$ and $\jv\in\Xi^*$
$$
h(\iv)=-\sum_{(a,b)\in Q}\upsilon_{a,b}(\iv)\log\upsilon_{a,b}(\iv),h_r(\iv)=-\sum_{a\in S}\upsilon_{a}(\iv)\log\upsilon_{a}(\iv)\text{ and }h_r(\jv)=-\sum_{a\in S}\upsilon_{a}(\jv)\log\upsilon_{a}(\jv).
$$

\begin{lemma}\label{lem:techent}
  For every $\varepsilon>0$ there exists $N\geq1$ such that for all $n\geq N$ and every $h\leq\log D$ and $h_r\leq\log R$
$$
\sharp\left\{\iv\in Q^n:h(\iv)\leq h\right\}\leq e^{n(h+\varepsilon)}
$$
and
$$
\sharp\left\{\jv\in S^n:h_r(\jv)\leq h_r\right\}\leq e^{n(h_r+\varepsilon)}.
$$
\end{lemma}

\begin{proof}
  We prove only the first inequality, the proof of the second one is analogous.

  Let $W=\{\pv\in\Upsilon:h(\pv)\leq h\}$. Then
  $$
  \left\{\iv\in Q^n:h(\iv)\leq h\right\}=\bigcup_{\substack{(q_{(a,b)})_{(a,b)\in Q}\in\N^D\\ \sum\limits_{(a,b)\in Q}q_{(a,b)}=n \\ -\sum\limits_{(a,b)\in Q}\frac{q_{(a,b)}}{n}\log\frac{q_{(a,b)}}{n}\leq h}}\left\{\iv\in Q^n:v_{a,b}(\iv)n=q_{(a,b)}\text{ for }(a,b)\in Q\right\}.
  $$
  For any $q_{(a,b)}\in\N$ with $\sum_{(a,b)\in Q}q_{(a,b)}=n$,
  $$
  \sharp\left\{\iv\in Q^n:\ v_{a,b}(\iv)n=q_{(a,b)}\text{ for }(a,b)\in Q\right\}=\frac{n!}{\prod\limits_{(a,b)\in Q}q_{(a,b)}!}.
  $$
  By using Stirling's formula, there exists $C>0$ such that for every $k\geq1$
  $$
  C^{-1}\frac{k^k}{e^k}\sqrt{2\pi k}\leq k!\leq C\frac{k^k}{e^k}\sqrt{2\pi k}.
  $$
  Thus, for any $q_{(a,b)}\in\N$ with $\sum_{(a,b)\in Q}q_{(a,b)}=n$ and with $-\sum\limits_{(a,b)\in Q}\frac{q_{(a,b)}}{n}\log\frac{q_{(a,b)}}{n}\leq h$
  $$
  \frac{n!}{\prod\limits_{(a,b)\in Q}q_{(a,b)}!}=\frac{n!}{\prod\limits_{\substack{(a,b)\in Q\\q_{(a,b)}\geq1}}q_{(a,b)}!}\leq C^{D+1}\sqrt{n}e^{-n\sum_{(a,b)\in Q}\frac{q_{(a,b)}}{n}\log\frac{q_{(a,b)}}{n}}\leq C^{D+1}\sqrt{n}e^{nh}.
  $$
  On the other hand, $\sharp\{(q_{(a,b)})_{(a,b)\in Q}\in\N^D: \sum\limits_{(a,b)\in Q}q_{(a,b)}=n\}\leq\binom{n+D+1}{D}$. Hence, for every $\varepsilon>0$ one can choose $N>1$ such that for every $n\geq N$
  $$
  \left\{\iv\in Q^n:h(\iv)\leq h\right\}\leq \binom{n+D+1}{D}C^{D+1}\sqrt{n}e^{nh}\leq e^{n(h+\varepsilon)}.
  $$
\end{proof}

\begin{lemma}\label{lem:crossbound}
For every $z\in [h_r(\pv_D), \log R]$ and for every $\varepsilon>0$ there exists $N>0$ such that for every $n>N$
\[
\sharp\{\iv\in Q^n: h_r(\iv)\geq z\} \leq e^{n(\psi(z)+\varepsilon)}.
\]
Moreover, For every $x\in [h(\pv_R), \log D]$ and for every $\varepsilon>0$ there exists $N'>0$ such that for every $n>N'$
\[
\sharp\Pi\{\iv\in Q^n: h(\iv)\geq x\} \leq e^{n(\varphi(x)+\varepsilon)}.
\]
\end{lemma}

\begin{proof}
By Lemma~\ref{lem:entropyrelations}, $\psi(z)$ is monotone decreasing on the interval $[h_r(p_D), \log R]$. So, if $h_r(\iv)\geq z$ then $h(\iv)\leq\psi(h_r(\iv))\leq\psi(z)$. Thus, the statement follows by Lemma~\ref{lem:techent}.
\end{proof}

\begin{lemma}\label{lem:cov}
Let $V_n$ be a set of finite length words with length $n$. Denote the number of approximate squares with side length $N^{-n}$ required to cover $\mathcal{V}=\bigcup_{\iv\in V_n}B_n(\iv)$ by $V_n'$. Then
$$
V_n'\leq\sharp\{\jv\in Q^{n/\tau}:\jv=\iv|_1^{n/\tau}\}\cdot\sharp\{\jv'\in S^{n(1-1/\tau)}:\Pi(\iv|_{n/\tau+1}^n)=\jv'\}.
$$
\end{lemma}

The proof is straightforward.

\subsection{Properties of the dimension functions}

Let
$$
\Upsilon_{\psi}:=\left\{\pv\in\Upsilon:\psi(h_r(\pv))=h(\pv)\ \&\ h_r(\pv)\geq h_r(\pv_D)\right\},
$$
and let
\begin{multline*}
\Theta:=\left\{(\pv_-,\pv_1,\pv_2,\pv_+)\in (\Upsilon_{\psi})^4:h_r(\pv_-)\in[h_r(\pv_D),h_r(\pv_d)],\ h_r(\pv_1)\in[h_r(\pv_-),\log R]\right.\\
\left.\text{ and }h_r(\pv_+)\in[h_r(\pv_d),\log R]\right\},
\end{multline*}
where $\pv_R,\pv_d$ and $\pv_D$ are defined in \eqref{eq:defmeas}.

\begin{lemma}\label{lem:region}
For any $h\geq0$,
$$
\max_{\pv_-,\pv_1,\pv_2,\pv_+\in\Upsilon} \DJ_{\ALPHA}(\pv_-,\pv_1,\pv_2,\pv_+,h)=\max_{(\pv_-,\pv_1,\pv_2,\pv_+)\in\Theta}\DJ_{\ALPHA}(\pv_-,\pv_1,\pv_2,\pv_+,h)
$$
\end{lemma}

\begin{proof}
It is easy to see that $\DJ_{\ALPHA}(\pv_-,\pv_1,\pv_2,\pv_+,h)$ is monotone increasing in $h(\pv_i)$ and $h_r(\pv_i)$ for $i=-,1,2,+$. Thus, for fixed $h_r(\pv_i)$ the value of $\DJ_{\ALPHA}$ can be increased by replacing $h(\pv_i)$ with $\psi(h_r(\pv_i))$.

On the other hand, since $\psi$ is continuous and concave with maxima at $h_r(\pv_D)$, if $h_r(\pv_D)\geq h_r(\pv_i)$ for some $i=-,1,2,+$ then the value of $\DJ_{\ALPHA}$ can be increased by replacing $h_r(\pv_i)$ with $h_r(\pv_D)$, and replacing $h(\pv_i)=\psi(h_r(\pv_i))$ with $\log D=h(\pv_D)$. Thus, we've shown that
$$
\max_{\pv_-,\pv_1,\pv_2,\pv_+\in\Upsilon} \DJ_{\ALPHA}(\pv_-,\pv_1,\pv_2,\pv_+,h)=\max_{(\pv_-,\pv_1,\pv_2,\pv_+)\in(\Upsilon_{\psi})^4}\DJ_{\ALPHA}(\pv_-,\pv_1,\pv_2,\pv_+,h).
$$

Now, let $(\pv_-,\pv_1,\pv_2,\pv_+)\in(\Upsilon_{\psi})^4$ be arbitrary. Since the function $a\mapsto\frac{\psi(a)+(\tau-1)a}{\log M}$ is a concave function with maxima at $a=h_r(\pv_d)$ and $a\mapsto\psi(a)$ is strictly decreasing on $[h_r(\pv_D),\log R]$, if $h_r(\pv_-)>h_r(\pv_d)$ then by choosing $\pv'=(1-\varepsilon)\pv_-+\varepsilon\pv_d$, we get that $d_i(\pv',\pv_1,\pv_2,\pv_+)>d_i(\pv_-,\pv_1,\pv_2,\pv_+)$ for every $i=1,\dots,5$. Thus, $\DJ_{\ALPHA}(\pv',\pv_1,\pv_2,\pv_+)\geq\DJ_{\ALPHA}(\pv_-,\pv_1,\pv_2,\pv_+)$ and we may assume that $h_r(\pv_-)\leq h_r(\pv_d)$.

Now, suppose that $h_r(\pv_1)<h_r(\pv_-)\leq h_r(\pv_d)$. Then $h(\pv_1)>h(\pv_-)$ and $\dim\pv_1>\dim\pv_-$. Thus, $\DJ_{\ALPHA}(\pv_1,\pv_1,\pv_2,\pv_+)\geq\DJ_{\ALPHA}(\pv_-,\pv_1,\pv_2,\pv_+)$. Hence, we may assume that $h_r(\pv_1)\geq h_r(\pv_-)$.

Finally, if $h_r(\pv_+)<h_r(\pv_d)$ then $d_i(\pv_-,\pv_1,\pv_2,\pv_+)\leq d_i(\pv_-,\pv_1,\pv_2,\pv_d)$ for $i=4,5,6$, which implies that
$\DJ_{\ALPHA}(\pv_-,\pv_1,\pv_2,\pv_d)\geq\DJ_{\ALPHA}(\pv_-,\pv_1,\pv_2,\pv_+)$. Thus, we may assume that $h_r(\pv_+)\geq h_r(\pv_d)$.
\end{proof}

Denote the maximizing subset of $\Theta$ by $\Theta_{M}^{\alpha,H}$, i.e.
$$
\max_{(\pv_-',\pv_1',\pv_2',\pv_+')\in\Theta}\DJ_{\ALPHA}(\pv_-',\pv_1',\pv_2',\pv_+',H)=\DJ_{\ALPHA}(\pv_-,\pv_1,\pv_2,\pv_+,H)\text{ for every }(\pv_-,\pv_1,\pv_2,\pv_+)\in\Theta_M^{\alpha,H}.
$$

\begin{lemma}\label{lem:drop}
	If $(\pv_-,\pv_1,\pv_2,\pv_+)\in\Theta_M^{\alpha,H}$ and $\alpha>0$ then $\DJ_{\ALPHA}(\pv_-,\pv_1,\pv_2,\pv_+,H)<\dim\pv_d$. In particular, $\pv_-\neq\pv_d\neq\pv_+$.
\end{lemma}

\begin{proof}
	Suppose that $\dim\pv_d=\DJ_{\ALPHA}(\pv_-,\pv_1,\pv_2,\pv_+,H)$ for $(\pv_-,\pv_1,\pv_2,\pv_+)\in\Theta_M^{\alpha,H}$. Then $\pv_+=\pv_-=\pv_d$, otherwise $\DJ_{\alpha}<\dim\pv_d$. By Lemma~\ref{lem:entropyrelations} and Lemma~\ref{lem:region}, $h(\pv_1)\leq h(\pv_d)$. But then 
	$$
	\frac{h(\pv_d)+(\tau-1)h_r(\pv_2)}{(\tau+\alpha)\log N}\geq d_3^{\alpha}(\pv_d,\pv_1,\pv_2)\geq\dim\pv_d\text{ implies that }(\tau-1)h(\pv_2)\geq(\tau+\alpha-1)h_r(\pv_d)>(\tau-1)h_r(\pv_d),
	$$
	which implies by Lemma~\ref{lem:entropyrelations} that $h(\pv_2)<h(\pv_d)$.
	Thus,
	$$
	d_5^{\alpha}<\frac{\tau h(\pv_d)+(\tau+\alpha)(\tau-1)h_r(\pv_d)+\alpha(1-1/\tau)H}{\tau(\tau+\alpha)\log N}=\dim\pv_d+\frac{\alpha\left((\tau-1)H- \tau h(\pv_d)\right)}{\tau^2(\tau+\alpha)\log N}<\dim\pv_d,
	$$
	since $H\leq\max_a\log T_a\leq h(\pv_d)$ clearly. This contradicts to $\DJ_{\alpha}=\min_{i=1,\dots,6}\{d_i^{\alpha}\}$.
\end{proof}

\begin{lemma}\label{lem:6implies5}
If $(\pv_-,\pv_1,\pv_2,\pv_+)\in\Theta_M^{\alpha,H}$ and $\DJ_{\ALPHA}(\pv_-,\pv_1,\pv_2,\pv_+,H)=\dim\pv_+$ then $$\dim\pv_+=d_5^{\alpha}(\pv_-,\pv_1,\pv_2,\pv_+,H).$$

Moreover, if $\DJ_{\ALPHA}(\pv_-,\pv_1,\pv_2,\pv_+,H)=d_i(\pv_-,\pv_1,\pv_2,\pv_+,H)$ for $i=4,5,6$ then $h(\pv_2)=h(\pv_+)$ and $h_r(\pv_2)=h_r(\pv_+)$, i.e. $\pv_2=\pv_+$ can be chosen.
\end{lemma}

\begin{proof}
  Suppose that $\max_{(\pv_-',\pv_1',\pv_2',\pv_+')\in\Theta}\DJ_{\ALPHA}(\pv_-,\pv_1,\pv_2,\pv_+,h)=\dim\pv_+$. If $h_r(\pv_d)< h_r(\pv_+)$ and $\dim\pv_+<\min\{d_5^{\alpha}(\pvv),d_4^{\alpha}(\pvv)\}$ (where $\pvv=(\pv_-,\pv_1,\pv_2,\pv_+)$) then by taking $\varepsilon>0$ sufficiently small and $\pv_+'$ such that $h_r(\pv_+')=(1-\varepsilon)h_r(\pv_+)+\varepsilon h_r(\pv_d)$ and $h(\pv_+')=\psi(h_r(\pv_+'))$, we get $$\dim\pv_+<\dim\pv_+'<\min\{d_5^{\alpha}(\pvv'),d_4^{\alpha}(\pvv')\}<\min\{d_5^{\alpha}(\pvv),d_4^{\alpha}(\pvv)\},$$where $\pvv'=(\pv_-,\pv_1,\pv_2,\pv_+')$. This contradicts to $\pvv\in\Theta_M^{\alpha,H}$. Thus, either $d_5^{\alpha}(\pvv)=\dim\pv_+$ or $d_4^{\alpha}(\pvv)=\dim\pv_+$ or $\pv_+=\pv_d$. But by Lemma~\ref{lem:drop}, $\pv_+\neq\pv_d$ and therefore $\dim\pv_+=\min\{d_5^{\alpha},d_4^{\alpha}\}$. 

Next, we show that $\dim\pv_+=d_5^{\alpha}$. Contrary suppose that $d_4^{\alpha}=\dim\pv_+<d_5^{\alpha}$. Simple manipulations shows that
\begin{equation}\label{eq:tech2}
\frac{\tau+\alpha}{\tau-1}d_5^{\alpha}-\frac{1+\alpha}{\tau-1}d_4^{\alpha}=\frac{h(\pv_2)-h_r(\pv_2)}{\log M}+\frac{\tau h_r(\pv_+)}{\log M}=\frac{h(\pv_2)-h_r(\pv_2)}{\log M}-\frac{h(\pv_+)-h_r(\pv_+)}{\log M}+\dim\pv_+.
\end{equation}
Thus, if $d_4^{\alpha}=\dim\pv_+<d_5^{\alpha}$ then
$$
0<\frac{\tau+\alpha}{\tau-1}\left(d_5^{\alpha}-\dim\pv_+\right)=\frac{h(\pv_2)-h_r(\pv_2)}{\log M}-\frac{h(\pv_+)-h_r(\pv_+)}{\log M}.
$$
Hence, $h_r(\pv_2)< h_r(\pv_+)$. Indeed, if $h_r(\pv_2)\geq h_r(\pv_+)$ and $h(\pv_2)-h_r(\pv_2)>h(\pv_+)-h_r(\pv_+)$ then $\psi(h_r(\pv_2))=h(\pv_2)>h(\pv_+)=\psi(h_r(\pv_+))$, which contradicts to the fact that $\psi$ monotone decreasing on $[h_r(\pv_D),\log R]$. But if $h_r(\pv_2)< h_r(\pv_+)$ then the value of $d_5^{\alpha}$ can be decreased and $d_4^{\alpha}$ can be increased by increasing $h_r(\pv_2)$ (the value $\dim\pv_+$ does not change), which contradicts to $(\pv_-,\pv_1,\pv_2,\pv_+)\in\Theta_M^{\alpha,H}$, and shows the first assertion of the lemma.

On the other hand, if $d_4^{\alpha}=d_5^{\alpha}=\dim\pv_+$ then by \eqref{eq:tech2}
$$
h(\pv_2)-h_r(\pv_2)=h(\pv_+)-h_r(\pv_+).
$$
Since $\psi$ is monotone decreasing on $[h_r(\pv_d),\log R]$, the proof is complete.
\end{proof}

\section{Lower bounds}\label{sec:lb}

\begin{prop}\label{prop:lbcyl}
 Let $f$ be a function defined in \eqref{eq:lingrov}. Then, for any $\pv_-,\pv_1,\pv_2,\pv_+\in\Upsilon^o$
$$
\dim_H\pi\Gamma_C(f,\{\jj_k\})\geq \DJ_{\ALPHA}(\pv_-,\pv_1,\pv_2,\pv_+,0).
$$
where $\{\jj_k\}$ is an arbitrary sequence.
\end{prop}

\begin{proof}
Let $q\geq1$ be arbitrary but fixed integer and let $\{n_k\}$ be a sequence of integers numbers such that
  \begin{eqnarray}
   \lim_{k\to\infty} \frac{\sum_{l=1}^{k-1}n_l}{n_k} &=& 0,\text{ and}\label{eq:nk2} \\
    n_{2k+1} &> & \tau^{2q}n_{2k}\text{ for all $k\geq1$.}\label{eq:nk3}
  \end{eqnarray}
   Let $\Gamma'$ be the set for which
   $$
   \Gamma'=\left\{\ii\in\Sigma:\sigma^{n_{2k-1}}\ii\in C_{f(n_{2k-1})}(\jj_{n_{2k-1}})\text{ for all }k\geq1\right\}.
   $$
   Clearly, $\Gamma'\subset\Gamma_C(f)$. We divide the proof into two cases.\\

  {\bf Case I: $\alpha\geq\tau-1$}. Observe, that in this case $m(\alpha,\tau)=1$ and $M(\alpha,\tau)=0$, thus $\pv_1$ does not play any role.

  Let
$$
\mu=\prod_{\ell=1}^{\infty}\eta_{\ell},
$$
be a measure supported on $\Gamma'$, where $\eta_{\ell}$ a probabilty measure on $Q$ such that
$$
\eta_{\ell}=\left\{
         \begin{array}{ll}
           \delta_{(\jj_{n_{2k-1}})_{\ell-n_{2k-1}+1}}, & \hbox{if }\ell\in(n_{2k-1},n_{2k-1}+f(n_{2k-1})], \\
           \pv_2, & \hbox{if }\ell\in(n_{2k-1}+f(n_{2k-1}),\tau n_{2k-1}+f(n_{2k-1})], \\
           \pv_+, & \hbox{if }\ell\in(\tau n_{2k-1}+f(n_{2k-1}),n_{2k}] \\
           \pv^{(m)}, & \hbox{if }\ell\in(\tau^{m}n_{2k},\tau^{m+1}n_{2k}], \\
           \pv_- & \hbox{if }\ell\in(\tau^{1/\varepsilon}n_{2k},n_{2k+1}],
         \end{array}
       \right.
$$
where $m=0,\dots,q$ and
\begin{equation}\label{eq:pmchoice}
\pv^{(m)}=\left(1-\frac{m}{q}\right)\pv_++\frac{m}{q}\pv_-.
\end{equation}

By \eqref{eq:nk2} and Lemma~\ref{lem:subseq}, it is easy to see that
\begin{eqnarray}
% \nonumber % Remove numbering (before each equation)
  \liminf_{k\to\infty}\frac{\log\mu(B_{n_{2k-1}}(\ii))}{-n_{2k-1}\log N} &=& \dim\pv_-,\label{eq:in1} \\
 \liminf_{k\to\infty}\frac{\log\mu(B_{\tau n_{2K-1}}(\ii))}{-\tau n_{2k-1}\log N} &=& \frac{h(\pv_-)}{\tau \log N},\label{eq:inball1} \\
 \liminf_{k\to\infty}\frac{\log\mu(B_{n_{2k-1}+f(n_{2k-1})}(\ii))}{-(n_{2k-1}+f(n_{2k-1}))\log N} &=& \frac{h(\pv_-)}{(1+\alpha)\log N},\label{eq:inball2} \\
  \liminf_{k\to\infty}\frac{\log\mu(B_{\tau(n_{2k-1}+f(n_{2k-1}))}(\ii))}{-\tau(n_{2k-1}+f(n_{2k-1}))\log N} &=& \frac{h(\pv_-)+(\tau-1)h_r(\pv_2)+(\tau-1)\alpha h_r(\pv_+)}{\tau(1+\alpha)\log N},\label{eq:onlythis} \\
  \liminf_{k\to\infty}\frac{\log\mu(B_{\tau n_{2k-1}+f(n_{2k-1})}(\ii))}{-(\tau n_{2k-1}+f(n_{2k-1}))\log N} &=& \frac{h(\pv_-)+(\tau-1)h_r(\pv_2)}{(\tau+\alpha)\log N},\label{eq:inball4} \\
  \liminf_{k\to\infty}\frac{\log\mu(B_{\tau(\tau n_{2k-1}+f(n_{2k-1}))}(\ii))}{-\tau(\tau n_{2k-1}+f(n_{2k-1}))\log N} &=& \frac{h(\pv_-)+(\tau-1)h_r(\pv_2)+(\tau-1)(\tau+\alpha)h_r(\pv_+)}{\tau(\tau+\alpha)\log N}, \\
   \liminf_{k\to\infty}\frac{\log\mu(B_{n_{2k}}(\ii))}{-n_{2k}\log N} &=& \dim\pv_+,\label{eq:inball3} \\
 \liminf_{k\to\infty}\frac{\log\mu(B_{\tau^mn_{2k}}(\ii))}{-\tau^mn_{2k}\log N} &=& \frac{h(\pv_+)+\sum_{n=1}^{m-1}\tau^{n-1}(\tau-1)h(\pv^{(n)})+\tau^{m-1}(\tau-1)h_r(\pv^{(m)})}{\tau^m\log N}\label{eq:in2}.
\end{eqnarray}

For the comfortability of the reader, we give details to \eqref{eq:onlythis}, the proof of the other equations is similar. By Lemma~\ref{lem:subseq},
\begin{multline*}
  \liminf_{k\to\infty}\frac{\log\mu(B_{\tau(n_{2k-1}+f(n_{2k-1}))}(\ii))}{-\tau(n_{2k-1}+f(n_{2k-1}))\log N}= \\
  \liminf_{k\to\infty}\left(\frac{\sum_{\ell=1}^{k-1}\left((\tau-1)n_{2\ell-1}h(\pv_2)+(n_{2\ell}-\tau n_{2\ell-1}-f(n_{2\ell-1}))h(\pv_+)+\sum_{m=0}^{q}\tau^m(\tau-1)n_{2\ell}h(\pv^{(m)})\right)}{\tau(n_{2k-1}+f(n_{2k-1}))\log N}+\right.\\
\left.\frac{\sum_{\ell=1}^{k-1}(n_{2\ell+1}-\tau^{q}n_{2\ell})h(\pv_-)+(\tau-1)n_{2k-1}h_r(\pv_2)+(\tau-1)f(n_{2k-1})h_r(\pv_+)}{\tau(n_{2k-1}+f(n_{2k-1}))\log N}\right).
\end{multline*}
The statement follows by \eqref{eq:lingrov} and \eqref{eq:nk2}.

Now we show that the formula in \eqref{eq:in2} can be bounded below by $\min\{\dim\pv_-,\dim\pv_+\}-O(1/q)$. By using the continuity of the entropy, we have that $|h_r(\pv^{(m)})-h_r(\pv^{(m+1)})|\leq O(1/q)$ and thus,
\begin{multline*}
\frac{h(\pv_+)+\sum_{n=1}^{m-1}\tau^{n-1}(\tau-1)h(\pv^{(n)})+\tau^{m-1}(\tau-1)h_r(\pv^{(m)})}{\tau^m\log N}=\\
\frac{h(\pv_+)+\sum_{n=1}^{m-1}\left(\tau^{n-1}(\tau-1)h(\pv^{(n)})+\tau^{n}(\tau-1)h_r(\pv^{(n+1)})-\tau^{n-1}(\tau-1)h_r(\pv^{(n)})\right)+(\tau-1)h_r(\pv^{(1)})}{\tau^m\log N}\geq\\
\frac{\dim\pv_++\sum_{n=1}^{m-1}\tau^{n-1}(\tau-1)\log M\dim\pv^{(n)}}{\tau^m\log N}-O(1/q)\frac{\tau^m-1}{\tau^m\log N}
\end{multline*}
By the concavity of the entropy and the dimension,
$$
\dim\pv^{(m)}\geq\left(1-\frac{m}{q}\right)\dim\pv_++\frac{m}{q}\dim\pv_-\geq\min\{\dim\pv_+,\dim\pv_-\},
$$
which implies that
\begin{equation}\label{eq:in3}
\liminf_{k\to\infty}\frac{\log\mu(B_{\tau^mn_{2k}}(\ii))}{-\tau^mn_{2k}\log N}\geq\frac{\dim\pv_++\sum_{n=1}^{m-1}\tau^{n-1}(\tau-1)\log M\dim\pv^{(n)}}{\tau^m\log N}-O(\varepsilon)\geq\min\{\dim\pv_+,\dim\pv_-\}-O(1/q).
\end{equation}
Equations \eqref{eq:in1}-\eqref{eq:in2} and \eqref{eq:in3} together with Lemma~\ref{lem:goodBer} imply that $\dim_H\pi\Gamma_C(f)\geq\dim_H\Gamma'\geq \DJ_{\ALPHA}(\pv_-,\pv_1,\pv_2,\pv_+,0)-O(1/q)$, and since $q\geq1$ was arbitrary, the assertion is proven.\\

{\bf Case II: $\alpha<\tau-1$}.  Let
$$
\mu=\prod_{\ell=1}^{\infty}\eta_{\ell},
$$
be a measure supported on $\Gamma'$, where $\eta_{\ell}$ a probabilty measure on $Q$ such that
$$
\eta_{\ell}=\left\{
\begin{array}{ll}
\pv_1, & \hbox{if }\ell\in(\dfrac{n_{2k-1}+f(n_{2k-1})}{\tau},n_{2k-1}], \\
\delta_{(\jj_{n_{2k-1}})_{\ell-n_{2k-1}+1}}, & \hbox{if }\ell\in(n_{2k-1},n_{2k-1}+f(n_{2k-1})], \\
\pv_2, & \hbox{if }\ell\in(n_{2k-1}+f(n_{2k-1}),\tau n_{2k-1}+f(n_{2k-1})], \\
\pv_+, & \hbox{if }\ell\in(\tau n_{2k-1}+f(n_{2k-1}),n_{2k}] \\
\pv^{(m)}, & \hbox{if }\ell\in(\tau^{m}n_{2k},\tau^{m+1}n_{2k}], \\
\pv_- & \hbox{if }\ell\in(\tau^{q}n_{2k},\dfrac{n_{2k+1}+f(n_{2k+1})}{\tau}],
\end{array}
\right.
$$
where $m=0,\dots,q$ and $\pv^{(m)}$ as in \eqref{eq:pmchoice}. By \eqref{eq:nk2} and Lemma~\ref{lem:subseq}, we get that
\begin{eqnarray}
% \nonumber % Remove numbering (before each equation)
\liminf_{k\to\infty}\frac{\log\mu(B_{\frac{n_{2k-1}+f(n_{2k-1})}{\tau}}(\ii))}{-\frac{n_{2k-1}+f(n_{2k-1})}{\tau}\log N} &=& \dim\pv_-,\label{eq:in11} \\
\liminf_{k\to\infty}\frac{\log\mu(B_{n_{2k-1}+f(n_{2k-1})}(\ii))}{-(n_{2k-1}+f(n_{2k-1}))\log N} &=& \frac{\frac{1+\alpha}{\tau}h(\pv_-)+\left(1-\frac{1+\alpha}{\tau}\right)h_r(\pv_1)}{(1+\alpha)\log N},\label{eq:inball11} \\
\liminf_{k\to\infty}\frac{\log\mu(B_{n_{2k-1}}(\ii))}{-n_{2k-1}\log N} &=& \frac{\frac{1}{\tau}h(\pv_-)+\frac{\alpha}{\tau}h_r(\pv_-)+\left(1-\frac{1+\alpha}{\tau}\right)h_r(\pv_1)}{\log N},\label{eq:nem1} \\
\liminf_{k\to\infty}\frac{\log\mu(B_{\tau n_{2k-1}}(\ii))}{-\tau n_{2k-1}\log N} &=& \frac{\frac{1+\alpha}{\tau}h(\pv_-)+\left(1-\frac{1+\alpha}{\tau}\right)h(\pv_1)+(\tau-1-\alpha)h_r(\pv_2)}{\tau\log N},\label{eq:nem2} \\
\liminf_{k\to\infty}\frac{\log\mu(B_{\tau(n_{2k-1}+f(n_{2k-1}))}(\ii))}{-\tau(n_{2k-1}+f(n_{2k-1}))\log N} &=& \frac{\frac{1+\alpha}{\tau}h(\pv_-)+\left(1-\frac{1+\alpha}{\tau}\right)h(\pv_1)+(\tau-1)h_r(\pv_2)+(\tau-1)\alpha h_r(\pv)}{\tau(1+\alpha)\log N}, \\
\liminf_{k\to\infty}\frac{\log\mu(B_{\tau n_{2k-1}+f(n_{2k-1})}(\ii))}{-(\tau n_{2k-1}+f(n_{2k-1}))\log N} &=& \frac{\frac{1+\alpha}{\tau}h(\pv_-)+\left(1-\frac{1+\alpha}{\tau}\right)h(\pv_1)+(\tau-1)h_r(\pv_2)}{(\tau+\alpha)\log N},\label{eq:inball12} \\
\liminf_{k\to\infty}\frac{\log\mu(B_{\tau(\tau n_{2k-1}+f(n_{2k-1}))}(\ii))}{-\tau(\tau n_{2k-1}+f(n_{2k-1}))\log N} &=& \frac{\frac{1+\alpha}{\tau}h(\pv_-)+\left(1-\frac{1+\alpha}{\tau}\right)h(\pv_1)+(\tau-1)h_r(\pv_2)+(\tau-1)(\tau+\alpha)h_r(\pv_+)}{\tau(\tau+\alpha)\log N}, \\
\liminf_{k\to\infty}\frac{\log\mu(B_{n_{2k}}(\ii))}{-n_{2k}\log N} &=& \dim\pv_+,\label{eq:inball13} \\
\liminf_{k\to\infty}\frac{\log\mu(B_{\tau^mn_{2k}}(\ii))}{-\tau^mn_{2k}\log N} &=& \frac{h(\pv_+)+\sum_{n=1}^{m-1}\tau^{n-1}(\tau-1)h(\pv^{(n)})+\tau^{m-1}(\tau-1)h_r(\pv^{(m)})}{\tau^m\log N}.\label{eq:in12}
\end{eqnarray}
Now, we show that \eqref{eq:nem1} and \eqref{eq:nem2} cannot attain $\DJ_{\ALPHA}(\pv_-,\pv_1,\pv_2,\pv_+,0)$.
A simple algebraic calculation shows that
\begin{equation}\label{eq:toineq1}
\frac{\frac{1+\alpha}{\tau}h(\pv_-)+\left(1-\frac{1+\alpha}{\tau}\right)h_r(\pv_1)}{(1+\alpha)\log N}\leq\frac{\frac{1}{\tau}h(\pv_-)+\frac{\alpha}{\tau}h_r(\pv_-)+\left(1-\frac{1+\alpha}{\tau}\right)h_r(\pv_1)}{\log N}.
\end{equation}
On the other hand, if
\begin{multline}\label{eq:tech1}
\frac{\frac{1+\alpha}{\tau}h(\pv_-)+\left(1-\frac{1+\alpha}{\tau}\right)h(\pv_1)+(\tau-1-\alpha)h_r(\pv_2)}{\tau\log N}<\\
\min\left\{\frac{\frac{1+\alpha}{\tau}h(\pv_-)+\left(1-\frac{1+\alpha}{\tau}\right)h(\pv_1)+(\tau-1)h_r(\pv_2)}{(\tau+\alpha)\log N},\frac{\frac{1+\alpha}{\tau}h(\pv_-)+\left(1-\frac{1+\alpha}{\tau}\right)h_r(\pv_1)}{(1+\alpha)\log N}\right\}
\end{multline}
then $$
\frac{1+\alpha}{\tau}h(\pv_-)+\left(1-\frac{1+\alpha}{\tau}\right)h(\pv_1)<(1+\alpha)h_r(\pv_2).
$$
Thus,
$$
\frac{\frac{1+\alpha}{\tau}h(\pv_-)+\left(1-\frac{1+\alpha}{\tau}\right)h(\pv_1)+(\tau-1-\alpha)h_r(\pv_2)}{\tau\log N}>\frac{\frac{1+\alpha}{\tau}h(\pv_-)+\left(1-\frac{1+\alpha}{\tau}\right)h(\pv_1)}{(1+\alpha)\log N},
$$
and by \eqref{eq:tech1} we have $h(\pv_1)<h_r(\pv_1)$, which is a contradiction. Hence,
\begin{multline}\label{eq:toineq2}
\frac{\frac{1+\alpha}{\tau}h(\pv_-)+\left(1-\frac{1+\alpha}{\tau}\right)h(\pv_1)+(\tau-1-\alpha)h_r(\pv_2)}{\tau\log N}\geq\\
\min\left\{\frac{\frac{1+\alpha}{\tau}h(\pv_-)+\left(1-\frac{1+\alpha}{\tau}\right)h(\pv_1)+(\tau-1)h_r(\pv_2)}{(\tau+\alpha)\log N},\frac{\frac{1+\alpha}{\tau}h(\pv_-)+\left(1-\frac{1+\alpha}{\tau}\right)h_r(\pv_1)}{(1+\alpha)\log N}\right\}.
\end{multline}
By Lemma~\ref{lem:goodBer}, the inequalities \eqref{eq:toineq1}, \eqref{eq:toineq2} with the equations \eqref{eq:in11}-\eqref{eq:in12}, and \eqref{eq:in3} imply that $\dim_H\pi\Gamma_C(f,\{\jj_k\})\geq\dim_H\Gamma'\geq \DJ_{\ALPHA}(\pv_-,\pv_1,\pv_2,\pv_+,0)-O(1/q)$. Since $q\geq1$ was arbitrary, the proof is complete.
\end{proof}

\begin{prop}\label{prop:lbball}
Let $f$ be a function defined in \eqref{eq:lingrov} and let $\{\jj_k=((a_1^{(k)},b_1^{(k)}),(a_2^{(k)},b_2^{(k)}),\dots)\}$ is a sequence on $\Sigma$ such that $\lim_{n\to\infty}\frac{1}{f(n)(1-1/\tau)}\sum_{k=f(n)/\tau}^{f(n)}\log T_{a_k^{(n)}}=H$. Then, for any $\pv_-,\pv_1,\pv_2,\pv_+\in\Upsilon^o$
$$
\dim_H\pi\Gamma_B(f,\{\jj_k\})\geq \DJ_{\ALPHA}(\pv_-,\pv_1,\pv_2,\pv_+,H).
$$
\end{prop}

\begin{proof}
Let $q\geq1$ be arbitrary large but fixed integer and let $\{n_k\}$ be a sequence of integers numbers such that \eqref{eq:nk2} and \eqref{eq:nk3} hold. Let $\Gamma'$ be the set for which
   $$
   \Gamma'=\left\{\ii\in\Sigma:\sigma^{n_{2k-1}}\ii\in B_{f(n_{2k-1})}(\jj_{n_{2k-1}})\text{ for all }k\geq1\right\}.
   $$
   Clearly, $\Gamma'\subset\Gamma_B(f,\{\jj_k\})$. Just like in the proof of Proposition~\ref{prop:lbcyl}, we divide the proof into two cases.\\

  {\bf Case I: $\alpha\geq\tau-1$}. Observe, that in this case $m(\alpha,\tau)=1$ and $M(\alpha,\tau)=0$, thus $\pv_1$ does not play any role.

  Let $\mu=\prod_{\ell=1}^{\infty}\eta_{\ell},$ be a measure supported on $\Gamma'$, where $\eta_{\ell}$ a probabilty measure on $Q$ such that
$$
\eta_{\ell}=\left\{
         \begin{array}{ll}
           \delta_{(\jj_{n_{2k-1}})_{\ell-n_{2k-1}+1}}, & \hbox{if }\ell\in(n_{2k-1},n_{2k-1}+\frac{f(n_{2k-1})}{\tau}], \\
           (\pv')_{\jj_{n_{2k-1}}}^{\ell-n_{2k-1}-\frac{f(n_{2k-1})}{\tau}}, & \hbox{if }\ell\in(n_{2k-1}+\frac{f(n_{2k-1})}{\tau},n_{2k-1}+f(n_{2k-1})], \\
           \pv_2, & \hbox{if }\ell\in(n_{2k-1}+f(n_{2k-1}),\tau n_{2k-1}+f(n_{2k-1})], \\
           \pv_+, & \hbox{if }\ell\in(\tau n_{2k-1}+f(n_{2k-1}),n_{2k}] \\
           \pv^{(m)}, & \hbox{if }\ell\in(\tau^{m}n_{2k},\tau^{m+1}n_{2k}], \\
           \pv_- & \hbox{if }\ell\in(\tau^{q}n_{2k},n_{2k+1}],
         \end{array}
       \right.
$$
where $m=0,\dots,q$, $\pv^{(m)}=\left(1-\frac{m}{q}\right)\pv_++\frac{m}{q}\pv_-$ and
$$
(\pv')_{\jj}^{k}=\left(\frac{\delta_{a=a_k}}{T_{a_k}}\right)_{(a,b)\in Q},\text{ where }\jj=((a_1,b_1),(a_2,b_2),\dots).
$$
It is easy to see by Lemma~\ref{lem:subseq} and \eqref{eq:nk2} that equations \eqref{eq:in1}, \eqref{eq:inball1}, \eqref{eq:inball2}, \eqref{eq:inball4}, \eqref{eq:inball3} and \eqref{eq:in2} hold for the measure $\mu$. The only the differences occur at the following positions
\begin{eqnarray}
% \nonumber % Remove numbering (before each equation)
\liminf_{k\to\infty}\frac{\log\mu(B_{n_{2k-1}+\frac{f(n_{2k-1})}{\tau}}(\ii))}{-(n_{2k-1}+\frac{f(n_{2k-1})}{\tau})\log N} &=&
\begin{cases}
\dfrac{h(\pv_-)}{(1+\frac{\alpha}{\tau})\log N}, & \hbox{if }\tau(\tau-1)\leq\alpha; \\
\dfrac{\frac{\tau+\alpha}{\tau^2}h(\pv_-)+\frac{\tau^2-\tau-\alpha}{\tau^2}h_r(\pv_-)}{(1+\frac{\alpha}{\tau})\log N}, & \hbox{if }\tau(\tau-1)>\alpha,
\end{cases}\label{eq:dontplay}\\
  \liminf_{k\to\infty}\frac{\log\mu(B_{\tau(n_{2k-1}+f(n_{2k-1}))}(\ii))}{-\tau(n_{2k-1}+f(n_{2k-1}))\log N} &=& \frac {h(\pv_-) + (\tau-1) h_r(\pv_2) + \alpha(\tau -1) h_r(\pv_+) + \alpha (1-1/\tau) H} {\tau (1+\alpha) \log N}\label{eq:justthis2} \\
  \liminf_{k\to\infty}\frac{\log\mu(B_{\tau(\tau n_{2k-1}+f(n_{2k-1}))}(\ii))}{-\tau(\tau n_{2k-1}+f(n_{2k-1}))\log N} &=& \frac {h(\pv_-)  + (\tau-1) h(\pv_2) + (\tau +\alpha)(\tau -1) h_r(\pv_+) + \alpha (1-1/\tau) H} {\tau (\tau+\alpha) \log N},
\end{eqnarray}
For the comfortability of the reader, we give the details for \eqref{eq:justthis2}. Similarly to the proof of equation \eqref{eq:justthis}, by Lemma~\ref{lem:subseq} we get
\begin{multline*}
\liminf_{k\to\infty}\frac{\log\mu(B_{\tau(n_{2k-1}+f(n_{2k-1}))}(\ii))}{-\tau(n_{2k-1}+f(n_{2k-1}))\log N}=\\ \frac {h(\pv_-) + (\tau-1) h_r(\pv_2) + \alpha(\tau -1) h_r(\pv_+)} {\tau (1+\alpha) \log N}+\liminf_{k\to\infty}\frac{f(n_{2k-1})(1-1/\tau)\sum_{a\in S}v_a(\jj_{n_{2k-1}}|_{f(n_{2k-1})/\tau}^{f(n_{2k-1})})\log T_a}{\tau(n_{2k-1}+f(n_{2k-1}))\log N}.
\end{multline*}
But by the assumption on $\{\jj_k\}$ and \eqref{eq:lingrov},
$$
\lim_{k\to\infty}\frac{f(n_{2k-1})(1-1/\tau)\sum_{a\in S}v_a(\jj_{n_{2k-1}}|_{f(n_{2k-1})/\tau}^{f(n_{2k-1})})\log T_a}{\tau(n_{2k-1}+f(n_{2k-1}))\log N}=\frac{\alpha(1-1/\tau)H}{\tau(1+\alpha)\log N}.
$$

It is easy to see that the right hand side of \eqref{eq:dontplay} is greater than $\frac{h(\pv_-)}{(1+\alpha)\log N}$ in both cases. Thus, one can finish the proof like at the end of the proof of Case I of Proposition~\ref{prop:lbcyl}.\\

{\bf Case II: $\alpha<\tau-1$}. Let $\mu=\prod_{\ell=1}^{\infty}\eta_{\ell},$ be a measure supported on $\Gamma'$, where $\eta_{\ell}$ a probabilty measure on $Q$ such that
$$
\eta_{\ell}=\left\{
         \begin{array}{ll}
           \pv_1, & \hbox{if }\ell\in(\dfrac{n_{2k-1}+f(n_{2k-1})}{\tau},n_{2k-1}], \\
           \delta_{(\jj_{n_{2k-1}})_{\ell-n_{2k-1}+1}}, & \hbox{if }\ell\in(n_{2k-1},n_{2k-1}+\frac{f(n_{2k-1})}{\tau}], \\
           (\pv')_{\jj_{n_{2k-1}}}^{\ell-n_{2k-1}-\frac{f(n_{2k-1})}{\tau}}, & \hbox{if }\ell\in(n_{2k-1}+\frac{f(n_{2k-1})}{\tau},n_{2k-1}+f(n_{2k-1})], \\
           \pv_2, & \hbox{if }\ell\in(n_{2k-1}+f(n_{2k-1}),\tau n_{2k-1}+f(n_{2k-1})], \\
           \pv_+, & \hbox{if }\ell\in(\tau n_{2k-1}+f(n_{2k-1}),n_{2k}] \\
           \pv^{(m)}, & \hbox{if }\ell\in(\tau^{m}n_{2k},\tau^{m+1}n_{2k}], \\
           \pv_- & \hbox{if }\ell\in(\tau^{q}n_{2k},n_{2k+1}],
         \end{array}
       \right.
$$
where $m=0,\dots,1/\varepsilon$, $\pv^{(m)}=\left(1-\frac{m}{q}\right)\pv_++\frac{m}{q}\pv_-$ and
$$
(\pv')_{\jj}^{k}=\left(\frac{\delta_{a=a_k}}{T_{a_k}}\right)_{(a,b)\in Q},\text{ where }\jj=((a_1,b_1),(a_2,b_2),\dots).
$$
By Lemma~\ref{lem:subseq}, one can prove by similar argument that equations \eqref{eq:in11}, \eqref{eq:inball11}, \eqref{eq:nem1}, \eqref{eq:nem2}, \eqref{eq:inball12}, \eqref{eq:inball13} and \eqref{eq:dontplay} hold. Moreover,
\begin{multline}
  \liminf_{k\to\infty}\frac{\log\mu(B_{\tau(n_{2k-1}+f(n_{2k-1}))}(\ii))}{-\tau(n_{2k-1}+f(n_{2k-1}))\log N}=\\ \frac {\frac{1+\alpha}{\tau}h(\pv_-)+\left(1-\frac{1+\alpha}{\tau}\right)h(\pv_1) + (\tau-1) h_r(\pv_2) + \alpha(\tau -1) h_r(\pv_+) + \alpha (1-1/\tau) H} {\tau (1+\alpha) \log N}
\end{multline}
\begin{multline}
  \liminf_{k\to\infty}\frac{\log\mu(B_{\tau(\tau n_{2k-1}+f(n_{2k-1}))}(\ii))}{-\tau(\tau n_{2k-1}+f(n_{2k-1}))\log N}=\\ \frac {\frac{1+\alpha}{\tau}h(\pv_-)+\left(1-\frac{1+\alpha}{\tau}\right)h(\pv_1)  + (\tau-1) h(\pv_2) + (\tau +\alpha)(\tau -1) h_r(\pv_+) + \alpha (1-1/\tau) H} {\tau (\tau+\alpha) \log N}.
\end{multline}

Since the right hand side of \eqref{eq:dontplay} is greater than $\frac{h(\pv_-)}{(1+\alpha)\log N}$ in both cases, by \eqref{eq:toineq1} and \eqref{eq:toineq2} and Lemma~\ref{lem:goodBer}, the proof is complete.
\end{proof}

\section{Upper bounds}\label{sec:ub}

Let $\pvv=(\pv_-,\pv_1,\pv_2,\pv_+)\in\Theta_M^{\alpha,H}$. For the tuple $\pvv\in\Theta_M^{\alpha,H}$, let $A_{\alpha}(\pvv)=\{i\in\{1,\dots,6\}:d_i^{\alpha}(\pvv,H)=\DJ_{\ALPHA}(\pvv,H)\}$. We decompose the proof of the upper bound into several cases according to wether $\alpha\geq\tau-1$ or not and to the maximal elements of $A_{\alpha}(\pvv)$. We note that we will handle the shrinking targets in case of balls and cylinders together.

Let us note here that we slightly abuse the notations in the upcoming lemmas. We denote the maximizing set and the maximizing vectors for shrinking sequence of cylinders and for shrinking sequence of approximate balls in the same way, but in general, they might be strictly different.

\subsection{Case $\alpha\geq\tau-1$} We note that in this case $\pv_1$ does not play any role. Moreover, observe that
\begin{equation}\label{eq:p-=pD}
\dim\pv_->\frac{h(\pv_-)}{(1+\alpha)\log N}=d_2^{\alpha}(\pvv)
\end{equation}
for every $\pv_-$ with $h_r(\pv_-)\in[h_r(\pv_D),h_r(\pv_d)]$. Thus, $\pv_-$ must be equal to $\pv_D$ for $(\pv_-,\pv_1,\pv_2,\pv_+)\in\Theta_M^{\alpha,H}$.

\begin{lemma}\label{lem:max2b}
	Let $\pvv\in\Theta_M^{\alpha,H}$ be such that $\max A_{\alpha}(\pvv)=2$. Then
	$$
	\dim_H\pi\Gamma_C(f,\{\jj_k\})\leq\dim_H\pi\Gamma_B(f,\{\jj_k\})\leq\frac{\log D}{(1+\alpha)\log N}=\DJ_{\ALPHA}(\pvv)
	$$
	for every sequence of $\{\jj_k\}$.
\end{lemma}

\begin{proof}
	Observe that for every sequence of $\{\jj_k\}$, $\Gamma_C(f,\{\jj_k\})\subseteq\Gamma_B(f,\{\jj_k\})$ by definition. It is easy to see that
	$$
	\sigma^{-n}(B_{f(n)}(\jj_n))\subseteq\bigcup_{|\iv|=n}B_{n+f(n)}(\iv\jj_n).
	$$
	Thus, by Lemma~\ref{lem:cov}
	$$
	\mathcal{H}_{N^{-(r+f(r))}}^s(\pi\Gamma_B(f,\{\jj_k\}))\leq \sum_{n=r}^{\infty}D^nN^{-s(n+f(n))}.
	$$
	Then, for arbitrary $s>d_2^{\alpha}(\pvv)=\frac{\log D}{(1+\alpha)\log N}$, we have $\mathcal{H}^s(\pi\Gamma_B(f,\{\jj_k\}))<\infty$, which implies the statement.
\end{proof}

\begin{lemma}\label{lem:max3b}
	Let $\pvv\in\Theta_M^{\alpha,H}$ be such that $\max A_{\alpha}(\pvv)=3$. Then
	$$
	\dim_H\pi\Gamma_C(f,\{\jj_k\})\leq\dim_H\pi\Gamma_B(f,\{\jj_k\})\leq d_3^{\alpha}(\pvv)=\DJ_{\ALPHA}(\pvv)
	$$
	for every sequence of $\{\jj_k\}$.
\end{lemma}

\begin{proof}
	If $\max A_{\alpha}(\pvv)=3$ then we may assume that $\pv_2=\pv_R$. Indeed, if $\pv_2\neq\pv_R$ then we can take $\varepsilon$ sufficiently small and $\pv_2'\in\Upsilon_{\psi}$ such that $h_r(\pv_2)=(1-\varepsilon)h_r(\pv_2)+\varepsilon h_r(\pv_R)$ and $h(\pv_2')=\psi(h_r(\pv_2))$. Let $\pvv'=(\pv_D,\pv_1,\pv_2',\pv_+)\in\Theta$ then $d_3^{\alpha}(\pvv')>d_3^{\alpha}(\pvv)$, $d_2^{\alpha}(\pvv')=d_2^{\alpha}(\pvv)$ and $d_3^{\alpha}(\pvv')<d_j(\pvv')$ for $j=4,\dots,6$. This, implies that either if $d_2^{\alpha}(\pvv)=d_3^{\alpha}(\pvv)$ and $\pv_2\neq\pv_R$ then we can choose another $\pvv'\in\Theta_M^{\alpha,H}$ such that $\max A(\pvv')=2$ (and apply Lemma~\ref{lem:max2b}) or if $d_2^{\alpha}(\pvv)>d_3^{\alpha}(\pvv)$ then $\pv_2=\pv_R$.

So, without loss of generality, we may assume that
$$
\DJ_{\ALPHA}(\pvv)=d_3^{\alpha}(\pvv)=\frac{\log D+(\tau-1)\log R}{(\tau+\alpha)\log N}.
$$
	Trivially,
	$$
	\sigma^{-n}(B_{f(n)}(\jj_n))\subseteq\bigcup_{|\iv|=n,|\jv|=(\tau-1)n}B_{\tau n+f(n)}(\iv\jj_n|_1^{f(n)}\jv).
	$$
	Hence, by Lemma~\ref{lem:cov}
	$$
	\mathcal{H}_{N^{-(\tau r+f(r))}}^s(\pi\Gamma_B(f,\{\jj_n\}))\leq \sum_{n=r}^{\infty}D^nR^{(\tau-1)n}N^{-s(\tau n+f(n))},
	$$
	which implies the statement by taking arbitrary $s>\frac{\log D+(\tau-1)\log R}{(\tau+\alpha)\log N}$.
\end{proof}

\begin{lemma}\label{lem:max4b}
	Let $\pvv\in\Theta_M^{\alpha,H}$ be such that $\max A_{\alpha}(\pvv)=4$. Then
	$$
	\dim_H\pi\Gamma_C(f,\{\jj_k\})\leq d_4^{\alpha}(\pvv,0)=\DJ_{\ALPHA}(\pvv,0)
	$$
	for every sequence $\{\jj_k\}$. Moreover, if $\{\jj_k=((a_1^{(k)},b_1^{(k)}),(a_2^{(k)},b_2^{(k)}),\dots)\}$ is a sequence on $\Sigma$ such that $$\lim_{n\to\infty}\frac{1}{f(n)(1-1/\tau)}\sum_{k=f(n)/\tau}^{f(n)}\log T_{a_k^{(n)}}=H$$ then
	$$
	\dim_H\pi\Gamma_B(f,\{\jj_k\})\leq d_4^{\alpha}(\pvv,H)=\DJ_{\ALPHA}(\pvv,H).
	$$
\end{lemma}

\begin{proof}
	Fix $\pvv=(\pv_-,\pv_1,\pv_2,\pv_+)\in\Theta_M^{\alpha,H}$ such that $\max A_{\alpha}(\pvv)=4$. Similarly to the beginning of Lemma~\ref{lem:max3b}, either we can choose $\pvv\neq\pvv'\in\Theta_M^{\alpha,H}$ such that $\max A(\pvv')=2$ (and apply Lemma~\ref{lem:max2b}) or $\pv_2=\pv_+=\pv_R$. So without loss of generality, we may assume that	
$$
\DJ_{\ALPHA}(\pvv,*)=\frac{\log D+(1+\alpha)(\tau-1)\log R+\alpha(1-1/\tau)*}{\tau(1+\alpha)\log N},
$$
where $*$ is $0$ in case of cylinders or $H$ in case of balls.

	Obviously,
	$$
	\sigma^{-n}(C_{f(n)}(\jj_n))\subseteq\bigcup_{|\iv|=n,|\jv|=(\tau-1)(n+f(n))}B_{\tau(n+f(n))}(\iv\jj_k|_1^{f(n)}\jv).
	$$
	and
	$$
	\sigma^{-n}(B_{f(n)}(\jj_n))\subseteq\bigcup_{\substack{|\iv|=n,|\jv|=(\tau-1)(n+f(n)) \\
			|\hbar|=f(n)(1-1/\tau):\Pi(\hbar)=\Pi(\jj|_{f(n)/\tau}^{f(n)})}}B_{\tau(n+f(n))}(\iv\jj_k|_1^{f(n)/\tau}\hbar\jv).
		$$
	Then by Lemma~\ref{lem:cov}
	$$
	\mathcal{H}_{N^{-\tau(r+f(r))}}^s(\pi\Gamma_C(f,\jj_k))\leq \sum_{n=r}^{\infty}D^nR^{(\tau-1)(n+f(n))}N^{-s\tau(n+f(n))},
	$$
	and by the assumption on the sequence $\{\jj_k\}$ for arbitrary $\varepsilon>0$ taking sufficiently large $r$
	$$
	\mathcal{H}_{N^{-\tau(r+f(r))}}^s(\pi\Gamma_B(f,\jj_k))\leq \sum_{n=r}^{\infty}D^nR^{(\tau-1)(n+f(n))}e^{f(n)(1-1/\tau)(H+\varepsilon)}N^{-s\tau(n+f(n))}.
	$$	
	Since $\varepsilon>0$ was arbitrary, by taking $s>\DJ_{\ALPHA}(\pvv,*)+C\varepsilon$ (with some proper choice of $C$), the statement follows.
\end{proof}

\begin{lemma}\label{lem:max5b}
	Let $\pvv\in\Theta_M^{\alpha,H}$ be such that $\max A_{\alpha}(\pvv)=5$. Then
	$$
	\dim_H\pi\Gamma_C(f,\{\jj_k\})\leq d_5^{\alpha}(\pvv,0)=\DJ_{\ALPHA}(\pvv,0)
	$$
	for every sequence $\{\jj_k\}$. Moreover, if $\{\jj_k=((a_1^{(k)},b_1^{(k)}),(a_2^{(k)},b_2^{(k)}),\dots)\}$ is a sequence on $\Sigma$ such that $$\lim_{n\to\infty}\frac{1}{f(n)(1-1/\tau)}\sum_{k=f(n)/\tau}^{f(n)}\log T_{a_k^{(n)}}=H$$ then
	$$
	\dim_H\pi\Gamma_B(f,\{\jj_k\})\leq d_5^{\alpha}(\pvv,H)=\DJ_{\ALPHA}(\pvv,H).
	$$
\end{lemma}

\begin{proof}
Again, we may assume without loss of generality that $\pv_+=\pv_R$. Indeed, if $\pv_+\neq\pv_R$ then by Lemma~\ref{lem:entropyrelations} one can take $\varepsilon>0$ sufficiently small, $\pv_+'=(1-\varepsilon)\pv_++\varepsilon\pv_R$ and $\pvv'=(\pv_-,\pv_1,\pv_2,\pv_+)$ such that $d_i(\pvv')=d_i(\pvv)$ for $i=2,3,4$; $d_5^{\alpha}(\pvv,*)<d_5^{\alpha}(\pvv',*)<d_6(\pvv')<d_6(\pvv)$. Thus, either there exists $\pvv'\in\Theta_M^{\alpha,H}$ such that $\max A(\pvv')\leq 4$ (and apply Lemma~\ref{lem:max2b}, Lemma~\ref{lem:max3b} or Lemma~\ref{lem:max4b}) or $\pv_+=\pv_R$.
So without loss of generality, we assume that
$$
\DJ_{\ALPHA}(\pvv,*)=d_5^{\alpha}(\pvv,*)=\frac{\log D+(\tau-1)h(\pv_2)+(\tau-1)(\tau+\alpha)\log R+\alpha(1-1/\tau)*}{\tau(\tau+\alpha)\log N},
$$
where $*$ is $0$ in case of cylinders or $H$ in case of balls.

If $d_3^{\alpha}(\pvv)>d_5^{\alpha}(\pvv,*)$, $d_4^{\alpha}(\pvv,*)>d_5^{\alpha}(\pvv,*)$ and $\pv_2\neq\pv_D$ then by Lemma~\ref{lem:entropyrelations}, one could take $\varepsilon>0$ sufficiently small, $\pv_2'\in\Upsilon_{\psi}$ such that $h_r(\pv_2')=(1-\varepsilon)h_r(\pv_2)+\varepsilon h_r(\pv_D)$ and $h(\pv_2')=\psi(h_r(\pv_2'))$ and $\pvv'=(\pv_D,\pv_1,\pv_2',\pv_+)\in\Theta$ such that $h(\pv_2)<h(\pv_2')$, $h_r(\pv_2)>h_r(\pv_2')$, $d_3^{\alpha}(\pvv)>d_3^{\alpha}(\pvv')>d_5^{\alpha}(\pvv',*)>d_5^{\alpha}(\pvv,*)$, $d_4^{\alpha}(\pvv,*)>d_4^{\alpha}(\pvv',*)>d_5^{\alpha}(\pvv',*)>d_5^{\alpha}(\pvv,*)$ and $d_2^{\alpha}(\pvv)=d_2^{\alpha}(\pvv')$. Thus, either we can reduce to the case $\max A(\pvv')\leq2$ or
\begin{equation}\label{eq:useb}
\text{if $\pv_2\neq\pv_D$ then either $d_3^{\alpha}(\pvv)=d_5^{\alpha}(\pvv,*)$ or $d_4^{\alpha}(\pvv,*)=d_5^{\alpha}(\pvv,*)$.}
\end{equation}

Let $\varepsilon>0$ be arbitrary but fixed. Let
\begin{eqnarray}
	& V_{-1,n}=\left\{\iv\in Q^{(\tau-1)n}:h(\iv)> h(\pv_2)\right\}\label{eq:p2dist1}\\
	& V_{0,n}=\left\{\iv\in Q^{(\tau-1)n}:h(\iv)\leq h(\pv_2)\right\}.\label{eq:p2dist2}
	\end{eqnarray}

We give two covers for $\sigma^{-n}(C_{f(n)}(\jj_n))$ and $\sigma^{-n}(B_{f(n)}(\jj_n))$. Since $V_{-1,n}\cup V_{0,n}=Q^{(\tau-1)n}$, we have
	$$
	\sigma^{-n}(C_{f(n)}(\jj_n))\subseteq\left(\bigcup_{|\iv|=n,\jv\in V_{-1,n}}B_{\tau n+f(n)}(\iv\jj_k|_1^{f(n)}\jv)\right)\bigcup\left(\bigcup_{\substack{|\iv|=n,\jv\in V_{0,n}\\|\hbar|=(\tau-1)(\tau n+f(n))}}B_{\tau(\tau n+f(n))}(\iv\jj_k|_1^{f(n)}\jv\hbar)\right),
	$$
and
$$
	\sigma^{-n}(C_{f(n)}(\jj_n))\subseteq\left(\bigcup_{\substack{|\iv|=n,\jv\in V_{-1,n}\\|\hbar|=(\tau-1)f(n)}}B_{\tau n+f(n)}(\iv\jj_k|_1^{f(n)}\jv\hbar)\right)\bigcup\left(\bigcup_{\substack{|\iv|=n,\jv\in V_{0,n}\\|\hbar|=(\tau-1)(\tau n+f(n))}}B_{\tau(\tau n+f(n))}(\iv\jj_k|_1^{f(n)}\jv\hbar)\right).
	$$
Similarly,
	$$
	\sigma^{-n}(B_{f(n)}(\jj_n))\subseteq\left(\bigcup_{|\iv|=n,\jv\in V_{-1,n}}B_{\tau n+f(n)}(\iv\jj_k|_1^{f(n)}\jv)\right)\bigcup\left(\bigcup_{\substack{|\iv|=n,\jv\in V_{0,n},|\hbar|=(\tau-1)(\tau n+f(n))\\|\jv'|=f(n)(1-1/\tau):\Pi(\jv')=\Pi(\jj|_{f(n)/\tau}^{f(n)})}}B_{\tau(\tau n+f(n))}(\iv\jj_k|_1^{f(n)/\tau}\jv'\jv\hbar)\right)
		$$
and
$$
	\sigma^{-n}(B_{f(n)}(\jj_n))\subseteq\left(\bigcup_{\substack{|\iv|=n,\jv\in V_{-1,n}\\|\hbar|=(\tau-1)f(n)}}B_{\tau n+f(n)}(\iv\jj_k|_1^{f(n)}\jv\hbar)\right)\bigcup\left(\bigcup_{\substack{|\iv|=n,\jv\in V_{0,n},|\hbar|=(\tau-1)(\tau n+f(n))\\|\jv'|=f(n)(1-1/\tau):\Pi(\jv')=\Pi(\jj|_{f(n)/\tau}^{f(n)})}}B_{\tau(\tau n+f(n))}(\iv\jj_k|_1^{f(n)/\tau}\jv'\jv\hbar)\right)
		$$
	
Applying Lemma~\ref{lem:techent} and Lemma~\ref{lem:crossbound}, we get for sufficently large $n$ that
	\begin{eqnarray}
	& \sharp\Pi(V_{-1,n})\leq\mathbbm{1}(\pv_2\neq \pv_D)\cdot e^{n(\tau-1)(h_r(\pv_2)+\varepsilon)},\label{eq:boundV-1}\\
	& \sharp V_{0,n}\leq e^{n(\tau-1)(h(\pv_2)+\varepsilon)}.\label{eq:boundV0}
	\end{eqnarray}
Thus, by Lemma~\ref{lem:cov} ,
\begin{multline*}
	\mathcal{H}_{N^{-r}}^s(\pi\Gamma_C(f))\leq
	\sum_{n=r}^{\infty}e^{C\varepsilon n}\left(\mathbbm{1}(\pv_2\neq\pv_D)\cdot D^ne^{(\tau-1)n h_r(\pv_2)-s(\tau n+f(n))\log N}+\right.\\
	\left.D^nR^{(\tau n+f(n))(\tau-1)}e^{(\tau-1)n h(\pv_2)-s\tau(\tau n+f(n))\log N}\right)=\\
\sum_{n=r}^{\infty}e^{C\varepsilon n}\left(\mathbbm{1}(\pv_2\neq\pv_D)N^{-(\tau n+f(n))(d_3^{\alpha}(\pvv)-s)}+N^{-\tau(\tau n+f(n))(d_5^{\alpha}(\pvv,0)-s)}\right)
	\end{multline*}
and
\begin{multline*}
	\mathcal{H}_{N^{-r}}^s(\pi\Gamma_C(f))\leq
	\sum_{n=r}^{\infty}e^{C\varepsilon n}\left(\mathbbm{1}(\pv_2\neq\pv_D)\cdot D^nR^{f(n)(\tau-1)}e^{(\tau-1)n h_r(\pv_2)-s\tau(n+f(n))\log N}+\right.\\
	\left.D^nR^{(\tau n+f(n))(\tau-1)}e^{(\tau-1)n h(\pv_2)-s\tau(\tau n+f(n))\log N}\right)=\\
\sum_{n=r}^{\infty}e^{C\varepsilon n}\left(\mathbbm{1}(\pv_2\neq\pv_D)N^{-\tau(n+f(n))(d_4^{\alpha}(\pvv,0)-s)}+N^{-\tau(\tau n+f(n))(d_5^{\alpha}(\pvv,0)-s)}\right)
	\end{multline*}

Similarly, by the assumption on $\{\jj_k\}$
	$$
	\mathcal{H}_{N^{-r}}^s(\pi\Gamma_B(f,\{\jj_k\}))\leq
	\sum_{n=r}^{\infty}e^{C\varepsilon n}\left(\mathbbm{1}(\pv_2\neq\pv_D)N^{-(\tau n+f(n))(d_3^{\alpha}(\pvv)-s)}+N^{-\tau(\tau n+f(n))(d_5^{\alpha}(\pvv,H)-s)}\right),
	$$
and
$$
\mathcal{H}_{N^{-r}}^s(\pi\Gamma_B(f,\{\jj_k\}))\leq
	\sum_{n=r}^{\infty}e^{C\varepsilon n}\left(\mathbbm{1}(\pv_2\neq\pv_D)N^{-\tau(n+f(n))(d_4^{\alpha}(\pvv,0)-s)}+N^{-\tau(\tau n+f(n))(d_5^{\alpha}(\pvv,H)-s)}\right),
$$
	where $C$ is a constant independent of $n$. Thus, by taking $s>\DJ_{\ALPHA}(\pvv,*)+C'\varepsilon$, we get
$$
\mathcal{H}^s(\pi\Gamma_C(f,\{\jj_k\})),\mathcal{H}^s(\pi\Gamma_B(f,\{\jj_k\}))<\infty,
$$
which implies the assertion by the arbitrariness of $\varepsilon$.
\end{proof}

\begin{lemma}\label{lem:max6b}
  Let $\pvv\in\Theta_M^{\alpha,H}$ be such that $\max A_{\alpha}(\pvv)=6$. Then
  $$
  	\dim_H\pi\Gamma_C(f,\{\jj_k\})\leq\dim\pv_+=\DJ_{\ALPHA}(\pvv,0)
  $$
  for every sequence $\{\jj_k\}$. Moreover, if $\{\jj_k=((a_1^{(k)},b_1^{(k)}),(a_2^{(k)},b_2^{(k)}),\dots)\}$ is a sequence on $\Sigma$ such that $$\lim_{n\to\infty}\frac{1}{f(n)(1-1/\tau)}\sum_{k=f(n)/\tau}^{f(n)}\log T_{a_k^{(n)}}=H$$ then
$$
  	\dim_H\pi\Gamma_B(f,\jj_k)\leq\dim\pv_+=\DJ_{\ALPHA}(\pvv,H).
$$
\end{lemma}

\begin{proof}
By Lemma~\ref{lem:6implies5}, $d_5^{\alpha}(\pvv,*)=d_6(\pvv)$. Moreover, by similar argument at the beginning of Lemma~\ref{lem:max5b},
\begin{equation}\label{eq:useb2}
\text{if $d_4^{\alpha}(\pvv,*)>d_6(\pvv)$ then either $d_3^{\alpha}(\pvv)=d_6(\pvv)$ or $\pv_2=\pv_D$.}
\end{equation}

Let $\varepsilon>0$ be arbitrary and let $q\geq1$ be arbitrary integer but fixed. Let $V_{k,n}$ be as defined in \eqref{eq:p2dist1} and \eqref{eq:p2dist2} for $k=-1,0$. Now, we define a sequence of probability vectors $\pv_+^{(k)}$. Precisely, let
	\begin{equation}\label{eq:forwardmeas}
	\widehat{\pv}_+^{(k)}=\frac{k}{q}\pv_R+\left(1-\frac{k}{q}\right)\pv_+\text{ and let }\pv^{(k)}\text{ such that }h_r(\pv_+^{(k)})=h_r(\widehat{\pv}_+^{(k)})\text{ and }h(\pv_+^{(k)})=\psi(h_r(\widehat{\pv}_+^{(k)})),
\end{equation}
\begin{multline}\label{eq:forward}
V_{k,n}=\left\{\iv\in S^{(\tau^k-1)(\tau n+f(n))}:\ h_r(\iv|_{\tau^{k-1}(\tau n+f(n))}^{\tau^k(\tau n+f(n))})\leq h_r(\pv_+^{(k-1)})\text{ and }\right.\\\left.h_r(\iv|_{\tau^j-1(\tau n+f(n))}^{\tau^j(\tau n+f(n))})> h_r(\pv_+^{(j-1)})\text{ for }1\leq j\leq k-1\right\}.
\end{multline}
Moreover, let
\begin{multline}\label{eq:forward2}
Z_{k,n}=\left\{\iv\in S^{(\tau^k-1)(n+f(n))}:\ h_r(\iv|_{\tau^{k-1}(n+f(n))}^{\tau^k(n+f(n))})\leq h_r(\pv_+^{(k-1)})\text{ and }\right.\\\left.h_r(\iv|_{\tau^j-1(n+f(n))}^{\tau^j(n+f(n))})> h_r(\pv_+^{(j-1)})\text{ for }1\leq j\leq k-1\right\}.
\end{multline}
For a visualization of the covers $V_{k,n}$ and $Z_{k,n}$ see Figure~\ref{fig:forw} and Figure~\ref{fig:forw2}.
\begin{figure}
  \centering
  \includegraphics[width=170mm]{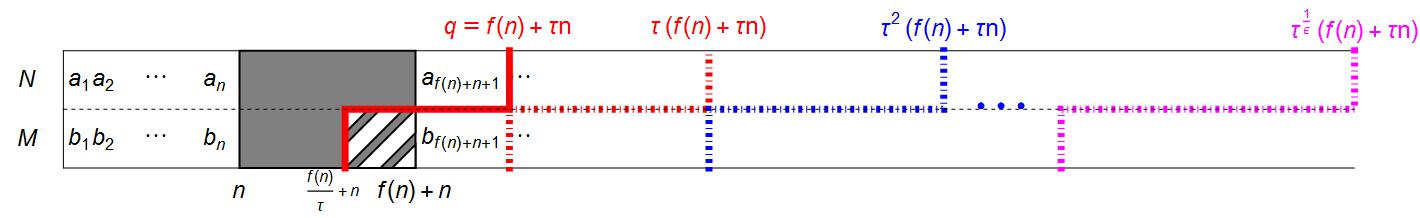}
  \caption{Visualization of the cover $V_{k,n}$ defined in \eqref{eq:forward}.}\label{fig:forw}
  \includegraphics[width=170mm]{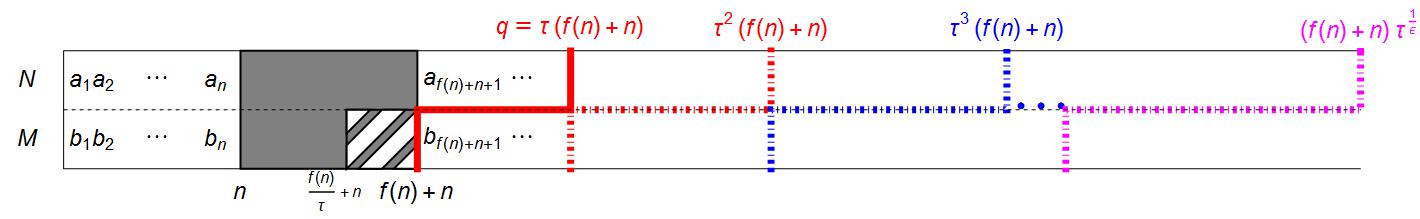}
  \caption{Visualization of the cover $Z_{k,n}$ defined in \eqref{eq:forward2}.}\label{fig:forw2}
\end{figure}

It is easy to see that
$$
	\sigma^{-n}(C_{f(n)}(\jj_n))\subseteq\left(\bigcup_{\substack{|\iv|=n\\\jv\in V_{-1,n}}}B_{\tau n+f(n)}(\iv\jj_k|_1^{f(n)}\jv)\right)\bigcup\left(\bigcup_{k=1}^q\bigcup_{\substack{|\iv|=n,\jv\in V_{0,n}\\\hbar\in V_{k,n}}}B_{\tau^k(\tau n+f(n))}(\iv\jj_k|_1^{f(n)}\jv\hbar)\right),
	$$
and
$$
\sigma^{-n}(C_{f(n)}(\jj_n))\subseteq\bigcup_{k=1}^q\bigcup_{\substack{|\iv|=n,\jv\in Z_{k,n}}}B_{\tau^{k}(n+f(n))}(\iv\jj_k|_1^{f(n)}\jv).
$$
	Similarly
	\begin{multline}
	\sigma^{-n}(B_{f(n)}(\jj_n))\subseteq\left(\bigcup_{|\iv|=n,\jv\in V_{-1,n}}B_{\tau n+f(n)}(\iv\jj_k|_1^{f(n)}\jv)\right)\bigcup\\
\left(\bigcup_{k=1}^q\bigcup_{\substack{|\iv|=n,\jv\in V_{0,n},\hbar\in V_{k,n}\\|\jv'|=f(n)(1-1/\tau):\Pi(\jv')=\Pi(\jj|_{f(n)/\tau}^{f(n)})}}B_{\tau^k(\tau n+f(n))}(\iv\jj_k|_1^{f(n)/\tau}\jv'\jv\hbar)\right),
\end{multline}
and
$$
\sigma^{-n}(B_{f(n)}(\jj_n))\subseteq\bigcup_{k=1}^q\bigcup_{\substack{|\iv|=n,\jv\in Z_{k,n}\\|\hbar|=f(n)(1-1/\tau):\Pi(\hbar)=\Pi(\jj|_{f(n)=\tau}^{f(n)})}}B_{\tau^{k}(n+f(n))}(\iv\jj_k|_1^{f(n)/\tau}\hbar\jv).
$$
By Lemma~\ref{lem:cov},
\begin{multline}\label{eq:this1}
\mathcal{H}_{N^{-r}}^s(\pi\Gamma_C(f,\{\jj_k\}))\leq\sum_{n=r}^{\infty}\left( D^n\cdot\sharp\Pi V_{-1,n}\cdot N^{-s(\tau n+f(n))}+\right.\\\left.
\sum_{k=1}^qD^n\cdot\sharp V_{0,n}\cdot\sharp V_{k,n}|_1^{\tau^{k-1}(\tau n+f(n))}\cdot\sharp\Pi V_{k,n}|_{\tau^{k-1}(\tau n+f(n))}^{\tau^{k}(\tau n+f(n))}\cdot N^{-s\tau^k(\tau n+f(n))}\right),
\end{multline}
and
\begin{multline}\label{eq:this1b}
\mathcal{H}_{N^{-r}}^s(\pi\Gamma_C(f,\{\jj_k\}))\leq\sum_{n=r}^{\infty}\left(\sum_{k=1}^qD^n\cdot\sharp Z_{k,n}|_1^{\tau^{k-1}(n+f(n))}\cdot\sharp\Pi Z_{k,n}|_{\tau^{k-1}(n+f(n))}^{\tau^{k}(n+f(n))}\cdot N^{-s\tau^k(n+f(n))}\right),
\end{multline}
where $$V|_t^{u}=\{\iv\in Q^{u-t}:\iv=\jv|_t^u\text{ for some }\jv\in V\}.$$
By applying Lemma~\ref{lem:techent} and Lemma~\ref{lem:crossbound}, we get that for sufficiently large $n$
\begin{equation}\label{eq:boundVk}
  \sharp V_{k,n}|_1^{\tau^{k-1}(\tau n+f(n))}\cdot\sharp\Pi V_{k,n}|_{\tau^{k-1}(\tau n+f(n))}^{\tau^{k}(\tau n+f(n))}\leq e^{(\tau n+f(n))\left(\sum_{j=1}^{k-1}\tau^{j-1}(\tau-1)(h(\pv_+^{(j-1)})+\varepsilon)+\tau^{k-1}(\tau-1)(h_r(\pv_+^{(k-1)})+\varepsilon)\right)},
\end{equation}
\begin{equation}\label{eq:boundZk}
  \sharp Z_{k,n}|_1^{\tau^{k-1}(n+f(n))}\cdot\sharp\Pi Z_{k,n}|_{\tau^{k-1}(n+f(n))}^{\tau^{k}(n+f(n))}\leq e^{(n+f(n))\left(\sum_{j=1}^{k-1}\tau^{j-1}(\tau-1)(h(\pv_+^{(j-1)})+\varepsilon)+\tau^{k-1}(\tau-1)(h_r(\pv_+^{(k-1)})+\varepsilon)\right)}.
\end{equation}
By the continuity of $h_r$, i.e. $|h_r(\pv_+^{(k)})-h_r(\pv_+^{(k-1)}))|\leq O(1/q)$, simple algebraic manipulations show that
	\begin{multline}\label{eq:sumtodim}
	\sum_{j=1}^{k-1}\tau^{j-1}(\tau-1)h(\pv_+^{(j-1)})+\tau^{k-1}(\tau-1)h_r(\pv_+^{(k-1)})=\\
	\sum_{j=1}^{k-1}\left(\tau^{j-1}(\tau-1)h(\pv_+^{(j-1)}))+\tau^{j-1}(\tau-1)^2(h_r(\pv_+^{(j)})-h_r(\pv_+^{(j-1)}))\right)+(\tau-1)h_r(\pv_+^{(0)})\leq\\
	\sum_{j=1}^{k-1}\tau^{j}(\tau-1)\dim\pv_+^{(j-1)}\log N+(\tau-1)h_r(\pv_+)+O(1/q)(\tau-1)(\tau^{k-1}-1)\leq\\ (\tau^{k}-\tau)\dim\pv_+\log N+(\tau-1)h_r(\pv_+)+O(1/q)(\tau-1)(\tau^{k-1}-1),
	\end{multline}
	where in the last inequality we used that $h_r(\pv_d)\leq h_r(\pv_+)\leq h_r(\pv_+^{(k)})$ and therefore $\dim\pv_+\geq\dim\pv^{(k)}_+$ by the concavity of dimension. Thus,
\begin{multline}
	\mathcal{H}_{N^{-r}}^s(\pi\Gamma_C(f,\{\jj_k\}))\leq\sum_{n=r}^{\infty}\left(\mathbbm{1}(\pv_2\neq\pv_D)\cdot e^{(d_3^{\alpha}(\pvv)+C'\varepsilon-s)(\tau n+f(n))\log N}+\right.\\\left.
e^{\tau(\tau n+f(n))\log N(d_5^{\alpha}(\pvv,0)-d_6(\pvv))}\sum_{k=1}^qe^{(d_6(\pvv)-s+C'(\varepsilon+O(1/q)))\tau^k(\tau n+f(n))\log N}\right)=\\
\sum_{n=r}^{\infty}\left(\mathbbm{1}(\pv_2\neq\pv_D)\cdot e^{(d_3^{\alpha}(\pvv)+C'\varepsilon-s)(\tau n+f(n))\log N}+\sum_{k=1}^qe^{(d_6(\pvv)-s+C'(\varepsilon+O(1/q)))\tau^k(\tau n+f(n))\log N}\right).
	\end{multline}
and
\begin{multline}
	\mathcal{H}_{N^{-r}}^s(\pi\Gamma_C(f,\{\jj_k\}))\leq\sum_{n=r}^{\infty}\left(D^ne^{(\tau-1)(n+f(n))h_r(\pv_+)-\tau d_6(\pvv)\log N}\cdot\sum_{k=1}^qe^{(d_6(\pvv)-s+C'(\varepsilon+O(1/q)))\tau^k(n+f(n))\log N}\right)=\\
\sum_{n=r}^{\infty}\left(N^{\tau(d_4^{\alpha}(\pv_-,\pv_1,\pv_+,\pv_+,0)-d_6(\pvv))}\cdot\sum_{k=1}^qe^{(d_6(\pvv)-s+C'(\varepsilon+O(1/q)))\tau^k(\tau n+f(n))\log N}\right).
	\end{multline}

Similarly,
\begin{multline}
	\mathcal{H}_{N^{-r}}^s(\pi\Gamma_B(f,\{\jj_n\}))\leq\\
\sum_{n=r}^{\infty}\left(\mathbbm{1}(\pv_2\neq\pv_D)\cdot e^{(d_3^{\alpha}(\pvv)+C'\varepsilon-s)(\tau n+f(n))\log N}+\sum_{k=1}^qe^{(d_6(\pvv)-s+C'(\varepsilon+O(1/q)))\tau^k(\tau n+f(n))\log N}\right).
	\end{multline}
and
\begin{multline}
	\mathcal{H}_{N^{-r}}^s(\pi\Gamma_B(f,\{\jj_n\}))\leq\\
\sum_{n=r}^{\infty}\left(N^{\tau(d_4^{\alpha}(\pv_-,\pv_1,\pv_+,\pv_+,H)-d_6(\pvv))(n+f(n))}\cdot\sum_{k=1}^qN^{(d_6(\pvv)-s+C'(\varepsilon+O(1/q)))\tau^k(n+f(n))}\right).
	\end{multline}
By Lemma~\ref{lem:6implies5}, if $d_4^{\alpha}(\pvv)=d_6(\pvv)$ then $\pv_2=\pv_+$ and therefore $d_4^{\alpha}(\pv_-,\pv_1,\pv_+,\pv_+,H)=d_6(\pvv)$. By \eqref{eq:useb}, if $d_4^{\alpha}(\pvv)>d_6(\pvv)$ and $\pv_2\neq\pv_D$ then $d_3^{\alpha}(\pvv)=d_6(\pvv)$. Then by choosing $s>d_6(\pvv)+C'(\varepsilon+O(1/q))$, we get $\mathcal{H}^s(\pi\Gamma_B(f,\{\jj_n\})),\mathcal{H}^s(\pi\Gamma_C(f,\{\jj_k\}))<\infty$. Since $\varepsilon>0$ and $q\geq1$ were arbitrary, the statement follows.
\end{proof}

\subsection{Case $\alpha<\tau-1$}

\begin{lemma}\label{lem:max1}
  Let $\pvv\in\Theta_M^{\alpha,H}$ be such that $\max A_{\alpha}(\pvv)=1$. Then
  $$
  \dim_H\pi\Gamma_C(f,\{\jj_k\})\leq\dim_H\pi\Gamma_B(f,\{\jj_k\})\leq\dim\pv_d=\DJ_{\ALPHA}(\pvv,*)
  $$
  for every sequence of $\{\jj_k\}$.
\end{lemma}

The proof is straightforward.

\begin{lemma}\label{lem:max2}
  Let $\pvv\in\Theta_M^{\alpha,H}$ be such that $\max A_{\alpha}(\pvv)=2$. Then
  $$
  \dim_H\pi\Gamma_C(f,\{\jj_k\})\leq\dim_H\pi\Gamma_B(f,\jj_k)\leq d_2^{\alpha}(\pvv)=\DJ_{\ALPHA}(\pvv,*)
  $$
  for every sequence of $\{\jj_k\}$.
\end{lemma}

\begin{proof}
If $\pv_1\neq\pv_R$ and $\max A_{\alpha}(\pvv)=2$ then for sufficiently small $\varepsilon>0$ one can take $\pv_1'$ and $\pvv'=(\pv_-,\pv_1',\pv_2,\pv_+)$ such that $h_r(\pv_1')=\varepsilon h_r(\pv_R)+(1-\varepsilon)h_r(\pv_1)$ and $h(\pv_1')=\psi(h_r(\pv_1'))$ such that $d_1(\pvv)=d_1(\pvv')$ and $d_2^{\alpha}(\pvv)<d_2^{\alpha}(\pvv')<d_j(\pvv')<d_j(\pvv)$ for all $j>3$. Thus, either there is $\pvv'\in\Theta_M^{\alpha,H}$ such that $\max A(\pvv')\leq1$ (and we apply Lemma~\ref{lem:max1}) or $\pv_1=\pv_R$.

On the other hand, if $\pv_-\neq\pv_D$ and $d_1(\pvv)>d_2^{\alpha}(\pvv)$ then one can take $\pv_-'$ and $\pvv'=(\pv_-',\pv_1,\pv_2,\pv_+)$ such that $h(\pv_-')=(1-\varepsilon)h(\pv_-)+\varepsilon h(\pv_D)$, $h_r(\pv_-')=\varphi(h(\pv_-'))$, $d_1(\pvv)>d_1(\pvv')>d_2^{\alpha}(\pvv')>d_2^{\alpha}(\pvv)$ and $d_2^{\alpha}(\pvv')<d_j(\pvv')$ for $j\geq3$ for sufficiently small $\varepsilon>0$. But this contradicts to the assumption that $\pvv\in\Theta_M^{\alpha,H}$. Thus,
\begin{equation}\label{eq:use}
\text{if $\pv_-\neq\pv_D$ then $d_1(\pvv)=d_2^{\alpha}(\pvv)$.}
\end{equation}

So without loss of generality, we assume
$$
\DJ_{\ALPHA}(\pvv,*)=d_2^{\alpha}(\pvv)=\frac{mh(\pv_-)+(1-m)\log R}{(1+\alpha)\log N},
$$
where $m=(1+\alpha)/\tau$.

\begin{figure}
  \centering
  \includegraphics[width=170mm]{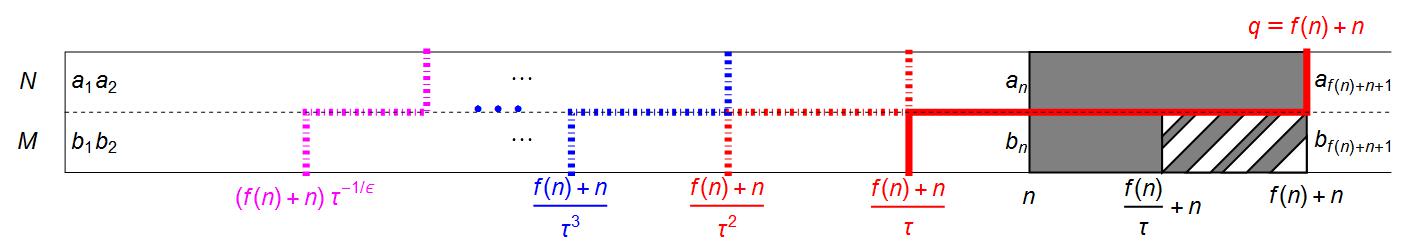}
  \caption{The backward cover $W_{k,n}$ defined in \eqref{eq:backward}.}\label{fig:backw}
\end{figure}

Let $q\geq1$ be arbitrary but fixed integer and let $\varepsilon>0$. We give a cover for $\sigma^{-n}(B_{f(n)}(\jj_n))$. Let us define the following subsets of
\begin{eqnarray}
&\widehat{\pv}_-^{(k)}=\frac{k}{q}\pv_D+\left(1-\frac{k}{q}\right)\pv_-\text{ and let }\pv^{(k)}\text{ such that }h(\pv_-^{(k)})=h(\widehat{\pv}_-^{(k)})\text{ and }h_r(\pv_-^{(k)})=\varphi(h(\widehat{\pv}_-^{(k)})),\label{eq:backwardmeas}\\
&W_{k,n}=\left\{\iv\in Q^n:\ h(\iv|_{(n+f(n))/\tau^{j+1}}^{(n+f(n))/\tau^{j}})\leq h(\pv_-^{(j-1)})\text{ for }q\geq j\geq k+1\text{ and }h(\iv|_{(n+f(n))/\tau^{k+1}}^{(n+f(n))/\tau^{k}})> h(\pv_-^{(k-1)})\right\},\label{eq:backward}\\
&\widetilde{W}_{0,n}=\left\{\iv\in Q^n:\ h(\iv|_{(n+f(n))/\tau^{j+1}}^{(n+f(n))/\tau^{j}})\leq h(\pv_-^{(j-1)})\text{ for all }j=1,\dots,q\right\}.
\end{eqnarray}
For a visualisation of the cover defined in \eqref{eq:backward}, see Figure~\ref{fig:backw}.
It is easy to see that $\widetilde{W}_{0,n}\cup\bigcup_{k=1}^qW_{k,n}=Q^n$. Thus,
$$
\sigma^{-n}(B_{f(n)}(\jj_n))\subseteq\left(\bigcup_{k=1}^{q}\bigcup_{\iv\in W_{k,n}}B_{\frac{n+f(n)}{\tau^{k}}}(\iv\jj_k)\right)\bigcup\bigcup_{\iv\in \widetilde{W}_{0,n}}B_{n+f(n)}(\iv\jj_k).
$$
Let $W'_{k,n}$ be the number of disjoint balls in $\bigcup_{\iv\in W_{k,n}}B_{\frac{n+f(n)}{\tau^{k}}}(\iv\jj_k)$. By Lemma~\ref{lem:techent}, Lemma~\ref{lem:crossbound} and Lemma~\ref{lem:cov}, we have that for $k\geq1$ and for sufficiently large $n$
\begin{eqnarray*}
&W'_{k,n}\leq\mathbbm{1}(\pv_-\neq\pv_D)\cdot D^{(n+f(n))/\tau^{q+1}} e^{\sum_{j=k+1}^q\frac{(n+f(n))(\tau-1)(h(\pv_-^{(j-1)})+\varepsilon)}{\tau^{j+1}}+\frac{(n+f(n))(\tau-1)(h_r(\pv_-^{(k-1)})+\varepsilon)}{\tau^{k+1}}},\\
&W'_{0,n}\leq D^{(n+f(n))/\tau^{q+1}}\cdot e^{\sum_{j=1}^q\frac{(n+f(n))(\tau-1)(h(\pv_-^{(j-1)})+\varepsilon)}{\tau^{j+1}}}\cdot R^{n-\frac{n+f(n)}{\tau}}.
\end{eqnarray*}
By using the continuity of the entropy, we get $|h(\pv_-^{(j)})-h(\pv_-^{(j-1)})|\leq O(1/q)$ for every $j$. Thus,
\begin{multline*}
\sum_{j=k+1}^q\frac{(n+f(n))(\tau-1)(h(\pv_-^{(j-1)})+\varepsilon)}{\tau^{j+1}}\leq\\
\sum_{j=k+1}^q\left(\frac{(n+f(n))(h(\pv_-^{(j-1)})+\varepsilon)}{\tau^{j}}-\frac{(n+f(n))(h(\pv_-^{(j)})+\varepsilon)}{\tau^{j+1}}\right)+nO(1/q)=\\
-\frac{n+f(n)}{\tau^{q+1}}h(\pv_-^{(q)})+\frac{n+f(n)}{\tau^{k+1}}h(\pv_-^{(k)})+(n+f(n))\left(\varepsilon+O(1/q)\right).
\end{multline*}
On the other hand, since $h_r(\pv_D)\leq h_r(\pv_-)\leq h_r(\pv_d)$ we get that $\dim\pv_D\leq\dim\pv_-\leq\dim\pv_d$ and by convexity of the dimension, $\dim\pv^{(k)}\leq\dim\pv_-$ for all $k=0,\dots,q$. Thus,
\begin{equation}\label{eq:boundWk}
W'_{k,n}\leq\mathbbm{1}(\pv_-\neq\pv_D)\cdot N^{\frac{n+f(n)}{\tau^{k}}\dim\pv_-}\cdot e^{2(n+f(n))\left(\varepsilon+O(1/q)\right)}.
\end{equation}
Moreover, we get for $k=0$ similarly
\begin{equation*}
W'_{0,n}\leq e^{\frac{n+f(n)}{\tau}h(\pv_-)+\left(n-\frac{n+f(n)}{\tau}\right)\log R}\cdot e^{2(n+f(n))\left(\varepsilon+O(1/q)\right)}.
\end{equation*}
Thus,
\begin{multline*}
  \mathcal{H}_{N^{-\frac{r+f(r)}{\tau^{q}}}}^s(\pi\Gamma_B(f,\{\jj_k\}))\leq
  \sum_{n=r}^{\infty}\left(\mathbbm{1}(\pv_-\neq\pv_D)\cdot\sum_{k=1}^qN^{\frac{n+f(n)}{\tau^{k}}(\dim\pv_--s)}\cdot e^{2(n+f(n))\left(\varepsilon+O(1/q)\right)}+\right.\\
  \left.e^{\frac{n+f(n)}{\tau}h(\pv_-)+\left(n-\frac{n+f(n)}{\tau}\right)\log R-s(n+f(n))\log N}\cdot e^{2(n+f(n))\left(\varepsilon+O(1/q)\right)}\right)=\\
 \sum_{n=r}^{\infty}e^{2(n+f(n))\left(\varepsilon+O(1/q)\right)}\left(\mathbbm{1}(\pv_-\neq\pv_D)\cdot\sum_{k=1}^qN^{\frac{n+f(n)}{\tau^{k}}(d_1(\pvv)-s)}+N^{-(n+f(n))(d_2^{\alpha}(\pvv)-s)}\right)
\end{multline*}
Hence, by \eqref{eq:use} and by choosing $s>\DJ_{\ALPHA}(\pvv)+2(\varepsilon+O(1/q))$, we get $$\mathcal{H}^s(\pi\Gamma_B(f,\{\jj_k\}))=\lim_{r\to\infty}\mathcal{H}_{N^{-\frac{r+f(r)}{\tau^{q}}}}^s(\pi\Gamma_B(f,\{\jj_k\}))<\infty.$$ Since $\varepsilon>0$ and $q\geq1$ were arbitrary, the statement of the lemma follows.
\end{proof}

\begin{lemma}\label{lem:max3}
  Let $\pvv\in\Theta_M^{\alpha,H}$ be such that $\max A_{\alpha}(\pvv)=3$. Then
  $$
  \dim_H\pi\Gamma_C(f,\{\jj_k\})\leq\dim_H\pi\Gamma_B(f,\{\jj_k\})\leq d_3^{\alpha}(\pvv)=\DJ_{\ALPHA}(\pvv,*)
  $$
  for every sequence of $\{\jj_k\}$.
\end{lemma}

\begin{proof}
By similar argument as at the beginning of Lemma~\ref{lem:max3b}, one can show that either $\pv_2=\pv_R$ or there exists $\pvv'\in\Theta_M^{\alpha,H}$ such that $\max A(\pvv')\leq 2$ (and apply Lemma~\ref{lem:max1} and Lemma~\ref{lem:max2}). Thus, without loss of generality, we assume that
$$
d_3^{\alpha}(\pvv)=\frac{mh(\pv_-)+(1-m)h(\pv_1)+(\tau-1)\log R}{(\tau+\alpha)\log N}.
$$

On the other hand, similarly to the beginning of Lemma~\ref{lem:max2},
\begin{equation}\label{eq:use2}
\text{if $\pv_1\neq\pv_D$ then $d_2^{\alpha}(\pvv)=d_3^{\alpha}(\pvv)$ and if $\pv_-\neq\pv_D$ then $d_1(\pvv)=d_3^{\alpha}(\pvv)$.}
\end{equation}

Let $\varepsilon>0$ and $q\geq1$ be arbitrary, and let $W_{k,n}$ be as defined in \eqref{eq:backward} for $k=1,\dots,q$. Moreover, let
\begin{equation}\label{eq:p1dist1}
W_{-1,n}=\left\{\iv\in Q^n:h(\iv|_{(n+f(n))/\tau}^{n})\leq h(\pv_1)\text{ and }h(\iv|_{(n+f(n))/\tau^{j+1}}^{(n+f(n))/\tau^{j}})\leq h(\pv_-^{(j-1)})\text{ for all }j=1,\dots,q\right\},
\end{equation}
\begin{equation}\label{eq:p1dist2}
W_{0,n}=\left\{\iv\in Q^n:h(\iv|_{(n+f(n))/\tau}^{n})> h(\pv_1)\text{ and }h(\iv|_{(n+f(n))/\tau^{j+1}}^{(n+f(n))/\tau^{j}})\leq h(\pv_-^{(j-1)})\text{ for all }j=1,\dots,q\right\}.
\end{equation}
Again, $\bigcup_{k=-1}^qW_{k,n}=Q^n$ and thus
\begin{multline*}
\sigma^{-n}(B_{f(n)}(\jj_n))\subseteq\left(\bigcup_{k=1}^{q}\bigcup_{\iv\in W_{k,n}}B_{\frac{n+f(n)}{\tau^{k}}}(\iv\jj_n)\right)\bigcup\\
\left(\bigcup_{\iv\in W_{0,n}}B_{n+f(n)}(\iv\jj_n)\right)\bigcup\left(\bigcup_{\iv\in W_{-1,n},|\jv|=(\tau-1)n}B_{\tau n+f(n)}(\iv\jj_n|_1^{f(n)}\jv)\right).
\end{multline*}
Denote by $W'_{k,n}$ the number of disjoint balls in $\bigcup_{\iv\in W_{k,n}}B_{\frac{n+f(n)}{\tau^{k}}}(\iv\jj_k)$ and denote $W'_{-1,n}$ the number of disjoint balls in $\bigcup_{\iv\in W_{-1,n},|\jv|=(\tau-1)n}B_{\tau n+f(n)}(\iv\jj_k|_1^{f(n)}\jv)$. By Lemma~\ref{lem:techent}, Lemma~\ref{lem:crossbound} and Lemma~\ref{lem:cov}, we have that for $k\geq1$ and for sufficiently large $n$
\begin{equation}\label{eq:boundW0}
W'_{0,n}\leq \mathbbm{1}(\pv_1\neq\pv_D)\cdot e^{\frac{n+f(n)}{\tau}h(\pv_-)+\left(n-\frac{n+f(n)}{\tau}\right)h_r(\pv_1)}\cdot e^{Cn\left(\varepsilon+O(1/q)\right)}.
\end{equation}
and
\begin{equation}\label{eq:boundW-1}
W'_{-1,n}\leq e^{\frac{n+f(n)}{\tau}h(\pv_-)+\left(n-\frac{n+f(n)}{\tau}\right)h(\pv_1)}\cdot e^{Cn\left(\varepsilon+O(1/q)\right)},
\end{equation}
where $C>0$ is a constant independent of $n,\varepsilon,q$. Then by using the last two inequalities and \eqref{eq:boundWk}, we have
\begin{multline*}
  \mathcal{H}_{N^{-\frac{r+f(r)}{\tau^{q}}}}^s(\pi\Gamma_B(f,\jj_k))\leq
  \sum_{n=r}^{\infty}\left(e^{C n\left(\varepsilon+O(1/q)\right)}\left(\mathbbm{1}(\pv_-\neq\pv_D)\cdot\sum_{k=1}^qN^{\frac{n+f(n)}{\tau^{k}}(\dim\pv_--s)}+\right.\right.\\
  \mathbbm{1}(\pv_1\neq\pv_D)\cdot e^{\frac{n+f(n)}{\tau}h(\pv_-)+\left(n-\frac{n+f(n)}{\tau}\right)h_r(\pv_1)-s(n+f(n))\log N}+\\
  \left.\left.e^{\frac{n+f(n)}{\tau}h(\pv_-)+\left(n-\frac{n+f(n)}{\tau}\right)h(\pv_1)+n(\tau-1)\log R-s(\tau n+f(n))\log N}\right)\right).
\end{multline*}

By using this and \eqref{eq:use2}, for any $s>\DJ_{\ALPHA}(\pv,*)+C'(\varepsilon+O(1/q))$
$$
\mathcal{H}^s(\pi\Gamma_B(f,\{\jj_k\}))<\infty.
$$
The lemma follows by the fact that $\varepsilon,q$ were arbitrary.
\end{proof}

\begin{lemma}\label{lem:max4}
  Let $\pvv\in\Theta_M^{\alpha,H}$ be such that $\max A_{\alpha}(\pvv)=4$. Then
  $$
    \dim_H\pi\Gamma_C(f,\{\jj_k\}) \leq d_4^{\alpha}(\pvv,0)=\DJ_{\ALPHA}(\pvv,0)
  $$
  for every sequence $\{\jj_k\}$. Moreover, if $\{\jj_k=((a_1^{(k)},b_1^{(k)}),(a_2^{(k)},b_2^{(k)}),\dots)\}$ is a sequence on $\Sigma$ such that $$\lim_{n\to\infty}\frac{1}{f(n)(1-1/\tau)}\sum_{k=f(n)/\tau}^{f(n)}\log T_{a_k^{(n)}}=H$$ then
  $$
      \dim_H\pi\Gamma_B(f,\jj_k)\leq d_4^{\alpha}(\pvv,H)=\DJ_{\ALPHA}(\pvv,H).
  $$
\end{lemma}

\begin{proof}
	By the same argument as in Lemma~\ref{lem:max4b}, we may assume that $\pv_2=\pv_+=\pv_R$ and thus,
$$
d_4^{\alpha}(\pvv,*)=\frac{mh(\pv_-)+(1-m)h(\pv_1)+(1+\alpha)(\tau-1)\log R+\alpha(1-1/\tau)*}{\tau(1+\alpha)\log N}.
$$
Moreover, similarly to Lemma~\ref{lem:max3}
\begin{equation}\label{eq:use3}
\text{$\pv_-\neq\pv_D$ implies $d_1(\pvv)=d_4^{\alpha}(\pvv,*)$ and $\pv_1\neq\pv_D$ implies $d_2^{\alpha}(\pvv)=d_4^{\alpha}(\pvv,*)$.}
\end{equation}
	
	Let $\varepsilon>0$ and $q\geq1$ be arbitrary, and let $W_{k,n}$ be as defined in \eqref{eq:backward}, \eqref{eq:p1dist1} and \eqref{eq:p1dist2} for $k=-1,\dots,q$. Thus, similarly to the previous lemma,
	\begin{multline*}
	\sigma^{-n}(C_{f(n)}(\jj_n))\subseteq\left(\bigcup_{k=1}^{q}\bigcup_{\iv\in W_{k,n}}B_{\frac{n+f(n)}{\tau^{k}}}(\iv\jj_k)\right)\bigcup\\
	\left(\bigcup_{\iv\in W_{0,n}}B_{n+f(n)}(\iv\jj_k)\right)\bigcup\left(\bigcup_{\iv\in W_{-1,n},|\jv|=(\tau-1)(n+f(n))}B_{\tau(n+f(n))}(\iv\jj_n|_1^{f(n)}\jv)\right),
	\end{multline*}
	and
	\begin{multline*}
	\sigma^{-n}(B_{f(n)}(\jj_n))\subseteq\left(\bigcup_{k=1}^{q}\bigcup_{\iv\in W_{k,n}}B_{\frac{n+f(n)}{\tau^{k}}}(\iv\jj_k)\right)\bigcup\\
	\left(\bigcup_{\iv\in W_{0,n}}B_{n+f(n)}(\iv\jj_k)\right)\bigcup\left(\bigcup_{\substack{\iv\in W_{-1,n}, |\jv'|=(\tau-1)(n+f(n))\\ |\jv|=f(n)(1-1/\tau),\Pi(\jj_n|_{f(n)/\tau}^n)=\Pi(\jv)}}B_{\tau(n+f(n))}(\iv\jj_n|_1^{f(n)/\tau}\jv\jv')\right).
	\end{multline*}
	Thus, by Lemma~\ref{lem:cov}, \eqref{eq:boundWk}, \eqref{eq:boundW0}, \eqref{eq:boundW-1} and choosing $n$ sufficiently large,
	\begin{multline*}
	\mathcal{H}_{N^{-\frac{r+f(r)}{\tau^{q}}}}^s(\pi\Gamma_C(f,\jj_k))\leq
	\sum_{n=r}^{\infty}\left(e^{C n\left(\varepsilon+O(1/q)\right)}\left(\mathbbm{1}(\pv_-\neq\pv_D)\cdot\sum_{k=1}^qN^{\frac{n+f(n)}{\tau^{k}}(d_1(\pvv)-s)}+\right.\right.\\
	\left.\left.\mathbbm{1}(\pv_1\neq\pv_D)\cdot N^{(n+f(n))(d_2^{\alpha}(\pvv)-s)}+N^{(d_4^{\alpha}(\pvv,0)-s)\tau(n+f(n))}\right)\right).
	\end{multline*}
	and
		\begin{multline*}
		\mathcal{H}_{N^{-\frac{r+f(r)}{\tau^{q}}}}^s(\pi\Gamma_B(f,\jj_k))\leq
		\sum_{n=r}^{\infty}\left(e^{C n\left(\varepsilon+O(1/q)\right)}\left(\mathbbm{1}(\pv_-\neq\pv_D)\cdot\sum_{k=1}^qN^{\frac{n+f(n)}{\tau^{k}}(d_1(\pvv)-s)}+\right.\right.\\
	\left.\left.\mathbbm{1}(\pv_1\neq\pv_D)\cdot N^{(n+f(n))(d_2^{\alpha}(\pvv)-s)}+N^{(d_4^{\alpha}(\pvv,H)-s)\tau(n+f(n))}\right)\right).
		\end{multline*}
		 Hence, by \eqref{eq:use3} for any $s>\DJ_{\ALPHA}(\pvv)+C'(\varepsilon+O(1/q))$
		$$
		\mathcal{H}^s(\pi\Gamma_B(f,\{\jj_k\})),\mathcal{H}^s(\pi\Gamma_C(f))<\infty\text{ almost surely.}
		$$
		Again, since $\varepsilon,q$ were arbitrary, the statement follows.
\end{proof}

\begin{lemma}\label{lem:max5}
  Let $\pvv\in\Theta_M^{\alpha,H}$ be such that $\max A_{\alpha}(\pvv)=5$. Then
  $$
    \dim_H\pi\Gamma_C(f,\{\jj_k\})\leq d_5^{\alpha}(\pvv,0)=\DJ_{\ALPHA}(\pvv,0)
    $$
    for every sequence $\{\jj_k\}$. Moreover, if $\{\jj_k=((a_1^{(k)},b_1^{(k)}),(a_2^{(k)},b_2^{(k)}),\dots)\}$ is a sequence on $\Sigma$ such that $$\lim_{n\to\infty}\frac{1}{f(n)(1-1/\tau)}\sum_{k=f(n)/\tau}^{f(n)}\log T_{a_k^{(n)}}=H$$ then
      $$
    \dim_H\pi\Gamma_B(f,\jj_k)\leq d_5^{\alpha}(\pvv,H)=\DJ_{\ALPHA}(\pvv,H).
  $$
\end{lemma}

\begin{proof}
	Similarly to Lemma~\ref{lem:max5b}, we can assume that $\pv_+=\pv_R$ and hence
$$
d_5^{\alpha}(\pvv,*)=\frac{mh(\pv_-)+(1-m)h(\pv_1)+(\tau-1)h(\pv_2)+(\tau+\alpha)(\tau-1)\log R+\alpha(1-1/\tau)*}{\tau(\tau+\alpha)\log N}
$$
and
$$
d_4^{\alpha}(\pvv,*)=\frac{mh(\pv_-)+(1-m)h(\pv_1)+(\tau-1)h_r(\pv_2)+\alpha(\tau-1)\log R+\alpha(1-1/\tau)*}{\tau(1+\alpha)\log N}.
$$
If $\pv_2\neq\pv_D$, $d_3^{\alpha}(\pvv)>d_5^{\alpha}(\pvv,*)$ and $d_4^{\alpha}(\pvv,*)>d_5^{\alpha}(\pvv,*)$ then we can choose $\varepsilon>0$ sufficiently small, $\pv_2'\in\Upsilon_{\psi}$ and $\pvv'=(\pv_-,\pv_1,\pv_2,\pv_+)\in\Theta$ such that $h(\pv_2')=(1-\varepsilon)h(\pv_2)+\varepsilon h(\pv_D)$, $d_3^{\alpha}(\pvv)>d_3^{\alpha}(\pvv')>d_5^{\alpha}(\pvv')>d_5^{\alpha}(\pvv)$, $d_4^{\alpha}(\pvv)>d_4^{\alpha}(\pvv')>d_5^{\alpha}(\pvv')>d_5^{\alpha}(\pvv)$ and $d_i(\pvv)=d_i(\pvv')$ for $i=1,2$. Thus, either there exists $\pvv'\in\Theta_M^{\alpha,H}$ such that $\max A(\pvv')\leq 2$ and we can apply Lemma~\ref{lem:max1} and Lemma~\ref{lem:max2} or
\begin{equation}\label{eq:use4}
  \text{$\pv_2\neq\pv_D$ implies either $d_3^{\alpha}(\pvv)=d_5^{\alpha}(\pvv,*)$ or $d_4^{\alpha}(\pvv,*)=d_5^{\alpha}(\pvv,*)$.}
\end{equation}
Moreover, similarly to Lemma~\ref{lem:max3},
\begin{equation}\label{eq:use4b}
\text{if $\pv_1\neq\pv_D$ then $d_2^{\alpha}(\pvv)=d_5^{\alpha}(\pvv,*)$ and if $\pv_-\neq\pv_D$ then $d_1(\pvv)=d_5^{\alpha}(\pvv,*)$.}
\end{equation}

	Let $\varepsilon>0$ and $q\geq1$ be arbitrary, and let $W_{k,n}$ be as defined in \eqref{eq:backward}, \eqref{eq:p1dist1} and \eqref{eq:p1dist2} for $k=-1,\dots,q$. Moreover, let $V_{-1,n}$ and $V_{0,n}$ be the same as defined in \eqref{eq:p2dist1} and \eqref{eq:p2dist2}.

We will introduce two different covers for the sets $\sigma^{-n}(C_{f(n)}(\jj_n))$, $\sigma^{-n}(B_{f(n)}(\jj_n))$ according to that $d_i=d_5^{\alpha}$ for which $i=3,4$. By using \eqref{eq:backward}, \eqref{eq:p1dist1}, \eqref{eq:p1dist2}, \eqref{eq:p2dist1}, \eqref{eq:p2dist2}, we get the following covers
	\begin{multline}\label{eq:covC15}
	\sigma^{-n}(C_{f(n)}(\jj_n))\subseteq\left(\bigcup_{k=1}^{q}\bigcup_{\iv\in W_{k,n}}B_{\frac{n+f(n)}{\tau^{k}}}(\iv\jj_n)\right)\bigcup\left(\bigcup_{\iv\in W_{0,n}}B_{n+f(n)}(\iv\jj_n)\right)\bigcup\\
	\left(\bigcup_{\iv\in W_{-1,n},\jv\in V_{-1,n}}B_{\tau n+f(n)}(\iv\jj_n|_1^{f(n)}\jv)\right)\bigcup\left(\bigcup_{\substack{\iv\in W_{-1,n},\jv\in V_{0,n} \\ |\jv'|=(\tau-1)(\tau n+f(n))}}B_{\tau(\tau n+f(n)}(\iv\jj_n|_1^{f(n)}\jv\jv')\right)
	\end{multline}
and
	\begin{multline}\label{eq:covC25}
	\sigma^{-n}(C_{f(n)}(\jj_n))\subseteq\left(\bigcup_{k=1}^{q}\bigcup_{\iv\in W_{k,n}}B_{\frac{n+f(n)}{\tau^{k}}}(\iv\jj_n)\right)\bigcup\left(\bigcup_{\iv\in W_{0,n}}B_{n+f(n)}(\iv\jj_n)\right)\bigcup\\
	\left(\bigcup_{\substack{\iv\in W_{-1,n},\jv\in V_{-1,n}\\|\hbar|=(\tau-1)f(n)}}B_{\tau(n+f(n))}(\iv\jj_n|_1^{f(n)}\jv)\right)\bigcup\left(\bigcup_{\substack{\iv\in W_{-1,n},\jv\in V_{0,n} \\ |\jv'|=(\tau-1)(\tau n+f(n))}}B_{\tau(\tau n+f(n))}(\iv\jj_n|_1^{f(n)}\jv\jv')\right)
	\end{multline}

Similarly,
	\begin{multline}\label{eq:covB15}
	\sigma^{-n}(B_{f(n)}(\jj_n))\subseteq\left(\bigcup_{k=1}^{q}\bigcup_{\iv\in W_{k,n}}B_{\frac{n+f(n)}{\tau^{k}}}(\iv\jj_n)\right)\bigcup\left(\bigcup_{\iv\in W_{0,n}}B_{n+f(n)}(\iv\jj_n)\right)\bigcup\\
	\left(\bigcup_{\iv\in W_{-1,n},\jv\in V_{-1,n}}B_{\tau n+f(n)}(\iv\jj_n|_1^{f(n)}\jv)\right)\bigcup\left(\bigcup_{\substack{\iv\in W_{-1,n},\jv'\in V_{0,n} \\ |\jv''|=(\tau-1)(\tau n+f(n)) \\ |\jv|=f(n)(1-1/\tau),\Pi(\jj_n|_{f(n)/\tau}^n)=\Pi(\jv)}}B_{\tau(\tau n+f(n))}(\iv\jj_n|_1^{f(n)/\tau}\jv\jv'\jv'')\right),
	\end{multline}
and
	\begin{multline}\label{eq:covB25}
	\sigma^{-n}(B_{f(n)}(\jj_n))\subseteq\left(\bigcup_{k=1}^{q}\bigcup_{\iv\in W_{k,n}}B_{\frac{n+f(n)}{\tau^{k}}}(\iv\jj_n)\right)\bigcup\left(\bigcup_{\iv\in W_{0,n}}B_{n+f(n)}(\iv\jj_n)\right)\bigcup\\
	\left(\bigcup_{\substack{\iv\in W_{-1,n},\jv\in V_{-1,n}\\|\hbar|=(\tau-1)f(n)}}B_{\tau(n+f(n))}(\iv\jj_n|_1^{f(n)}\jv\hbar)\right)\bigcup\left(\bigcup_{\substack{\iv\in W_{-1,n},\jv'\in V_{0,n} \\ |\jv''|=(\tau-1)(\tau n+f(n)) \\ |\jv|=f(n)(1-1/\tau),\Pi(\jj_n|_{f(n)/\tau}^n)=\Pi(\jv)}}B_{\tau(\tau n+f(n))}(\iv\jj_n|_1^{f(n)/\tau}\jv\jv'\jv'')\right).
	\end{multline}

	By using Lemma~\ref{lem:cov}, \eqref{eq:boundWk}, \eqref{eq:boundW0}, \eqref{eq:boundW-1}, \eqref{eq:boundV-1}, \eqref{eq:boundV0} and choosing $r$ sufficiently large, we get that cover \eqref{eq:covC15} implies that
	\begin{multline*}
	\mathcal{H}_{N^{-\frac{r+f(r)}{\tau^{q}}}}^s(\pi\Gamma_C(f,\{\jj_k\}))\leq
	\sum_{n=r}^{\infty}\left(e^{C n\left(\varepsilon+O(1/q)\right)}\left(\mathbbm{1}(\pv_-\neq\pv_D)\cdot\sum_{k=1}^qN^{\frac{n+f(n)}{\tau^{k}}(d_1(\pvv)-s)}+\right.\right.\\
		\left.\left.\mathbbm{1}(\pv_1\neq\pv_D)\cdot N^{(d_2^{\alpha}(\pvv)-s)(n+f(n))}+\mathbbm{1}(\pv_2\neq\pv_D)\cdot N^{(d_3^{\alpha}(\pvv)-s)(\tau n+f(n))}+N^{(d_5^{\alpha}(\pvv,0)-s)\tau(\tau n+f(n))}
	\right)\right).
	\end{multline*}
	and cover \eqref{eq:covC25} implies that
\begin{multline*}
	\mathcal{H}_{N^{-\frac{r+f(r)}{\tau^{q}}}}^s(\pi\Gamma_C(f,\{\jj_k\}))\leq
	\sum_{n=r}^{\infty}\left(e^{C n\left(\varepsilon+O(1/q)\right)}\left(\mathbbm{1}(\pv_-\neq\pv_D)\cdot\sum_{k=1}^qN^{\frac{n+f(n)}{\tau^{k}}(d_1(\pvv)-s)}+\right.\right.\\
		\left.\left.\mathbbm{1}(\pv_1\neq\pv_D)\cdot N^{(d_2^{\alpha}(\pvv)-s)(n+f(n))}+\mathbbm{1}(\pv_2\neq\pv_D)\cdot N^{(d_4^{\alpha}(\pvv,0)-s)\tau(n+f(n))}+N^{(d_5^{\alpha}(\pvv,0)-s)\tau(\tau n+f(n))}
	\right)\right).
	\end{multline*}

Similarly, \eqref{eq:covB15} implies that
	\begin{multline*}
	\mathcal{H}_{N^{-\frac{r+f(r)}{\tau^{q}}}}^s(\pi\Gamma_B(f,\{\jj_k\}))\leq
	\sum_{n=r}^{\infty}\left(e^{C n\left(\varepsilon+O(1/q)\right)}\left(\mathbbm{1}(\pv_-\neq\pv_D)\cdot\sum_{k=1}^qN^{\frac{n+f(n)}{\tau^{k}}(d_1(\pvv)-s)}+\right.\right.\\
		\left.\left.\mathbbm{1}(\pv_1\neq\pv_D)\cdot N^{(d_2^{\alpha}(\pvv)-s)(n+f(n))}+\mathbbm{1}(\pv_2\neq\pv_D)\cdot N^{(d_3^{\alpha}(\pvv)-s)(\tau n+f(n))}+N^{(d_5^{\alpha}(\pvv,H)-s)\tau(\tau n+f(n))}
\right)\right),
	\end{multline*}
and \eqref{eq:covB25} implies that
\begin{multline*}
	\mathcal{H}_{N^{-\frac{r+f(r)}{\tau^{q}}}}^s(\pi\Gamma_B(f,\{\jj_k\}))\leq
	\sum_{n=r}^{\infty}\left(e^{C n\left(\varepsilon+O(1/q)\right)}\left(\mathbbm{1}(\pv_-\neq\pv_D)\cdot\sum_{k=1}^qN^{\frac{n+f(n)}{\tau^{k}}(d_1(\pvv)-s)}+\right.\right.\\
		\left.\left.\mathbbm{1}(\pv_1\neq\pv_D)\cdot N^{(d_2^{\alpha}(\pvv)-s)(n+f(n))}+\mathbbm{1}(\pv_2\neq\pv_D)\cdot N^{(d_4^{\alpha}(\pvv,H)-s)\tau(n+f(n))}+N^{(d_5^{\alpha}(\pvv,H)-s)\tau(\tau n+f(n))}
	\right)\right).
	\end{multline*}
By \eqref{eq:use4} and \eqref{eq:use4b}, one can finish the proof analogously to the end of the proof of Lemma~\ref{lem:max4}.
\end{proof}

\begin{lemma}\label{lem:max6a}
  Let $\pvv\in\Theta_M^{\alpha,H}$ be such that $\max A_{\alpha}(\pvv)=6$. Then
  $$
    \dim_H\pi\Gamma_C(f,\{\jj_k\})\leq d_6(\pvv)=\DJ_{\ALPHA}(\pvv,0)
    $$
    for every sequence $\{\jj_k\}$. Moreover, if $\{\jj_k=((a_1^{(k)},b_1^{(k)}),(a_2^{(k)},b_2^{(k)}),\dots)\}$ is a sequence on $\Sigma$ such that $$\lim_{n\to\infty}\frac{1}{f(n)(1-1/\tau)}\sum_{k=f(n)/\tau}^{f(n)}\log T_{a_k^{(n)}}=H$$ then
      $$
    \dim_H\pi\Gamma_B(f,\jj_k)\leq d_6(\pvv)=\DJ_{\ALPHA}(\pvv,H).
  $$
\end{lemma}

\begin{proof}
	Let $\varepsilon>0$ and $q\geq1$ be arbitrary, and let $W_{k,n}$ be as defined in \eqref{eq:backward}, \eqref{eq:p1dist1} and \eqref{eq:p1dist2} for $k=-1,\dots,q$, let $V_{k,n}$ be as defined in \eqref{eq:p2dist1}, \eqref{eq:p2dist2} and \eqref{eq:forward} for $k=-1,\dots,q$, and let $Z_{k,n}$ be as defined in \eqref{eq:boundZk} for $k=1,\dots,q$.

	By Lemma~\ref{lem:6implies5}, $d_5^{\alpha}(\pvv,*)=d_6(\pvv)$. Moreover, by similar arguments like in the beginning of the proof of Lemma~\ref{lem:max6b} and of Lemma~\ref{lem:max3}
\begin{eqnarray}
  &&\text{if $\pv_1\neq\pv_D$ then $d_2^{\alpha}(\pvv)=d_6(\pvv)$ and if $\pv_-\neq\pv_D$ then $d_1(\pvv)=d_6(\pvv)$,}\label{eq:usesok}\\
  &&\text{if $d_4^{\alpha}(\pvv,*)>d_6(\pvv)$ then either $\pv_2=\pv_D$ or $d_3^{\alpha}(\pvv)=d_6(\pvv)$.}\label{eq:usesok2}
\end{eqnarray}

	Now, if $d_3^{\alpha}(\pvv)>d_5^{\alpha}(\pvv)=d_6(\pvv)$ and $\pv_2\neq\pv_D$ then by taking $\pv_2'$ (and $\pvv'=(\pv_-,\pv_1,\pv_2',\pv_+)$) such that $\pv_2'$ is sufficiently close to $\pv_2$, $h(\pv_2')>h(\pv_2)$, $h_r(\pv_2)>h_r(\pv_2')$ and $d_3^{\alpha}(\pvv')>d_5^{\alpha}(\pvv')>d_6(\pvv')=d_6(\pvv)$. But this contradicts to Lemma~\ref{lem:6implies5}, thus either $d_3^{\alpha}(\pvv)=d_6(\pvv)$ or $\pv_2=\pv_D$. Similarly, either $d_2^{\alpha}(\pvv)=d_6(\pvv)$ or $\pv_1=\pv_D$ and either $d_1(\pvv)=d_6(\pvv)$ or $\pv_-=\pv_D$ hold.

	By using \eqref{eq:backward}, \eqref{eq:p1dist1}, \eqref{eq:p1dist2}, \eqref{eq:p2dist1}, \eqref{eq:p2dist2}, \eqref{eq:forward} and \eqref{eq:forward2}, we get the following covers for $\sigma^{-n}(C_{f(n)}(\jj_n))$
	\begin{multline*}
	\sigma^{-n}(C_{f(n)}(\jj_n))\subseteq\left(\bigcup_{k=1}^{q}\bigcup_{\iv\in W_{k,n}}B_{\frac{n+f(n)}{\tau^{k}}}(\iv\jj_n)\right)\bigcup\left(\bigcup_{\iv\in W_{0,n}}B_{n+f(n)}(\iv\jj_n)\right)\bigcup\\
	\left(\bigcup_{\iv\in W_{-1,n},\jv\in V_{-1,n}}B_{\tau n+f(n)}(\iv\jj_n|_1^{f(n)}\jv)\right)\bigcup\left(\bigcup_{\ell=1}^{q+1}\bigcup_{\substack{\iv\in W_{-1,n},\jv\in V_{0,n} \\ \hbar\in V_{\ell,n}}}B_{\tau^{\ell}(\tau n+f(n))}(\iv\jj_n|_1^{f(n)}\jv\hbar)\right).
	\end{multline*}
	and
\begin{multline*}
	\sigma^{-n}(C_{f(n)}(\jj_n))\subseteq\left(\bigcup_{k=1}^{q}\bigcup_{\iv\in W_{k,n}}B_{\frac{n+f(n)}{\tau^{k}}}(\iv\jj_n)\right)\bigcup\left(\bigcup_{\iv\in W_{0,n}}B_{n+f(n)}(\iv\jj_n)\right)\bigcup\\
	\left(\bigcup_{\ell=1}^{q+1}\bigcup_{\substack{\iv\in W_{-1,n}\\\jv\in Z_{k,n} }}B_{\tau^{\ell}(n+f(n))}(\iv\jj_n|_1^{f(n)}\jv)\right).
	\end{multline*}
Analogously,
	\begin{multline*}
	\sigma^{-n}(B_{f(n)}(\jj_n))\subseteq\left(\bigcup_{k=1}^{q}\bigcup_{\iv\in W_{k,n}}B_{\frac{n+f(n)}{\tau^{k}}}(\iv\jj_n)\right)\bigcup\left(\bigcup_{\iv\in W_{0,n}}B_{n+f(n)}(\iv\jj_n)\right)\bigcup\\
	\left(\bigcup_{\iv\in W_{-1,n},\jv\in V_{-1,n}}B_{\tau n+f(n)}(\iv\jj_n|_1^{f(n)}\jv)\right)\bigcup\left(\bigcup_{\ell=1}^{q+1}\bigcup_{\substack{\iv\in W_{-1,n},\jv'\in V_{0,n}, \hbar\in V_{\ell,n} \\ |\jv|=f(n)(1-1/\tau),\Pi(\jj_n|_{f(n)/\tau}^n)=\Pi(\jv)}}B_{\tau^{\ell}(\tau n+f(n))}(\iv\jj_n|_1^{f(n)/\tau}\jv\jv'\hbar)\right),
	\end{multline*}
and
	\begin{multline*}
	\sigma^{-n}(B_{f(n)}(\jj_n))\subseteq\left(\bigcup_{k=1}^{q}\bigcup_{\iv\in W_{k,n}}B_{\frac{n+f(n)}{\tau^{k}}}(\iv\jj_n)\right)\bigcup\left(\bigcup_{\iv\in W_{0,n}}B_{n+f(n)}(\iv\jj_n)\right)\bigcup\\
	\left(\bigcup_{\ell=1}^{q+1}\bigcup_{\substack{\iv\in W_{-1,n}, \hbar\in Z_{\ell,n} \\ |\jv|=f(n)(1-1/\tau),\Pi(\jj_n|_{f(n)/\tau}^n)=\Pi(\jv)}}B_{\tau^{\ell}(n+f(n))}(\iv\jj_n|_1^{f(n)/\tau}\jv\hbar)\right),
	\end{multline*}

	By \eqref{eq:boundWk}, \eqref{eq:boundW0}, \eqref{eq:boundW-1}, \eqref{eq:boundV-1}, \eqref{eq:boundV0}, \eqref{eq:boundVk}, \eqref{eq:sumtodim} and choosing $n$ sufficiently large,
	\begin{multline*}
	\mathcal{H}_{N^{-\frac{r+f(r)}{\tau^{q}}}}^s(\pi\Gamma_C(f,\jj_k))\leq
	\sum_{n=r}^{\infty}\left(e^{C n\left(\varepsilon+O(1/q)\right)}\left(\mathbbm{1}(\pv_-\neq\pv_D)\cdot\sum_{k=1}^qN^{\frac{n+f(n)}{\tau^{k}}(\dim\pv_--s)}+\right.\right.\\
	\mathbbm{1}(\pv_1\neq\pv_D)\cdot e^{\frac{n+f(n)}{\tau}h(\pv_-)+\left(n-\frac{n+f(n)}{\tau}\right)h_r(\pv_1)-s(n+f(n))\log N}+\\
	\left.\left.\mathbbm{1}(\pv_2\neq\pv_D)\cdot e^{\frac{n+f(n)}{\tau}h(\pv_-)+\left(n-\frac{n+f(n)}{\tau}\right)h(\pv_1)+(\tau-1)n h_r(\pv_2)-s(\tau n+f(n))\log N}\right)\right)+\\
	e^{\frac{n+f(n)}{\tau}h(\pv_-)+\left(n-\frac{n+f(n)}{\tau}\right)h(\pv_1)+(\tau-1)n h(\pv_2)+(\tau n+f(n))(\tau-1)h_r(\pv_+)}\cdot\\
	\sum_{\ell=1}^{q}e^{(\tau n+f(n))\left((\tau^{\ell}-\tau)\dim\pv_+\log N+O(1/q)(\tau-1)(\tau^{\ell-1}-1)+(\tau^{\ell}-1)\varepsilon-s\tau^{\ell}\log N\right)}\leq\\
\sum_{n=r}^{\infty}\left(e^{C n\left(\varepsilon+O(1/q)\right)}\left(\mathbbm{1}(\pv_-\neq\pv_D)\cdot\sum_{k=1}^qN^{\frac{n+f(n)}{\tau^{k}}(d_1(\pvv)-s)}+\right.\right.\\
		\mathbbm{1}(\pv_1\neq\pv_D)\cdot N^{(d_2^{\alpha}(\pvv)-s)(n+f(n))}+\mathbbm{1}(\pv_2\neq\pv_D)\cdot N^{(d_3^{\alpha}(\pvv)-s)(\tau n+f(n))}+\\
	\left.\left.N^{(d_5^{\alpha}(\pvv,0)-d_6(\pvv))\tau(\tau n+f(n))}\cdot\sum_{\ell=1}^{q}N^{\tau^{\ell}(\tau n+f(n))(d_6(\pvv)-s)}\right)\right),
	\end{multline*}
and
\begin{multline*}
\mathcal{H}_{N^{-\frac{r+f(r)}{\tau^{q}}}}^s(\pi\Gamma_C(f,\jj_k))\leq\sum_{n=r}^{\infty}\left(e^{C n\left(\varepsilon+O(1/q)\right)}\left(\mathbbm{1}(\pv_-\neq\pv_D)\cdot\sum_{k=1}^qN^{\frac{n+f(n)}{\tau^{k}}(d_1(\pvv)-s)}+\right.\right.\\
\left.\left.\mathbbm{1}(\pv_1\neq\pv_D)\cdot N^{(d_2^{\alpha}(\pvv)-s)(n+f(n))}+N^{(d_4^{\alpha}(\pv_-,\pv_1,\pv_+,\pv_+,0)-d_6(\pvv))\tau(n+f(n))}\cdot\sum_{\ell=1}^{q}N^{\tau^{\ell}(n+f(n))(d_6(\pvv)-s)}\right)\right).
\end{multline*}
Moreover, similarly
	\begin{multline*}
	\mathcal{H}_{N^{-\frac{r+f(r)}{\tau^{q}}}}^s(\pi\Gamma_B(f,\jj_k))\leq\sum_{n=r}^{\infty}\left(e^{C n\left(\varepsilon+O(1/q)\right)}\left(\mathbbm{1}(\pv_-\neq\pv_D)\cdot\sum_{k=1}^qN^{\frac{n+f(n)}{\tau^{k}}(d_1(\pvv)-s)}+\right.\right.\\
		\mathbbm{1}(\pv_1\neq\pv_D)\cdot N^{(d_2^{\alpha}(\pvv)-s)(n+f(n))}+\mathbbm{1}(\pv_2\neq\pv_D)\cdot N^{(d_3^{\alpha}(\pvv)-s)(\tau n+f(n))}+\\
	\left.\left.N^{(d_5^{\alpha}(\pvv,H)-d_6(\pvv))\tau(\tau n+f(n))}\cdot\sum_{\ell=1}^{q}N^{\tau^{\ell}(\tau n+f(n))(d_6(\pvv)-s)}\right)\right),
	\end{multline*}
and
\begin{multline*}
\mathcal{H}_{N^{-\frac{r+f(r)}{\tau^{q}}}}^s(\pi\Gamma_C(f,\jj_k))\leq\sum_{n=r}^{\infty}\left(e^{C n\left(\varepsilon+O(1/q)\right)}\left(\mathbbm{1}(\pv_-\neq\pv_D)\cdot\sum_{k=1}^qN^{\frac{n+f(n)}{\tau^{k}}(d_1(\pvv)-s)}+\right.\right.\\
\left.\left.\mathbbm{1}(\pv_1\neq\pv_D)\cdot N^{(d_2^{\alpha}(\pvv)-s)(n+f(n))}+N^{(d_4^{\alpha}(\pv_-,\pv_1,\pv_+,\pv_+,H)-d_6(\pvv))\tau(n+f(n))}\cdot\sum_{\ell=1}^{q}N^{\tau^{\ell}(n+f(n))(d_6(\pvv)-s)}\right)\right).
\end{multline*}

By \eqref{eq:usesok} and \eqref{eq:usesok2}, taking $s>\dim\pv_++C'(O(1/q)+\varepsilon)$ we get
	$$
	\mathcal{H}^s(\pi\Gamma_B(f,\jj_k)),\mathcal{H}^s(\pi\Gamma_C(f,\jj_k))<\infty,
	$$
	which implies the lemma since $\varepsilon,q$ were arbitrary.
\end{proof}

\section{Remaining proofs}

\begin{proof}[Proof of Theorem~\ref{thm:maincyl2} and Theorem~\ref{thm:mainball2}]
The lower bound in Theorem~\ref{thm:maincyl2} follows by Proposition~\ref{prop:lbcyl}, and the lower bound in Theorem~\ref{thm:mainball2} follows by Proposition~\ref{prop:lbball}.

The upper bounds in Theorem~\ref{thm:maincyl2} and in Theorem~\ref{thm:mainball2} follows by the combinations of lemmas, stated in Section~\ref{sec:ub}.
\end{proof}

\begin{proof}[Proof of Theorem~\ref{thm:maincyl1}]
It is easy to see that for every sequence $\xv_n\in\Lambda$ there exists a sequence in $\jj_n\in\Sigma$ such that $\pi(\jj_n)=\xv_n$ and $\pi(C_{f(n)}(\jj_n))=\mathcal{P}_{f(n)}(\xv_n)$. Thus, by \eqref{eq:conj}
$$
\left\{\yv\in\Lambda:T^n\yv\in\mathcal{P}_{f(n)}(\xv_n)\text{ i.o. }\right\}=\pi\left\{\ii\in\Sigma:\pi(\sigma^n\ii)\in\pi(C_{f(n)}(\jj_n))\text{ i.o. }\right\}.
$$
But there exists at most 9 distinct sequences $\jj_n^{(k)}$ for $k=1,\dots,9$ such that 
\begin{multline*}
\left\{\ii\in\Sigma:\sigma^n\ii\in C_{f(n)}(\jj_n)\text{ i.o. }\right\}\subseteq\left\{\ii\in\Sigma:\pi(\sigma^n\ii)\in\pi(C_{f(n)}(\jj_n))\text{ i.o. }\right\}\\
\subseteq\bigcup_{k=1}^9\left\{\ii\in\Sigma:\sigma^n\ii\in C_{f(n)}(\jj_n^{(k)})\text{ i.o. }\right\}.
\end{multline*}
Then the statement follows by Theorem~\ref{thm:maincyl2}.
\end{proof}

\begin{proof}[Proof of Theorem~\ref{thm:mainball1}]
Let $f(n):=\frac{-\log r(n)}{\log N}$. Thus, $\lim_{n\to\infty} f(n)/n=\alpha$.

Let $\mu$ be s $T$-inv. ergodic measure. Then there exists a $\nu$, $\sigma$-inv. ergodic measure such that $\mu=\nu\circ\pi^{-1}$. Hence, $\int\log T_{\lfloor Nx\rfloor+1}d\mu(x,y)=\int\log T_{a_1}d\nu(\ii)=:H$.
 
Moreover, for the sequence of random variables $\xv_n$ with identical distribution $\mu$, there exists a sequence of random variables $\jj_n\in\Sigma$ identically distributed with measure $\nu$ such that $\pi(\jj_n)=\xv_n$. By ergodicity and Lemma~\ref{cor:weneed}, 
$$
\lim_{n\to\infty}\frac{1}{f(n)(1-1/\tau)}\sum_{k=f(n)/\tau}^{f(n)}\log T_{a_k^{(n)}}=H
$$
for a.e. $\{\jj_n\}$, where $\jj_n=((a_1^{(n)},b_1^{(n)}),(a_2^{(n)},b_2^{(n)}),\dots)$.

For every $\ell=1,\dots,Q^2$, let $\jj_n^{(\ell)}$ be sequences in $\Sigma$ such that $\jj_n$ and $\jj_n^{(\ell)}$ differs at most at the positions $f(n)/\tau$ and $f(n)$, and $\jj_n^{(\ell)}\neq\jj_n^{(\ell')}$ if $\ell\neq\ell'$. Thus, for every $\ell$
$$
\lim_{n\to\infty}\frac{1}{f(n)(1-1/\tau)}\sum_{k=f(n)/\tau}^{f(n)}\log T_{a_k^{(n,\ell)}}=H
$$
almost surely, where $\jj_n^{(\ell)}=((a_1^{(n,\ell)},b_1^{(n,\ell)}),(a_2^{(n,\ell)},b_2^{(n,\ell)}),\dots)$. On the other hand,
$$
B_f(n)(\jj_n)\subseteq \pi^{-1}B(\xv_n,r(n))\subseteq \bigcup_{\ell=1}^{Q^2}B_f(n)(\jj_n^{(\ell)}).
$$
Therefore,
\begin{multline*}
\pi\left\{\ii\in\Sigma:\sigma^n\ii\in B_{f(n)}(\jj_n)\text{ i.o.}\right\}\subseteq\{\yv\in\Lambda:T^n\yv\in B(\xv_n,r(n))\text{ i.o.}\} \\
  \subseteq\bigcup_{\ell=1}^{Q^2}\pi\left\{\ii\in\Sigma:\sigma^n\ii\in B_{f(n)}(\jj_n^{(\ell)})\text{ i.o.}\right\}.
\end{multline*}
One can finish the proof by applying Theorem~\ref{thm:mainball2}.
 \end{proof}

\section{Examples}

In this section, we present some examples and facts on the dimension formula. If we assume that $R=\sharp S\geq 2$ and there exists $a_1\neq a_2\in S$ such that $T_{a_1}\neq T_{a_2}$ (that is, the Hausdorff and box counting dimension of the corresponding Bedford-McMullen carpet are not equal) then none of the dimension functions $d_i$ can be omitted in the dimension formula. Under the same setup, we present examples when $\alpha$ is relatively large, larger than $\tau-1$, and $\DJ_{\alpha}=d_4^{\alpha}=d_5^{\alpha}=d_6^{\alpha}$ for the maximizing measures. Also, we show that the result of Hill and Velani \cite[Theorem~1.2]{HillVelanitori} for two dimensional tori follows by Theorem~\ref{thm:mainball1}. For the graph of the dimension of a quite typical system, see Figure~\ref{fig:ex}.
\begin{figure}
	\centering
	\includegraphics[width=120mm]{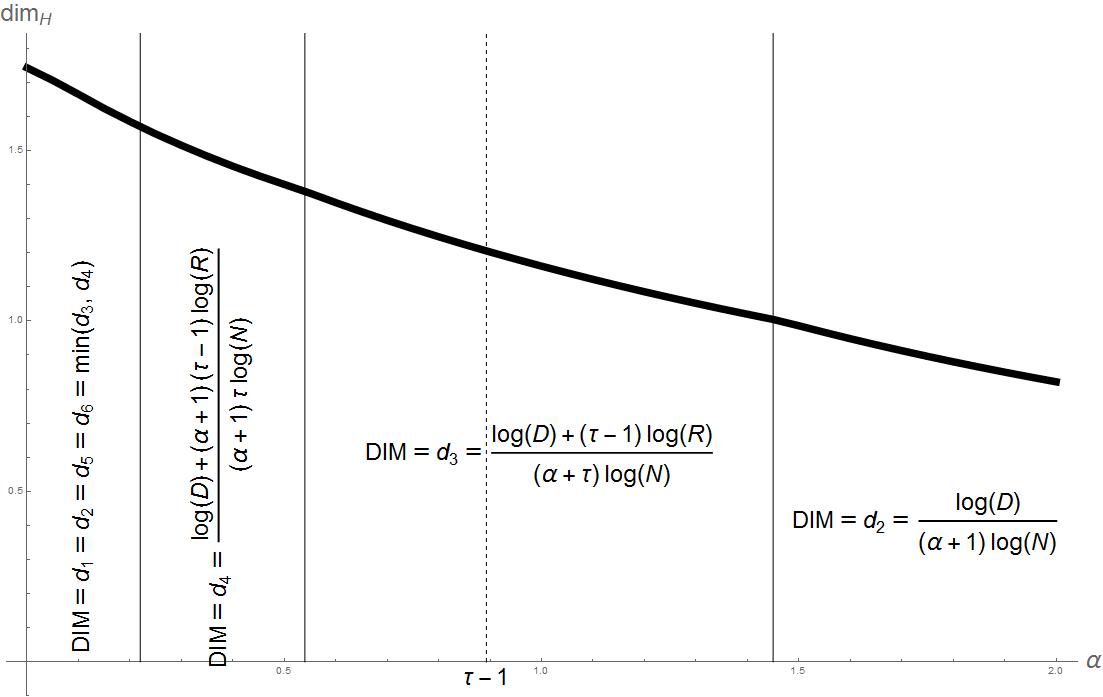}
	\caption{Graph of dimension of shrinking target set of parameters $N=3,M=8,T_1=5,T_2=2,T_3=8$ and $H=0.8$.}\label{fig:ex}
\end{figure}

\subsection{Phase transitions for small $\alpha$}. If $\alpha=0$ then $\pv_-=\pv_1=\pv_2=\pv_+=\pv_d$, where $\pv_-,\pv_1,\pv_2,\pv_+\in\Upsilon$ are the probability vectors, for which the maximum attains in \eqref{eq:thmball} (or \eqref{eq:thmcyl}) and $\pv_d$ is the probability measure defined in \eqref{eq:defmeas}. Clearly, $\pv_d$ is in the open set $I=\{\pv:\ h_r(\pv)<\log R\ \&\ h(\pv)<\log D\}$. 

\begin{claim}\label{c:1}
  For sufficiently small $\alpha>0$ and for every $(\pv_-,\pv_1,\pv_2,\pv_+)\in\Theta^{\alpha,H}_M$, $$\DJ_{\ALPHA}(\pv_-,\pv_1,\pv_2,\pv_+)=d_1(\pv_-)=\min\{d_2^{\alpha}(\pv_-,\pv_1),d_3^{\alpha}(\pv_-,\pv_1,\pv_2),d_4^{\alpha}(\pv_-,\pv_1,\pv_2,\pv_+),d_5^{\alpha}(\pv_-,\pv_1,\pv_2,\pv_+)\}.$$
\end{claim}

\begin{proof}
Let us argue by contradiction. Since $h_r(\pv_-)\leq h_r(\pv_d)$ by Lemma~\ref{lem:entropyrelations}, if $d_1>\min\{d_2^{\alpha},d_3^{\alpha},d_4^{\alpha},d_5^{\alpha}\}=\DJ_{\ALPHA}$ with $h_r(\pv_-)>h_r(\pv_D)$ then by decreasing $h_r(\pv_-)$, one could decrease the value $d_1$ (strictly) and (strictly) increase $\min\{d_2^{\alpha},d_3^{\alpha},d_4^{\alpha},d_5^{\alpha}\}$, which contradicts to that the maximum was attained at $(\pv_-,\pv_1,\pv_2,\pv_+)$. Thus, either $\pv_-=\pv_D$ or $\DJ_{\ALPHA}=d_1\leq\min\{d_2^{\alpha},d_3^{\alpha},d_4^{\alpha},d_5^{\alpha}\}$. But if $\pv_-=\pv_D$ with $\DJ_{\ALPHA}<d_1<\min\{d_2^{\alpha},d_3^{\alpha},d_4^{\alpha},d_5^{\alpha}\}$ then $\mathrm{DIM}(\alpha)=\min\{d_2^{\alpha},d_3^{\alpha},d_4^{\alpha},d_5^{\alpha}\}<d_1=\dim\pv_D<\dim\pv_d=\mathrm{DIM}(0)$, which contradicts to the fact that the function $\alpha\mapsto\mathrm{DIM}(\alpha)$ is continuous at $0$. Hence, $\DJ_{\ALPHA}=d_1\leq\min\{d_2^{\alpha},d_3^{\alpha},d_4^{\alpha},d_5^{\alpha}\}$ for sufficiently small $\alpha$.

On the other hand, if $\DJ_{\ALPHA}=d_1<\min\{d_2^{\alpha},d_3^{\alpha},d_4^{\alpha},d_5^{\alpha}\}$ then by increasing $h_r(\pv_-)$, $d_1$ strictly increases and $\min\{d_2^{\alpha},d_3^{\alpha},d_4^{\alpha},d_5^{\alpha}\}$ strictly decreases. Thus, either $\pv_-=\pv_d$ or $\DJ_{\ALPHA}=d_1=\min\{d_2^{\alpha},d_3^{\alpha},d_4^{\alpha},d_5^{\alpha}\}$. But by Lemma~\ref{lem:drop}, $\pv_-\neq\pv_d$ for $\alpha>0$.
\end{proof}

\begin{claim}\label{c:2}
  For sufficiently small $\alpha>0$ and for every $(\pv_-,\pv_1,\pv_2,\pv_+)\in\Theta^{\alpha,H}_M$, $$\DJ_{\ALPHA}(\pv_-,\pv_1,\pv_2,\pv_+)=d_1(\pv_-)=d_2^{\alpha}(\pv_-,\pv_1)=\min\{d_3^{\alpha}(\pv_-,\pv_1,\pv_2),d_4^{\alpha}(\pv_-,\pv_1,\pv_2,\pv_+),d_5^{\alpha}(\pv_-,\pv_1,\pv_2,\pv_+)\}.$$
\end{claim}

\begin{proof}
  Let us argue again by contradiction. By Claim~\ref{c:1} and Lemma~\ref{lem:entropyrelations}, if $\DJ_{\ALPHA}=d_1=\min\{d_3^{\alpha},d_4^{\alpha},d_5^{\alpha}\}<d_2^{\alpha}$ holds with $h_r(\pv_-)<h_r(\pv_1)$ then by decreasing $h_r(\pv_1)$ the value $\min\{d_3^{\alpha},d_4^{\alpha},d_5^{\alpha}\}$ strictly increases, $d_2^{\alpha}$ strictly decreases and $d_1$ does not change. But this contradicts to Claim~\ref{c:1}, since then there would be $(\pv_-,\pv_1',\pv_2,\pv_+)\in\Omega$ such that $\DJ_{\ALPHA}=d_1<\min\{d_2^{\alpha},d_3^{\alpha},d_4^{\alpha},d_5^{\alpha}\}$. Thus, either $h_r(\pv_-)=h_r(\pv_1)$ or $\DJ_{\ALPHA}=d_1=d_2^{\alpha}=\min\{d_3^{\alpha},d_4^{\alpha},d_5^{\alpha}\}$. But if $h_r(\pv_-)=h_r(\pv_1)$ then $d_2^{\alpha}<d_1$, which contradicts again to Claim~\ref{c:1}.
\end{proof}

\begin{claim}\label{c:3}
  For sufficiently small $\alpha>0$ and for every $(\pv_-,\pv_1,\pv_2,\pv_+)\in\Theta^{\alpha,H}_M$, $$\DJ_{\ALPHA}(\pv_-,\pv_1,\pv_2,\pv_+)=d_1(\pv_-)=d_2^{\alpha}(\pv_-,\pv_1)=d_5^{\alpha}(\pv_-,\pv_1,\pv_2,\pv_+)=\min\{d_3^{\alpha}(\pv_-,\pv_1,\pv_2),d_4^{\alpha}(\pv_-,\pv_1,\pv_2,\pv_+)\}.$$
\end{claim}

\begin{proof}
Similarly to the proof of Claim~\ref{c:2}, if $\DJ_{\ALPHA}=d_1=d_2^{\alpha}=\min\{d_3^{\alpha},d_4^{\alpha}\}<d_5^{\alpha}$ then by increasing $h_r(\pv_2)$, $d_5^{\alpha}$ decreases, $\min\{d_3^{\alpha},d_4^{\alpha}\}$ increases and $d_1,d_2^{\alpha}$ does not change. Thus, it contradicts to Claim~\ref{c:2}, since there would be a vector $(\pv_-,\pv_1,\pv_2',\pv_+)\in\Omega$ such that $\DJ_{\ALPHA}=d_1=d_2^{\alpha}<\min\{d_3^{\alpha},d_4^{\alpha},d_5^{\alpha}\}$.
\end{proof}

\begin{claim}\label{c:final}
  For sufficiently small $\alpha>0$ and for every $(\pv_-,\pv_1,\pv_2,\pv_+)\in\Theta^{\alpha,H}_M$,
   \begin{multline*}
  \DJ_{\ALPHA}(\pv_-,\pv_1,\pv_2,\pv_+)=d_1(\pv_-)=d_2^{\alpha}(\pv_-,\pv_1)=d_6(\pv_+)=\\
  d_5^{\alpha}(\pv_-,\pv_1,\pv_2,\pv_+)=\min\{d_3^{\alpha}(\pv_-,\pv_1,\pv_2),d_4^{\alpha}(\pv_-,\pv_1,\pv_2,\pv_+)\}.
  \end{multline*}
\end{claim}

\begin{proof}
Again, if $d_6>\DJ_{\ALPHA}=d_1=d_2^{\alpha}=d_5^{\alpha}=\min\{d_3^{\alpha},d_4^{\alpha}\}$ with $h_r(\pv_+)<\log R$ then by increasing $h_r(\pv_+)$ the values $\DJ_{\ALPHA}=d_1=d_2^{\alpha}$ does not change, $d_5^{\alpha}$ increases and $\min\{d_3^{\alpha},d_4^{\alpha}\}$ either increases or does not change. Thus, one may find $(\pv_-,\pv_1,\pv_2,\pv_+')\in\Omega(\alpha)$ such that $\min\{d_5^{\alpha},d_6\}>\DJ_{\ALPHA}=d_1=d_2^{\alpha}$, which contradicts to Claim~\ref{c:3}.
\end{proof}

\begin{claim}\label{c:final2}
 For every $\alpha>0$ sufficiently small and $(\pv_-,\pv_1,\pv_2,\pv_+)\in\Theta^{\alpha,H}_M$, if $H=\min_a\log T_a$ then $d_3^{\alpha}> d_4^{\alpha}$ and if $H=\max_a\log T_a$ then $d_4^{\alpha}> d_3^{\alpha}$.
\end{claim}

\begin{proof}
Simple algebraic manipulations show that
$$
d_4^{\alpha}\geq d_3^{\alpha}\text{ if and only if }\frac{H}{\tau\log N}+\frac{h_r(\pv_+)}{\log N}\geq d_3^{\alpha}.
$$
If $\alpha=0$ then $\pv_-=\pv_1=\pv_2=\pv_+=\pv_d$ and
$$
\frac{H}{\tau\log N}+\frac{h_r(\pv_+)}{\log N}\geq d_3^{\alpha}\text{ if and only if }h(\pv_d)-h_r(\pv_d)\geq H.
$$
Since $h(\pv_d)-h_r(\pv_d)>\min_a\log T_a$, by choosing $\alpha$ sufficiently small and $H=\min_a\log T_a$ we get that $d_4^{\alpha}> d_3^{\alpha}$. On the other hand, $h(\pv_d)-h_r(\pv_d)<\max_a\log T_a$. Thus, for $H=\max_a\log T_a$ choosing $\alpha$ sufficiently small, we get $d_4^{\alpha}<d_3^{\alpha}$.
\end{proof}

We have proven the following proposition.

\begin{prop}
   Let $M>N\geq2$, $R=\sharp S\geq 2$ and let us assume that there exists $a_1\neq a_2\in S$ such that $T_{a_1}\neq T_{a_2}$. Let $r(n)$ be such that $\lim_{n\to\infty}\frac{-\log r(n)}{n\log N}=\alpha>0$ sufficiently small. Then there exist $\mu_1,\mu_2$ ergodic, $T$-invariant measures with $\supp\mu=\Lambda$ and $H_i=\int\log T_{\lfloor Nx\rfloor+1}d\mu_i(x,y)$ such that
   \begin{multline*}
     \dim_H\{\yv\in\Lambda:T^n\yv\in B(\xv_n^{(1)},r(n))\text{ infinitely often}\}=\max_{\pv_-,\pv_1,\pv_2,\pv_+\in\Upsilon^4}\DJ_{\ALPHA}(\pv_-,\pv_1,\pv_2,\pv_+,H_1)=\\
     d_1(\pv_-)=d_2^{\alpha}(\pv_-,\pv_1)=d_3^{\alpha}(\pv_-,\pv_1,\pv_2)=d_5^{\alpha}(\pv_-,\pv_1,\pv_2,\pv_+,H_1)=d_6(\pv_+),
\end{multline*}
and
\begin{multline*}
     \dim_H\{\yv\in\Lambda:T^n\yv\in B(\xv_n^{(2)},r(n))\text{ infinitely often}\}=\max_{\pv_-,\pv_1,\pv_2,\pv_+\in\Upsilon^4}\DJ_{\ALPHA}(\pv_-,\pv_1,\pv_2,\pv_+,H_2)=\\
     d_1(\pv_-)=d_2^{\alpha}(\pv_-,\pv_1)=d_4^{\alpha}(\pv_-,\pv_1,\pv_2,\pv_+,H_2)=d_5^{\alpha}(\pv_-,\pv_1,\pv_2,\pv_+,H_2)=d_6(\pv_+),
\end{multline*}
 where $\xv_n^{(i)}$ are identically distributed sequences with measures $\mu_i$.
\end{prop}

\subsection{Phase transition for relatively large $\alpha$} Now, let us consider the case when the shrinking target sequence is a sequence of cylinders (in particular $H=0$). Let us choose $M>N\geq2$, $D>R\geq2$ and $T_a$ for $a\in S$ parameters such that
\begin{equation}
\frac{\log M}{\log N}<\frac{R\log D}{R\log R+\sum_{a\in S}\log T_a}.
\end{equation}
For example, $N=4,M=5, T_1=5, T_2=T_3=T_4=1$ satisfies this property.

Let $\alpha$ be such that it satisfies the following inequalities
\begin{equation}\label{eq:assonalpha}
\frac{\log D}{\log R}-1>\frac{R\log D}{R\log R+\sum_{a\in S}\log T_a}-1>\alpha>\tau-1
\end{equation}
Thus, $\pv_1$ does not play any role and by \eqref{eq:p-=pD}, $\pv_-=\pv_D$. In particular,
$$
\mathrm{DIM}(\alpha)=\max_{\pv\in\Upsilon_{\psi}}\left(\min\left\{\frac{\log D}{(1+\alpha)\log N},\frac{\log D+(1+\alpha)(\tau-1)h_r(\pv_+)}{\tau(1+\alpha)\log N},\dim\pv_+\right\}\right).
$$
It is easy to see by \eqref{eq:assonalpha},
$$
\frac{\log D}{(1+\alpha)\log N}>\frac{\log D+(1+\alpha)(\tau-1)h_r(\pv_+)}{\tau(1+\alpha)\log N}
$$
for every $\pv_+\in\Upsilon_{\psi}$, and
$$
\frac{\log D+(1+\alpha)(\tau-1)\log R}{\tau(1+\alpha)\log N}>\dim\pv_R,
$$
where $\pv_R$ is defined in \eqref{eq:defmeas}. Hence, by Lemma~\ref{lem:entropyrelations} and by taking any $h_r(\pv_+)<\log R$, $\dim\pv_+>\dim\pv_R$ and $d_4^{\alpha}$ decreases. Thus, the maximum cannot be achieved at $\pv_R$ and if $\pv_+$ is a prob. vector, where the maximum is achieved at, then
$$
\mathrm{DIM}(\alpha)=\frac{\log D+(1+\alpha)(\tau-1)h_r(\pv_+)}{\tau(1+\alpha)\log N}=\dim\pv_+.
$$
In particular, by Lemma~\ref{lem:6implies5}, $d_4^{\alpha}(\pv_D,\pv_D,\pv_+,\pv_+)=d_5^{\alpha}(\pv_D,\pv_D,\pv_+,\pv_+)=d_6(\pv_+)=\DJ_{\ALPHA}$.

\subsection{Dimension value for large $\alpha$} Next, we show that regardless of the choice of the holes (balls or cylinders) and regardless of the center points, the dimension is depending only on $D,N$ and $\alpha$ for large choice of $\alpha$.

\begin{claim}
	For every choice of $D,R,\{T_a\}_{a\in S}, N$ and $M$, there exists $A>0$ such that for every $\alpha>A$
	\begin{multline*}
	\mathrm{DIM}(\alpha)=\dim_H\{\yv\in\Lambda:T^n\yv\in\mathcal{P}_{f(n)}(x_n)\text{ infinitely often}\}=\\
	\dim_H\{\yv\in\Lambda:T^n\yv\in B(\xv_n,r(n))\text{ infinitely often}\}=\frac{\log D}{(1+\alpha)\log N},
	\end{multline*}
	where $x_n$ is an arbitrary sequence of points in $\Lambda$ and $f(n)$ and $r(n)$ are chosen such that $\lim_{n\to\infty}\frac{-\log r(n)}{n\log N}=\lim_{n\to\infty}\frac{f(n)}{n}=\alpha$.
\end{claim}

\begin{proof}
	We may assume that $A>\tau-1$. It is easy to see that
	$$
	d_4^{\alpha}(\pv_D,\pv_D,\pv_2,\pv_+)\geq\min\left\{\frac{\log D+(\tau-1)h_r(\pv_D)}{\tau\log N},\frac{(\tau-1)h_r(\pv_D)+(1-1/\tau)H}{\tau\log N}\right\}>0
	$$
	and
	$$
	d_5^{\alpha}(\pv_D,\pv_D,\pv_2,\pv_+)\geq\min\left\{\frac{\log D+(\tau-1)h(\pv_R)+\tau(\tau-1)h_r(\pv_D)}{\tau^2\log N},\frac{(\tau-1)h_r(\pv_D)+(1-1/\tau)H}{\tau\log N}\right\}>0
	$$
	for every $\pv_2,\pv_+$ and $\alpha\in[0,\infty)$. On the other hand,
	$$
	d_3^{\alpha}(\pv_D,\pv_D,\pv_2)\leq\frac{\log D+(\tau-1)\log R}{(\tau+\alpha)\log N}\searrow0\text{ as }\alpha\to\infty.
	$$
	and
	$$
	d_2^{\alpha}(\pv_D)=\frac{\log D}{(1+\alpha)\log N}\searrow0\text{ as }\alpha\to\infty
	$$
	Thus, there exists $a>0$ such that for every $\alpha>a$
	$$
	A_{\alpha}(\pvv)=\max\{i\in\{1,\dots,6\}:d_i(\pvv)=\DJ_{\ALPHA}(\pvv)\}\leq 3.
	$$
	But similarly to the beginning of the proof of Lemma~\ref{lem:max3b},
	$$
	\mathrm{DIM}(\alpha)=\min\left\{\frac{\log D}{(1+\alpha)\log N},\frac{\log D+(\tau-1)\log R}{(\tau+\alpha)\log N}\right\}.
	$$
	For $\alpha>\frac{\log D}{\log R}-1$, $\frac{\log D}{(1+\alpha)\log N}<\frac{\log D+(\tau-1)\log R}{(\tau+\alpha)\log N}$. By choosing $A=\max\{\tau-1,\frac{\log D}{\log R}-1,a\}$, the assertion follows.
\end{proof}

\subsection{Example of Hill and Velani} Finally, we show that the result of Hill and Velani \cite[Theorem~2]{HillVelanitori} is a consequence of Theorem~\ref{thm:mainball1} for the 2-dimensional tori case. 

\begin{claim}
	Let $T=(Nx\mod1,My\mod1)$ with $M>N\geq2$ integers as in \eqref{eq:transform}. Then
	$$
	\dim_H\left\{\xv\in[0,1]^2:T^n(\xv)\in B_{r(n)}(\yv_n)\text{ i.o.}\right\}=\min\left\{\frac{\tau+1}{\alpha+1},\frac{2\tau}{\tau+\alpha}\right\},
	$$
	where $\tau=\log M/\log N$, $\lim_{n\to\infty}\frac{-\log r(n)}{n\log N}=\alpha$ and $\yv_n$ is an arbitrary sequence in $[0,1]^2$.
\end{claim}

\begin{proof}
	In this case $D=NM$, $R=N$ and $T_a=M$ for every $a=1,\dots,N$. 
    Thus, by \eqref{eq:defmeas},
	$$
	\pv_R=\pv_D=\pv_d=\left(\frac{1}{NM}\right)_{(a,b)\in Q}.
	$$
	We may apply Theorem~\ref{thm:mainball1} in this case for arbitrary sequence, since $a\mapsto \log T_a$ is constant. Denote this measure by $\pv_U$. Hence,
	
	\begin{multline*}
	\dim_H\left\{\xv\in[0,1]^2:T^n(\xv)\in B_{r(n)}(\yv_n)\text{ i.o.}\right\}\geq\DJ_{\alpha}(\pv_U,\pv_U,\pv_U,\pv_U,\log M)=\\
	\min\left\{2,\frac{\tau+1}{1+\alpha},\frac{2\tau}{\tau+\alpha},\frac{2\alpha\tau-2\alpha+2\tau}{\tau(1+\alpha)},\frac{2\tau^2+2\alpha\tau-2\alpha}{\tau(\tau+\alpha)},2\right\}.
	\end{multline*}
	But simple algebraic manipulations show that
	$$
	2>\frac{2\tau}{\tau+\alpha}, \frac{2\tau^2+2\alpha\tau-2\alpha}{\tau(\tau+\alpha)}>\frac{2\tau}{\tau+\alpha}\text{ and }\frac{2\alpha\tau-2\alpha+2\tau}{\tau(1+\alpha)}>\frac{\tau+1}{1+\alpha}
	$$
	for any $\alpha>0$ and $\tau>1$. This proves the lower bound. For the upper bound, observe that for any probability vector $\pv$, $h_r(\pv)\leq\log N=h_r(\pv_U)$ and $h(\pv)\leq\log(NM)=h(\pv_U)$. Hence,
	$$
	\DJ_{\alpha}(\pv_-,\pv_1,\pv_2,\pv_+,\log M)\leq \DJ_{\alpha}(\pv_U,\pv_U,\pv_U,\pv_U,\log M)
	$$
	for every $(\pv_-,\pv_1,\pv_2,\pv_+)$, which proves the upper bound.
\end{proof}

\bibliographystyle{plain}
\bibliography{affin_2016}

\end{document}